\documentclass[11pt]{article}
\usepackage{a4wide,amsmath,amsbsy,amsfonts,amssymb,
stmaryrd,amsthm,mathrsfs,graphicx, amscd, tikz-cd, cancel, supertabular, tikz, qtree,latexsym,epsfig,color,lipsum,float,biblatex,mathtools,setspace,framed,breakurl}
\usepackage[top=1.2in,bottom=1.2in,left=1.21in,right=1.21in]{geometry}
\usepackage[all,cmtip]{xy}
\usepackage[all]{xy}  
\usepackage[titles]{tocloft}  
\usepackage[normalem]{ulem}
\usepackage[utf8]{inputenc}
\usepackage[english]{babel}
\usepackage[hidelinks]{hyperref}
\usepackage[hang,flushmargin]{footmisc}
\usepackage[T1]{fontenc} \usepackage{ae} \usepackage{aecompl}
        \hypersetup{
           breaklinks=true,   
           colorlinks=true,   
           pdfusetitle=true,  
        }

\usetikzlibrary{matrix,arrows,decorations.pathmorphing}

\hypersetup{colorlinks,linkcolor={blue},citecolor={green},urlcolor={darkgray}}  

\newtheorem{theorem}{Theorem}

\newtheorem{definition}{Definition}

\newtheorem{example}{Example}

\newtheorem{lemma}{Lemma}

\newtheorem{remark}{Remark}

\makeatletter

\newcommand\ackname{Acknowledgements}
\if@titlepage
  \newenvironment{acknowledgements}{%
      \titlepage
      \null\vfil
      \@beginparpenalty\@lowpenalty
      \begin{center}%
        \bfseries \ackname
        \@endparpenalty\@M
      \end{center}}%
     {\par\vfil\null\endtitlepage}
\else
  
\fi

\title{Singular Monopoles on Closed 3-Manifolds}
\author{Saman Habibi Esfahani}
\date{\today}

\begin{document}
\maketitle

\begin{abstract}
We prove the existence of non-trivial irreducible $SU(2)$-monopoles with Dirac singularities on any rational homology 3-sphere, equipped with any Riemannian metric, using a gluing construction.
\end{abstract}

\section{Introduction}

The theory of Yang-Mills connections and, in particular, instantons revolutionized the study of $4$-manifolds \cite{ MR710056, MR892034, MR882829}. The Bogomolny monopoles appear as the dimensional reduction of instantons to $3$-manifolds.

\begin{definition}[The Bogomolny Monopole]
Let $(M,g)$ be an oriented Riemannian 3-manifold. Let $G$ be a compact Lie group. Let $P \to M$ be a principal $G$-bundle and $\mathfrak{g}_P$ the associated adjoint bundle. Let $A$ be a connection on $P$ and $\varPhi$ a section of $\mathfrak{g}_P$. A pair $(A,\varPhi)$ is called a monopole if it satisfies the Bogomolny equation, which is
\begin{align}\label{moneq} 
    *F_A = d_{A} \varPhi,
\end{align}
where $*$ is the Hodge star operator on the $\mathfrak{g}_P$-valued differential forms on $M$, defined using the Riemannian metric $g$ and the orientation on $M$.

\end{definition}

The theory of monopoles on non-compact $3$-manifolds is very rich. Jaffe and Taubes proved the existence of non-trivial $SU(2)$-monopoles on $\mathbb{R}^3$, using a gluing construction \cite{MR614447}. The gluing constructions, originating from the works of Taubes, have been used to construct solutions to various differential equations \cite{MR658473,  MR866030,  MR1424428, MR3110581, MR3509964, MR3801425}. From the gluing construction of monopoles, one can read the dimension of the moduli spaces of monopoles on $\mathbb{R}^3$. This can also be proven using a variation of the Atiyah-Singer index theorem, called the Callias index theorem, which is an index theorem for Dirac operators on non-compact odd-dimensional manifolds  \cite{MR934202, MR3403964}. The moduli spaces of monopoles on $\mathbb{R}^3$ are ALF hyperkähler manifolds, which have been extensively studied, originating from the works of Atiyah and Hitchin \cite{MR934202}. Furthermore, there exists an explicit parametrization of the moduli spaces of monopoles on $\mathbb{R}^3$ in terms of rational maps, due to Donaldson \cite{MR769355, MR804459}. 

Floer studied monopoles on asymptotically Euclidean 3-manifolds \cite{MR866030} and, more recently, Oliveira studied monopoles on asymptotically conical 3-manifolds and stated that there exists a $(4k-1)$-dimensional family of non-trivial irreducible smooth $SU(2)$-monopoles on any asymptotically conical $3$-manifold $(M,g)$ with $b^2(M) = 0$ \cite{MR3509964}. It is proven by Kottke that the expected dimension of the moduli space of monopoles on an asymptotically conical 3-manifold, whose ends are asymptotic to a cone on $\Sigma$, is $4k + \frac{1}{2}b^1(\Sigma)- b^0(\Sigma)$ \cite{MR3403964}. 

The theory of monopoles on compact $3$-manifolds is quite different from the ones on non-compact manifolds. When the structure group $G$ is compact, every smooth monopole on a closed oriented Riemannian 3-manifold satisfies a stronger condition,
\begin{align*}
    *F_A = d_A \varPhi = 0,
\end{align*}
and therefore, $A$ is a flat connection and $\varPhi$ is a covariantly constant section. These monopoles are sometimes referred to as trivial monopoles.

There is another class of monopoles on compact $3$-manifolds which have non-flat curvature. These monopoles are smooth on the complement of finitely many points with prescribed Dirac singularities at these points.

\begin{definition}[Dirac Singularity]
Let $P \to M \setminus \{p_1, \hdots, p_n\}$ be a principal $SU(2)$-bundle. A monopole $(A,\varPhi)$ on this bundle is called a monopole with Dirac singularities if close to the singular point $p_i$, we have
\begin{align} \label{sing}
    |\varPhi| = \frac{k}{2r} + m + O(r),
\end{align}
where the norm is defined with respect to the adjoint-invariant inner product on the adjoint bundle $\mathfrak{g}_P$, $r$ is the geodesic distance from $p_i$, $k \in \mathbb{N}$ is a positive integer, called the charge of the monopole at the singular point, and $m$ is a constant, called the mass of the monopole at the singular point. The pair $(A, \varPhi)$ is called a monopole with a Dirac singularity.
\end{definition}

Pauly studied the deformation of these singular monopoles with the structure group $SU(2)$ \cite{MR1624279}, and using the Atiyah-Singer index theorem and exploiting a theorem of Kronheimer \cite{kronheimer1985monopoles} --- which states that close to the points with Dirac singularities, monopoles up to gauge, can be understood as smooth $S^1$-invariant instantons on a 4-dimensional space --- he proved that the expected dimension of the moduli space of singular monopoles with charge $k \in \mathbb{N}$ on a compact Riemannian $3$-manifold $(M,g)$ is equal to $4k$. However, this argument does not imply that the moduli spaces are non-empty.  

In this article, we prove the existence of $SU(2)$-monopoles with Dirac singularities on rational homology 3-spheres equipped with any Riemannian metric.  The proof is based on a gluing construction. Furthermore, this construction gives a geometric interpretation to Pauly's dimensional formula for the moduli spaces of singular monopoles on rational homology 3-spheres. This gluing construction is also motivated by the study of monopoles in higher dimensions \cite{MR4495257}.

\begin{theorem} \label{theorem1}
Let $M$ be an oriented rational homology 3-sphere equipped with a Riemannian metric $g$. Let $S_p = \{p_1, \hdots, p_n \}$ and $S_q = \{q_1, \hdots, q_k \}$ be two sets of points
in $M$ where all $n+k$ points are disjoint. Let $k_1, \hdots, k_n$ be $n$ negative integers, where $k + \sum_{i=1}^n k_i = 0$. Then there exists an irreducible $SU(2)$-monopole $(A,\varPhi)$ with Dirac singularities with charge $|k_i|$ at $p_i$ for all $i \in  \{1,\hdots,n\}$ on a principal $SU(2)$-bundle $P \to M \setminus S_p$ such that 
\begin{align*}
    (A,\varPhi) = (A_0, \varPhi_0) + (a,\varphi),
\end{align*}
where $(A_0, \varPhi_0)$ is equal to a scaled BPS-monopole on a small neighbourhood $B_{ \epsilon_j}(q_j)$ of each point $q_j$ for $j \in \{1, \hdots, k\}$ and is equal to the lift of a $U(1)$-Dirac monopole with charge $k_i$ at $p_i$ for $i \in  \{1, \hdots, n\}$ on $M \setminus \cup_{j=1}^k B_{3 \epsilon_j}(q_j)$. Moreover, the pair $(a,\varphi) \in W^{1,2}_{\alpha_1,\alpha_2}$ for suitable values of $\alpha_1$ and $\alpha_2$, where $W^{1,2}_{\alpha_1,\alpha_2}$ is a weighted Sobolev space, defined in Definition \ref{weighted}, such that $\|(a,\varphi)\|_{W^{1,2}_{\alpha_1,\alpha_2}} \to 0$ as the masses at the Dirac singularities go to infinity.
\end{theorem}

The proof of Theorem \ref{theorem1} is based on a gluing construction.

\begin{itemize}
\item The first step is to produce an Abelian Dirac monopole on $(M,g)$ with some singular points $p_i$ with negative charges and some singular points $q_j$ with charge $+1$ such that the total charge of the monopole is zero. We construct the Dirac monopole in Section \ref{dirmon}.  

\item The second step is to smooth out the singularities with charge $+1$ and construct an approximate monopole. The smoothing process is carried over by gluing model SU(2)-monopoles --- called the scaled BPS-monopoles --- to  the singular points with charge $+1$ and leaving out the rest of the singular points not smoothed-out. This has been done in Section \ref{appsolsection}.

\item The third step is the deformation. The resulting configuration from the previous step is an approximate monopole and it does not necessarily satisfy the Bogomolny equation. However in a suitable norm, it is close to a solution and can be deformed into a genuine monopole. We solve the linearized Bogomolny equation in Section \ref{solveequation}, and then consider the quadratic terms, and solve the full Bogomolny equation in Section \ref{quadraticterm}.

\end{itemize}

\begin{remark}
The gluing construction for $SU(2)$-monopoles with Dirac singularities at the points $p_i$ with charges $|k_i|$ depends on $4k$-parameters, where $k = - \sum_{i=1}^n k_i$. This is equal to the expected dimension of the moduli space of singular $SU(2)$-monopoles with charge $k$, as computed by Pauly. $3k$ of this number is accounted by the position of the highly concentrated BPS-monopoles, $k-1$ of this number by choices of the framings at these points. There are $k$ points which we can fix the frames there; however, $1$ parameter vanishes after taking the action of the gauge group into account. The remaining 1 degree of freedom comes from changing the average mass of the monopole. 
\end{remark}

\noindent 
\textbf{Acknowledgments.}
This article is part of the PhD thesis of the author at Stony Brook University. I am grateful to my advisor Simon Donaldson for his guidance, encouragement, and support. Moreover, I would like to thank Aliakbar Daemi, Lorenzo Foscolo, Jason Lotay, Gonçalo Oliveira, and Yao Xiao for helpful conversations. This work was completed while the author was in residence at the Simons Laufer Mathematical Sciences Institute (previously known as MSRI) Berkeley, California, during the Fall 2022 semester, supported by NSF Grant DMS-1928930.

\section{Dirac Monopoles on Closed 3-Manifolds}\label{dirmon}

In this section, we study and later construct Dirac monopoles on rational homology 3-spheres.

\subsection{Local Model of Dirac Monopoles}

In this subsection, we study Dirac monopoles close to the points with Dirac singularities.

Let $(A_D, \varPhi_D)$ be a $U(1)$-monopole with a Dirac singularity at $p \in M$ with signed charge $k \in \mathbb{Z} \setminus \{ 0 \}$, defined on a small neighbourhood of $p$ in $M$. A Dirac monopole is a monopole with Dirac singularities on a bundle with structure group $U(1)$. Close to a singular point $p$, the Higgs field $\varPhi_D$ has the following form,
\begin{align} \label{Dirac}
    \varPhi_D &= -\frac{k}{2r} + m + O(r),
\end{align}
where $r$ denotes the geodesic distance from $p$ and $k$ is the signed charge at $p$. Note that unlike \ref{sing}, the left-hand-side is the section itself and not its norm.

The Bianchi identity shows that the curvature $2$-form of a $U(1)$-connection is closed. Furthermore, from the Chern-Weil theory we know that the $2$-form 
\begin{align*}
\frac{F_{A_D}}{2 \pi} = \frac{*d\varPhi_D}{2 \pi},
\end{align*}
presents $c_1(L)$, the first Chern class of a line bundle $L$ where the monopole is defined on. Restricting the bundle to a sufficiently small punctured neighbourhood of a singular point $p$ with charge $k$, the line bundle $L_{|_{B_{\epsilon}(p) \setminus \{ p \}}} \to B_{\epsilon}(p) \setminus \{ p \}$ is isomorphic to $H_p^{k}$, where $H_p$ is the Hopf line bundle centered at $p$,  with the first Chern number $c_1$
\begin{align*}
    c_1 = \lim_{\epsilon \to 0} \frac{1}{2 \pi} \int_{\partial B_{\epsilon}(p)} *d\varPhi_D = \lim_{\epsilon \to 0} \frac{1}{2 \pi} \int_{\partial B_{\epsilon}(p)} \left( \frac{k}{2\epsilon^2} + O(1) \right) vol_{\partial B_{\epsilon}(p)} = k.
\end{align*}
The model connection $A$ of a Dirac monopole on $\mathbb{R}^3$ close to a singular point $0 \in \mathbb{R}^3$ with charge $k$ is an $SO(3)$-invariant connection defined on the line bundle $H_0^{k} \to \mathbb{R}^3 \setminus \{ 0 \}$. Let $S^2_0(1)$ be the unit $2$-sphere centred at the origin in $\mathbb{R}^3$. We can cover $S^2_0(1)$ by $U^+$ and $U^-$, where
$U^+ = S^2_0(1) \setminus \{ (0,0,-1)\}$ and $U^- = S^2_0(1) \setminus \{ (0,0,1)\} $. In spherical coordinates $(\rho, \theta, \varphi)$, the connection $A$ on $U^+$ and $U^-$ is given by the following 1-forms,
\begin{align*}
    A_{|_{ U^-}} = 
   k \frac{(1 - \cos (\varphi))}{2} d\theta,\quad \quad
    A_{|_{ U^+}} =
   k \frac{(-1 - \cos (\varphi))}{2} d\theta,
\end{align*}
with the transition function $e^{ik\theta}$. Note that on $ U^- \cap  U^+$  
\begin{align*}
A_{|_{ U^-}} - A_{|_{ U^+}} = k d \theta.    
\end{align*}
We extend the connection radially to $H_0^k \to \mathbb{R}^3 \setminus \{0\}$ to get $A$.

Using geodesic normal coordinates, we can define a diffeomorphism
\begin{align*}
\eta:  B_{\epsilon}(0) \subset  \mathbb{R}^3 \to B_{\epsilon}(p) \subset M,    
\end{align*}
between a small neighbourhood of the origin in $\mathbb{R}^3$ and a small neighbourhood of a point $p \in M$. Furthermore, by choosing a bundle isomorphism, covering $\eta$, we can identify the bundles above these open neighbourhoods and pull back the connection $A_D$ to a punctured neighbourhood of the origin in $\mathbb{R}^3$. 

\begin{lemma}
The connection of the Dirac monopole with charge $k$, denoted by $A_D$, close to a singular point $p \in M$, up to a gauge transformation, can be written as the following,
\begin{align} \label{a}
    \eta^* A_D = A + a, \quad \text{ with } \quad |a|  = O(r),
\end{align}
where the gauge transformation --- which is just addition by an exact 1-form --- corresponds to tensoring $H^k_p$ by a flat line bundle. 
\end{lemma}

\begin{proof}
The pair $(\eta^*A_D, \eta^* \varPhi_D)$ is not necessarily a monopole with respect to the Euclidean metric on $B_{\epsilon}(0) \subset \mathbb{R}^3$; however, it is a monopole with respect to the pull-back metric $\eta^*g$, and therefore, $\eta^*\varPhi_D = -\frac{k}{2r} + m + O(r)$, where $r$ is the geodesic distance from the origin with respect to $\eta^*g$.  

$A$ is the connection of a monopole with a Higgs field $\varPhi = -\frac{k}{2r_0} + m_0 + O(r_0)$, where $r_0$ is the distance to the origin with respect to the Euclidean metric, and therefore,
\begin{align*}
    |*_0 d(\eta^* A_D - A)|_{g_0} = |d (\eta^* \varPhi_D - \varPhi)|_{g_0} = |d (\frac{k}{2r}-\frac{k}{2r_0})|_{g_0} + O(1).
\end{align*}
Moreover,
\begin{align*}
|r - r_0| = \max \{ R_{i,j,k,l} \} O(r_0^3) + O(r_0^4),
\end{align*}
where $R_{i,j,k,l}$ is the Riemann curvature tensor of $\eta^* g$, and therefore,
\begin{align*}
    |d(\eta^* A_D - A)|_{\eta^*g} = O(1),
\end{align*}
which shows in a suitable gauge, 
$|a|_{\eta^* g} = O(r).$
\end{proof}

\subsection{Construction of Dirac Monopoles}\label{ConsDiracMon}

In this section, we construct a Dirac monopole $(A_D, \varPhi_D)$ with prescribed charges and singularities on a rational homology 3-sphere $(M,g)$. One can construct Dirac monopoles on any closed Riemannian 3-manifold; however, here we only consider the case where $H_2(M,\mathbb{Q}) = 0$.

\begin{lemma}
Let $(M,g)$ be an oriented rational homology 3-sphere equipped with a Riemannian metric $g$. Let $p_1, \hdots ,p_n$ be $n$ distinct points in $M$ with non-zero integer-valued charges $k_1, \hdots ,k_n$, respectively, where $\sum_{i=1}^n k_i = 0$. Then there exists a monopole $(A_D, \varPhi_D)$ with Dirac singularities with charge $k_i$ at $p_i$ on a principal $U(1)$-bundle $P \to M \setminus\{p_1,  \hdots , p_n\}$. This monopole, up to gauge transformations and adding a constant to the Higgs field, is unique.
\end{lemma}

\begin{proof}
From the monopole equation it can be seen that on the complement of the singular points we have $\Delta \varPhi_D = 0$,  and therefore, $ \varPhi_D$ is a harmonic section of the adjoint bundle on $M \setminus \{ p_1, \hdots, p_n \}$. A Dirac monopole singular at the points $p_1, \hdots, p_n$ with corresponding signed charges $k_1, \hdots, k_n$ is a solution to the equation
\begin{align}\label{Delta}
    \Delta \varPhi_D = \sum_{i=1} ^n k_i \delta_{p_i},
\end{align}
on a compact Riemannian $3$-manifold $(M,g)$, in the sense of currents, where $\delta_{p_i}$ is the Dirac delta function centered at the point $p_i$.

The Dirac delta function can also be understood as a map $\delta_{p_i} : C^{\infty}(M) \to \mathbb{R}$, defined by $\delta_{p_i}(f) = f(p_i)$, or as a $3$-dimensional cohomology satisfying the equation $\int_M f \delta_{p_i} = f(p_i)$ for any smooth function $f$. By a slight abuse of notation, we would denote any of them by $\delta_{p_i}$.

The equation \ref{Delta} has a solution if and only if 
\begin{align}
    \sum_{i=1} ^n k_i = 0.
\end{align}
This can be seen as a generalization of the well-known fact that on a closed oriented Riemannian manifold $(M,g)$ the equation $\Delta f = h$, for a smooth function $h$,  has a solution if and only if $\int_M h vol_g= 0$ to the case where $h$ is not a smooth function, but a distribution. 

Let $\delta_{p_i}$ be a 3-form representative of the Poincar\'e dual of the 0-cycle $\{p_i\}$. The equation $\Delta \widetilde{\varPhi}_D = \sum_i k_i \delta_{p_i}$ has a solution if and only if $\delta = \sum_i k_i \delta_{p_i}$ is an element of the orthogonal complement of harmonic 3-forms $\mathcal{H}^3$. From the Hodge decomposition theorem we have,
\begin{align*}
    \Omega^3(M) = d \Omega^2(M) \oplus \mathcal{H}^3.
\end{align*}
On a closed oriented Riemannian 3-manifold, $\mathcal{H}^3$ is 1-dimensional, generated by the volume form $vol_g$ of the Riemannian metric $g$ --- note that the volume form of $g$ is parallel and harmonic. On the other hand 
\begin{align*}
    \langle \delta, vol_g \rangle =  \sum_{i=1}^n k_i \int_M \delta_{p_i} \wedge * vol_g = \sum_{i=1}^n k_i = 0,
\end{align*}
and therefore, the equation $\Delta \widetilde{\varPhi}_D = \sum_i k_i \delta_{p_i}$ has a solution. We can define the Higgs field of the Dirac monopole by $\varPhi_D := * \widetilde{\varPhi}_D$. 

Furthermore, the solution to this equation is unique up to addition by a constant. For any two solutions $\varPhi_D$ and $\varPhi_D'$ of the equation \ref{Delta}, we have 
\begin{align*}
    \Delta(\varPhi_D - \varPhi_D') = \Delta \varPhi_D - \Delta \varPhi_D' =  \sum_{i=1} ^n k_i \delta_{p_i} - \sum_{i=1} ^n k_i \delta_{p_i} = 0,
\end{align*}
and therefore, $\varPhi_D - \varPhi_D'$ is a harmonic function on the closed manifold $M$; hence, it is constant.

Also, note that the assumption on the total charge being zero is necessary. For any $\varPhi$ with  $\Delta \varPhi= \sum_{i=1} ^n k_i \delta_{p_i}$, we have
\begin{align*}
    \int_{M \setminus \{ p_1, \hdots, p_n\} } \Delta \varPhi vol_g &= \lim_{\epsilon \to 0} \int_{M \setminus \cup_{i=1}^n B_{ \epsilon}(p_i) } \Delta \varPhi vol_g = \lim_{\epsilon \to 0} \int_{M \setminus \cup_{i=1}^n B_{ \epsilon}(p_i) }  d^* d \varPhi vol_g \\
    &= \lim_{\epsilon \to 0} \int_{M \setminus \cup_{i=1}^n B_{ \epsilon}(p_i) }  d * d \varPhi = \lim_{\epsilon \to 0} \int_{ \cup_{i=1}^n \partial B_{ \epsilon}(p_i) } * d \varPhi,
\end{align*}
which is zero since $\Delta \varPhi = 0$ on $M \setminus \{ p_1,\hdots, p_n \}$. On the other hand $\varPhi = -\frac{k_i}{2r_i} + m_i + O(r_i)$ close to the point $p_i$, and therefore,
\begin{align*}
*d\varPhi = \frac{k_i}{2r_i^2} \iota_{\partial r_i} vol_g + O(1), 
\end{align*}
hence,
\begin{align*}
    \lim_{\epsilon \to 0} \int_{ \cup_{i=1}^n \partial B_{\epsilon}(p_i) } * d \varPhi &= 
    \lim_{\epsilon \to 0} \sum_{i=1}^n \int_{\partial B_{\epsilon}(p_i) } (\frac{k_i}{2r_i^2} \iota_{\partial r_i} vol_g + O(1))
    \\& =     \lim_{\epsilon \to 0} \sum_{i=1} ^n\int_{\partial B_{\epsilon}(p_i) } \frac{k_i}{2r_i^2}  vol_{\partial B_{\epsilon}(p_i)}
    = 2\pi \sum_{i=1} ^n k_i,
\end{align*}
and therefore, $k := \sum_{i=1} ^n k_i = 0$. This is in contrast  with the non-compact case, where some of the charges can run into infinity. 

Now suppose there exists a connection $A_D$ on a line bundle $\pi: L \to M \setminus \{ p_1, \hdots, p_n \}$ such that $F_{A_D} = *d\varPhi_D$. For any other connection 1-form $A'_D$ satisfying this equation, since ${H_{dR}^1(M \setminus \{ p_1, \hdots, p_n \}) = 0}$ and $d(A_D - A'_D) = 0$, we have $(A_D - A_D') = \pi^*(df)$ for a smooth function $f \in C^{\infty}(M \setminus \{ p_1, \hdots, p_n \})$. This implies that if such a connection exists, it is determined by the Higgs field up to addition by an exact form, which corresponds to tensoring the line bundle which the connection $A_D$ is defined on by a flat line bundle. 

Now we focus on the existence problem for such a connection. Let $F_D$ be a 2-form defined by $ F_D := * d \varPhi_D$. We should determine when we can realize this $2$-form as the curvature $2$-form of a connection $A_D$ on a principal $U(1)$-bundle on $M \setminus \{ p_1, \hdots, p_n \}$. This would be the case if the $2$-form $F_D$ has integer periods in $H^2(M,\mathbb{R})$. Recall the following lemma from the Chern-Weil theory.

\begin{lemma} \label{4}
For any integral closed $2$-form $F$ on a manifold $X$, there is a line bundle $L \to X$, unique up to isomorphism, with a connection $1$-form $A$ with curvature $2$-form $F$.
\end{lemma}

In our case, note that $\frac{1}{2 \pi} F_D$ is a closed $2$-form on $M \setminus \{ p_1, \hdots, p_n \}$, and therefore, we can consider the cohomology class $[\frac{1}{2 \pi}F_D] \in H^2(M \setminus \{ p_1, \hdots, p_n \}, \mathbb{R})$. We need to show that the cohomology class $[\frac{1}{2 \pi}F_D]$ vanishes in $H^2(M \setminus \{ p_1, \hdots, p_n \}, \mathbb{R}/\mathbb{Z})$. However, since $M$ is a rational homology 3-sphere,  $H_2(M, \mathbb{Z}) = 0$, and therefore, $H_2(M \setminus \{ p_1, \hdots, p_n \}, \mathbb{Z})$ is generated by 2-spheres $\partial B_{ \epsilon}(p_i)$. We have
\begin{align*}
\frac{1}{2 \pi} \int_{\partial B_{\epsilon}(p_i)}F_D = \frac{1}{2 \pi} \int_{\partial B_{\epsilon}(p_i)} * d\varPhi_D = k_i \equiv 0 \text{ modulo } \mathbb{Z},   
\end{align*} 
and therefore, by Lemma \ref{4}, there is a principal $U(1)$-bundle $L \to M \setminus \{ p_1, \hdots, p_n \}$ and a connection 1-form $A_D$ on $L$, where $F_D$ is the curvature of $A_D$ and $\varPhi_D$ is a section of the adjoint bundle. 
\end{proof}

\subsubsection{Mass of Monopoles}

On non-compact manifolds with ends of suitable types, for instance asymptotically conical ones, the mass of a monopole $(A, \varPhi)$ with a sufficiently fast decaying curvature can be defined at the ends of the manifold as the limit of the Higgs field $\lim_{r \to \infty} |\varPhi|$, where $r$ denotes the geodesic distance from a fixed point $x_0 \in M$. Similarly, we defined the mass of the monopole at a singular point to be the constant $m$ appearing in Formula \ref{sing}. 

Let the vector $\overrightarrow{m} = (m_1, \hdots, m_n) \in \mathbb{R}^n$ denote the masses of the monopole at the singular points $p_1, \hdots, p_n$ on a closed manifold $(M,g)$. For any Dirac monopole $(A_D,\varPhi_D)$ with vector mass $\overrightarrow{m}$, the monopole $(A_D, \varPhi_D + c)$ has the vector mass $\overrightarrow{m} + c = (m_1 + c, \hdots, m_n + c)$. In the asymptotically conical case, one can find a monopole with fixed charges and fixed masses at the ends of the manifold \cite{MR3509964}; however, recall that in the case of compact manifolds, the Dirac monopole, up to addition by a constant, is given by fixing the charges at the singular points, and we do not have much freedom in the choices of the masses of the monopole at the singular points. We can only add a constant to the mass vector, and therefore, we can only fix the position of the singular points, the charges at $(n-1)$ of them and the mass in one of the singular points. 

We define the average mass by $\overline{m} = \frac{m_1 + \hdots + m_n}{n}$. The relative mass $m_i'$ at each singular point $p_i$ is defined by $m_i = \overline{m} + m_i'$. The relative mass at each point $p_i$ is a function of the charges and the locations of the singular points, and as one moves these points around, these relative masses change. Here, for our gluing construction to work, we would make the average mass sufficiently large by adding a constant.

In the study of the moduli spaces of monopoles on $\mathbb{R}^3$, one can assume a normalizing condition, to have mass 1 at infinity, since there is a natural identification between the moduli spaces of monopoles with different masses. However, this is not true for the moduli spaces of monopoles on other 3-manifolds, and there is no natural identification between the moduli spaces of monopoles with different masses.

\subsubsection{Lifting the Dirac Monopole}
For carrying on our gluing construction, we should lift the Dirac monopole we constructed to an $SU(2)$-bundle, so we can glue the scaled $SU(2)$ BPS-monopoles to this lifted background Dirac monopole. Consider the rank 2 vector bundle $L \oplus L^{-1} \to M \setminus \{ p_1, \hdots, p_n\}$. Let $\nabla_{A_D}$ be the covariant derivative of $A_D$ on $L$. This induces a covariant derivative on $L \oplus L^{-1}$, namely $\nabla_{A_D} \oplus (-\nabla_{A_D})$.

Moreover, close to a singular point $p_i$, the adjoint bundle of the corresponding $SU(2)$-bundle can be decomposed as $\underline{\mathbb{R}} \oplus H_{p_i}^{k_i}$, with the induced Higgs field $\varPhi_D \oplus 0$. Close to a singular point $p_i$ with charge $k_i$, the rank 2 bundle is isomorphic to $H_{p_i}^{k_i} \oplus H_{p_i}^{-k_i}$. 

We can fix a basis for $\mathfrak{su}(2)$,
\begin{align*}
    \sigma_1 = 
    \begin{pmatrix}
    0 & -i \\
    -i & 0
    \end{pmatrix}, \quad
    \sigma_2 = 
    \begin{pmatrix}
    0 & -1\\
    1 & 0
    \end{pmatrix}, \quad
    \sigma_3 =  
    \begin{pmatrix}
    -i & 0\\
    0 & i
    \end{pmatrix}.
\end{align*}
Suppose $(A_D, \varPhi_D)$ is a $U(1)$-Dirac monopole, defined on a $U(1)$-bundle $P_{U(1)}$ with associated line bundle $L$. The induced $SU(2)$-monopole is $(A_D \sigma_3, \varPhi_D \sigma_3)$, which for simplicity and by an abuse of notation we still denote this monopole by $(A_D, \varPhi_D)$.

\section{Approximate Solutions}\label{appsolsection}

In this section, we construct an approximate $SU(2)$-monopole $(A_0, \varPhi_0)$ on an $SU(2)$-bundle $P \to M$, with prescribed Dirac singularities at some isolated points $p_i$ for $i \in \{1,  \hdots , n\}$. The pair $(A_0, \varPhi_0)$ would not be a genuine monopole, but an approximate one.

We take the lifted Dirac monopoles we constricted in Section \ref{ConsDiracMon} as the background monopole. The monopoles we are gluing to these Dirac monopoles are defined by scalings of the BPS-monopole on $\mathbb{R}^3$. 

\subsection[BPS-Monopoles on $\mathbb{R}^3$]{BPS-Monopoles on $\pmb{\mathbb{R}^3}$} \label{BPSmon}

In this section, we introduce the BPS-monopoles on $\mathbb{R}^3$ and recall a basic lemma about their asymptotic behaviour.

Let $P \to \mathbb{R}^3$ be a principal $SU(2)$-bundle. Let $A$ be a connection on $P$ and $\varPhi$ a section of the adjoint bundle. For pairs $(A,\varPhi)$ with suitable asymptotic decay, we can define the Yang-Mills-Higgs action functional,
\begin{align*}
    \mathcal{YMH}(A, \varPhi):= \frac{1}{2}\int_{\mathbb{R}^3} (|F_A|^2+|d_A\varPhi|^2)dx dy dz,
\end{align*}
where the norms are defined with respect to the adjoint-invariant inner product on the adjoint bundle. 

The critical points of Yang-Mills-Higgs action functional are the solutions to the following equations,
\begin{align*}
    d_A^* F_A &= - * [\varPhi, d_A\varPhi],\\
    d_A^* d_A \varPhi &= 0.
\end{align*}
Monopoles satisfy these equations, in fact, they are the minimizers of this action functional. Monopoles with finite Yang-Mills-Higgs energy satisfy the following decay conditions,
\begin{align}\label{decayC}
    |F_A| = |d_A \varPhi| = O(r^{-2}) , \quad \quad |\varPhi| \to m, \quad \text{as } \quad r \to \infty,
\end{align}
where the constant $m = \lim_{|x| \to \infty} |\varPhi(x)|$ is the mass of the monopole at infinity. 

By a scaling, one can change the mass of a given monopole on $\mathbb{R}^3$. If $(A,\varPhi)$ is a monopole on $\mathbb{R}^3$ with mass $m$ at infinity, then
\begin{align*}
    (A^{\lambda},\varPhi^{\lambda})(x) := ( A, \lambda \varPhi) (\lambda x),
\end{align*}
is also a monopole on $\mathbb{R}^3$ with mass $\lambda m$. This shows there are natural identifications between the moduli spaces of monopoles with different positive masses, and therefore, one can assume a normalizing condition, and let $m = 1$. 

In the case $G = SU(2)$, to any pair $(A,\varPhi)$ on an $SU(2)$-bundle on $\mathbb{R}^3$, which is not necessarily a monopole, with the decay conditions \ref{decayC} and mass $m$, one can assign an integer charge, defined by 
\begin{align*}
    k := \lim_{R \to \infty} \frac{1}{4 \pi m} \int_{S_R(0)} \langle \varPhi, F_A \rangle.
\end{align*}
Note that this notion of charge is different from but closely related to the notion of charge at a singularity. To differentiate these two, we might call this one the charge of the monopole and the one we defined earlier the charge of the monopole at a singularity. 

A very important problem in the theory of monopoles, also related to the gluing constructions, is to understand the moduli space of monopoles on $\mathbb{R}^3$ with charge $k$.
The seminal work of Taubes shows that for any charge $k$, there are irreducible $SU(2)$-monopoles with charge $k$ on $\mathbb{R}^3$. Atiyah and Hitchin showed that the moduli space of centered $SU(2)$-monopoles with charge $k$ is a $(4k-4)$-dimensional smooth hyperkähler manifold \cite{MR934202}. In all of these constructions, the BPS-monopole plays a crucial role. The BPS-monopole is an explicit charge +1 solution to the Bogomolny equation which was discovered by Prasad and Sommerfield \cite{prasad1975exact}.

\begin{lemma} \label{deflem}
There is a unique $SU(2)$-monopole on $\mathbb{R}^3$, centred at the origin with mass 1 at infinity and charge $+1$, called the BPS-monopole. Denoting this monopole by $(A_{BPS}, \varPhi_{BPS})$, we have
\begin{align*}
    A_{BPS} (x) = (\frac{1}{\sinh (r)} - \frac{1}{r})(n \times \sigma) \cdot dx, \quad \quad \varPhi_{BPS}(x) = (\frac{1}{\tanh (r)} - \frac{1}{r}) n \cdot \sigma,
\end{align*}
where $r = |x|$, $n = \frac{x}{r}$, $\sigma = \frac{1}{2} (\sigma_1, \sigma_2, \sigma_3) \in \mathbb{R}^3 \otimes \mathfrak{su}(2)$, and
\begin{align*}
    \sigma_1 = 
    \begin{pmatrix}
    0 & -i \\
    -i & 0
    \end{pmatrix}, \quad
    \sigma_2 =  
    \begin{pmatrix}
    0 & -1\\
    1 & 0
    \end{pmatrix}, \quad
    \sigma_3 =  
    \begin{pmatrix}
    -i & 0\\
    0 & i
    \end{pmatrix},
\end{align*}
where $\cdot$ and $\times$ are formal inner and cross product on vectors with three components. 

Although the BPS-monopole looks singular at $\{ 0 \}$, it extends smoothly to the origin. We have $\varPhi_{BPS}(0) = 0$; moreover, this is the only zero of the Higgs field. Furthermore, $|\varPhi(x)| < 1$ for all $x \in \mathbb{R}^3$ and $\lim_{|x| \to \infty}|\varPhi(x)| \to 1$.
\end{lemma}

The idea of constructing an irreducible singular $SU(2)$-monopole with singular points $p_1, \hdots, p_n$ with corresponding charges $k_1, \hdots, k_n$, where $k_i \in \mathbb{N}$, is to start with an Abelian Dirac monopole with singularities at the points $p_i$ with negative charges $k'_i$ such that $k'_i = - k_i$, and some other well-separated singular points $S_q = \{ q_1, \hdots, q_k \}$ all with signed charge $+1$, such that 
\begin{align*}
\text{Total charge } := k + \sum_{i=1}^n k_i' = 0.
\end{align*}

Since the total charge is zero, there exists a reducible Dirac monopole $(A_D, \varPhi_D)$ on an $SU(2)$-bundle with these prescribed singularities. We will glue the scaled BPS-monopoles with charge $+1$ to this background Dirac monopole at the points $q_j$ with charge $+1$ at the singularities. The scaling is necessary since close to the singular points $q_j$, $|\varPhi_D(x)| \to \infty$ as $dist (x,q_j) \to 0$, and therefore, near the singular points $|\varPhi_D|$ is quite large, which implies the Higgs field of the $SU(2)$-monopole which we are gluing to the background monopole close to the point $q_j$ should be large too. We change the mass of the BPS-monopole by a suitable scaling.

The Higgs field of the BPS-monopole $(A_{BPS},\varPhi_{BPS})$ is non-zero on $\mathbb{R}^3 \setminus \{ 0 \}$, and therefore, it induces a decomposition of the adjoint bundle $\underline{\mathbb{R}} \oplus L$, where $\underline{\mathbb{R}}$ is the sub-bundle generated by the image of the Higgs field $\varPhi_{BPS}$ and $L$ is the orthogonal sub-bundle in the adjoint bundle. Corresponding to this decomposition any section of the adjoint-bundle or any adjoint-bundle-valued tensor $f$ supported away from $0 \in \mathbb{R}^3$ can be written as $f = f^L + f^T$, where $f^L$ and $f^T$ are called the longitude and the transverse components, respectively. The following key lemma follows from the work of Jaffe and Taubes \cite[Section IV.1]{MR614447}, also Lemma 2.13 in \cite{MR3801425}, which is fundamental in the gluing construction.

\begin{lemma}\label{preerror}
Let $(a,\varphi)$ be a pair of a connection denoted by $a$ on the $SU(2)$-bundle $P \to \mathbb{R}^3 \setminus \{ 0 \}$ and a section $\varphi$ of the adjoint bundle with finite Yang-Mills-Higgs energy, with charge $k$ and mass $1$ at infinity, which does not necessarily satisfy the Bogomolny equation. Then we have 
\begin{align*}
    |\varPhi_D^L - \varphi^L| =   O(r^{\nu}), \quad 
    |\varPhi_D^T - \varphi^T| =   O(e^{-r}), \quad  |A_D - a| = 
    O(e^{-r}),
\end{align*}
for some $\nu<0$. 

Moreover, for the BPS-monopole with charge $+1$ and mass $1$ centered at the origin, we have
\begin{align*}
    |\varPhi_D ^L- \varPhi_{BPS}^L| =    O(e^{-r}), \quad 
    |\varPhi_D^T - \varPhi_{BPS}^T| =   O(e^{-r}), \quad  |A_D - A_{BPS}| = 
    O(e^{-r}),
\end{align*}
and therefore,
\begin{align*}
    |\varPhi_D - \varPhi_{BPS}| =    O(e^{-r}), \quad  |A_D - A_{BPS}| = 
    O(e^{-r}).
\end{align*}
\end{lemma}

For any scaling factor $\lambda > 1$, we have
\begin{align} \label{dec}
    |\varPhi^{\lambda}_{BPS} - \varPhi_D^{\lambda}| =   O(\lambda e^{-\lambda r}), \quad \quad |A^{\lambda}_{BPS} - A_D^{\lambda}| = 
    O(\lambda e^{-\lambda r}).
\end{align}
The result of scaling Dirac monopoles on $\mathbb{R}^3$ is quite simple. The connection of the Dirac monopole is radially invariant, $A_D^{\lambda}=A_D$. Furthermore, for the Higgs field of the Dirac monopole we have
\begin{align*}
    \varPhi_D^{\lambda}(x) =
    \lambda \varPhi_D(\lambda x) = \lambda ( 1 - \frac{2k}{\lambda r}) = 
    \lambda - \frac{2k}{r} = \varPhi_D(x) + (\lambda - 1).
\end{align*}
Therefore, this scaling just adds the constant $(\lambda - 1)$ to the Higgs field of the Dirac monopole,
\begin{align*}
(A_D^{\lambda},\varPhi_D^{\lambda})(x) = (A_D, \varPhi_D + (\lambda - 1))(x).    
\end{align*}

\begin{remark}
As we increase the average mass of the $SU(2)$-monopole we get closer to the boundary of the moduli space of monopoles. The Dirac monopoles can be understood as a part of the boundary --- or corner --- of the moduli space of monopoles, so the strategy, similar to the other gluing constructions, is to start from a boundary point of the moduli space and then move --- deform --- towards inside. 
\end{remark}

By adding a positive large constant, if necessary, we can assume that the local description of the Higgs field of the monopole, close to each singular point $p_i$ or $q_j$, has the form
\begin{align*}
    \varPhi_D = -\frac{k_i}{2r_i} + m_i + O(r_i),
\end{align*}
with $m_i > 0$. Close to the singular points with positive charges, the Higgs field goes to negative infinity; however, for a fixed positive $\epsilon_0$, we can increase the mass of the Higgs field such that on 
$\overline{M \setminus \cup_j B_{\epsilon_0}(q_j)}$ we have $\varPhi_D \geq \overline{m}/2$, simply because $\overline{M \setminus (\cup_i B_{\epsilon_0}(p_i) \cup_j B_{\epsilon_0}(q_j)})$ is compact and the Higgs field goes to plus infinitiy at the points $p_i$. Similar to the non-compact case, as observed by Oliveira \cite{MR3509964}, a more relevant inequality would be of the type where $\epsilon_0$ depends on the average mass, as in the following lemma.

\begin{lemma} \label{mass}
By increasing the mass of $(A_D, \varPhi_D)$ we would have $\varPhi_D \geq \overline{m}/2$ on $K(\epsilon_0) := \overline{M \setminus \cup_j B_{\epsilon_0}(q_j)}$ where  $\epsilon_0 = \sqrt{2/\overline{m}}$.
\end{lemma}

\begin{proof} 
Let $\epsilon > 0$ be a sufficiently small number such that on $\epsilon$-neighbourhood of singular points $p_i$ or $q_j$,
$\varPhi_D = -k_i/2r_i + m_i + O(r_i)$ for a positive $m_i$. By making $\epsilon$ smaller, if necessary, we can assume 
$\varPhi_D + 1 > -k_i/2r_i + m_i$. Furthermore, by adding a constant to the Higgs field, we would have $\varPhi_D \geq \overline{m}/2$ on $\overline{M \setminus \cup_j B_{\epsilon}(q_j)}$. 

Now we need to show the same holds for $\epsilon_0 < r_i < \epsilon$. It is enough to show that on this region we have $-1/2r_i + m_i - 1 \geq \overline{m}/2$ or, equivalently, $\overline{m}/2 + m_i' \geq 1/2r_i + 1$. By adding a constant to the Higgs field, we assume $\overline{m} \geq 2$ and $\overline{m} \geq 2m_i'$ for all $i$, and therefore, it is enough to have 
\begin{align}
    r_i \geq \frac{1}{ 2\overline{m} - 2}.
\end{align}
This holds if we let $\epsilon_0 = \sqrt{\frac{2}{\overline{m}}}$, which is larger than $ \frac{1}{ 2\overline{m} - 2}$ when $\overline{m}$ is sufficiently large.
\end{proof}

The singular points $p_i$ can be arbitrarily close to each other, but for fixed masses, the construction breaks down as one moves the singular points $q_j$ very close to each other or to the points $p_i$. However, if we allow the average mass to increase, these points can be arbitrarily close. For the gluing construction to work we increase the average mass such that
\begin{align*}
 \min_{i,j} \{dist_{i\neq j}(q_i, q_j),dist_{i,j}(q_i, p_j)\} \geq \epsilon_0 = \sqrt{\frac{2}{\overline{m}}}.   
\end{align*}

\subsection{The Construction of Approximate Monopole}
The approximate monopole we are constructing is equal to the Dirac monopole on $M \setminus \cup_j B_{2 \epsilon_j}(q_j)$ with a large mass, and equal to the pull-back of an appropriately scaled BPS-monopole on each $B_{ \epsilon_j}(q_j)$, where the scaling factor $\lambda_j = m_j$ is the mass of the Dirac monopole at $q_j$. This can be done after identifying a small neighbourhood of $q_j$ with a neighbourhood of the origin in $\mathbb{R}^3$ and the bundles above them. 

Fix a diffeomorphism between $2\epsilon_j$-neighbourhood of the singular point $q_j$ and a neighbourhood of origin in $\mathbb{R}^3$ using the geodesic normal coordinates,
\begin{align*}
    \eta_j: B_{2 \epsilon_j}(q_j) \subset M \to \mathbb{R}^3.
\end{align*}
Moreover, we can fix an identification between the associated vector bundles above these neighbourhoods covering $\eta_j$, which by an abuse of notation, we also denote this bundle map, called the framing, by $\eta_j$,
\begin{align*}
    \eta_j:   (\underline{\mathbb{R}} \oplus H_{q_j})_{B_{2 \epsilon_j}(q_j) \setminus \{q_j \}}
    \rightarrow
    \underline{\mathfrak{su}(2)}_{\mathbb{R}^3 \setminus \{ 0 \}}.
\end{align*}
Using these identifications we can pull back the scaled BPS-monopoles to the $2\epsilon_j$-neighbourhood of $q_j$. Although the Dirac monopole is not defined at the point $q_j$, the pull-back of the BPS-monopole and the bundle it is defined on extend smoothly over $q_j$. Different identifications can result in different pairs on $M \setminus S_p$, even up to gauge. Up to isomorphism there is a $U(1)$-freedom in the choice of the framing for each $q_j$, and assuming we have $k$ such points we would get $k$ parameters for the choices of framings --- up to gauge $(k-1)$ parameters. 

Suppose for each $q_j$, a framing $\eta_j$ is fixed. Now we can pull back the bundles and the scaled BPS-monopoles to $B_{2 \epsilon_j}(q_j)$. We denote these local pairs by
\begin{align*}
(\eta_j^*(A^{\lambda_j}_{BPS}), \eta_j^*(\varPhi^{\lambda_j}_{BPS})).
\end{align*}
Using cut-off functions we can glue these local monopoles to the background monopole. Suppose $\xi_j$ is a cut-off function supported around $q_j$ such that
\begin{align*}
    \xi_j = 
    \begin{cases}
        1 \quad \text{ on } \quad B_{ \epsilon_j}(q_j)\\
        0 \quad \text{ on } \quad M \setminus B_{2 \epsilon_j}(q_j),
    \end{cases}
\end{align*}
\noindent
and
\begin{align*}
    \xi_0 = 
    \begin{cases}
        1 \quad \text{ on } \quad M \setminus \cup_j B_{2 \epsilon_j}(q_j)\\
        0 \quad \text{ on } \quad \cup_j B_{ \epsilon_j}(q_j),
    \end{cases}
\end{align*}
where on $\epsilon_j \leq r_j \leq 2\epsilon_j$ we have $\xi_0 + \xi_j = 1$ for each $j \in \{1, \hdots, k\}$, and 
\begin{align*}
|\nabla \xi_j | \leq 2 \epsilon_j^{-1} \quad \text{ and } \quad |\nabla \xi_0 | \leq 2 \max_{j \in \{ 1, \hdots,k \}} \{ \epsilon_j^{-1} \}.  
\end{align*}
The approximate monopole has the form
\begin{align*}
    (A_0, \varPhi_0) = 
    \xi_0 (A_D, \varPhi_D) +
    \sum_{j=1} ^k \xi_j (\eta_j^*(A^{\lambda_j}_{BPS}), \eta_j^*(\varPhi^{\lambda_j}_{BPS})).
\end{align*}
Note that the assumption $\xi_0 + \xi_j = 1$ assures $A_0$ is a connection.

\subsection{Pointwise Approximation of the Error}

This pair $(A_0, \varPhi_0)$ is an approximate solution and does not necessarily satisfy the Bogomolny equation. We define the error term by
\begin{align*}
    e_0 = *F_{A_0} - d_{A_0} \varPhi_0.
\end{align*}
In this section, we estimate the error term $e_0$ in different regions on $M$.

$e_0$ is zero on $M \setminus (\cup_j B_{2 \epsilon_j}(q_j) \cup S_p)$, since on this region the approximate monopole is equal to the Dirac monopole. The error term is non-zero on $\cup_j B_{ 2\epsilon_j}(q_j)$. It is correct that on each $B_{ \epsilon_j}(q_j)$ the approximate monopole is equal to the pull-back of the scaled BPS-monopole, but we only know that the scaled BPS-monopole is a monopole with respect to the Euclidean metric and not with respect to the arbitrary Riemannian metric $g$ on $M$. The error term $e_0$ is also non-zero on the necks $\cup_j (B_{2 \epsilon_j}(q_j) \setminus B_{ \epsilon_j}(q_j))$, both because $g$ is not necessarily flat, and also because of the use of the cut-off functions. 

\begin{lemma} \label{errores}
Let $\epsilon_j = \lambda_j^{-\frac{1}{2}}$. Then when the average mass is sufficiently large, we have the following pointwise error estimate,
\begin{align*}
    (e_0)_{|_{B_{2\epsilon_j}(q_j)}}  = O(1).
\end{align*}
\end{lemma}

\begin{proof}
We denote the error coming from the manifold not being flat around $q_j$ by
\begin{align*}
    e^{BPS}_j:= \left( *F_{\eta_j^*(A^{\lambda_j}_{BPS})} - d_{\eta_j^*(A^{\lambda_j}_{BPS})}\eta_j^*(\varPhi_{BPS}^{\lambda_j}) \right)_{|_{B_{2\epsilon_j} (q_j)}}.
\end{align*}
The size of the error depends on how much the metric $g$ is different from the Euclidean metric. For the comparison between the metric $g$ on $B_{2\epsilon_j} (q_j)$ and the Euclidean metric, we should first pull back the Euclidean metric $g_0$ to $B_{2\epsilon_j} (q_j)$ using the same map that we used to pull back the scaled BPS-monopole to $B_{2\epsilon_j} (q_j)$. We denote the Euclidean metric pulled back to $B_{2\epsilon_j} (q_j)$ and its Hodge star operator by $g_0$ and $*_0$, respectively.
\begin{align*}
    (e^{BPS}_j)_{|_{B_{2\epsilon_j} (q_j)}}  &= *F_{\eta_j^*(A^{\lambda_j}_{BPS})} - d_{\eta_j^*(A^{\lambda_j}_{BPS})}\eta_j^*(\varPhi_{BPS}^{\lambda_j}) 
    = (*_0F_{\eta_j^*(A^{\lambda_j}_{BPS})} - d_{\eta_j^*(A^{\lambda_j}_{BPS})}\eta_j^*(\varPhi_{BPS}^{\lambda_j}))
    \\&+ (*F_{\eta_j^*(A^{\lambda_j}_{BPS})} - *_0F_{\eta_j^*(A^{\lambda_j}_{BPS})}) 
    = *F_{\eta_j^*(A^{\lambda_j}_{BPS})} - *_0F_{\eta_j^*(A^{\lambda_j}_{BPS})} \\&= 
    (*-*_0)F_{\eta_j^*(A^{\lambda_j}_{BPS})}.
\end{align*}
Note that
\begin{align*}
*_0F_{\eta_j^*(A^{\lambda_j}_{BPS})} - d_{\eta_j^*(A^{\lambda_j}_{BPS})}\eta_j^*(\varPhi_{BPS}^{\lambda_j}) = 0,    
\end{align*}
since $(A^{\lambda_j}_{BPS},\varPhi_{BPS}^{\lambda_j})$ is a monopole with respect to the Euclidean metric $g_0$.

For each $j$, for sufficiently small $\epsilon_j$, fix a local geodesic normal coordinate system on $B_{2\epsilon_j} (q_j)$, denoted by $(x_1,x_2,x_3)$. We can think about the components of the Riemannian metric $g_{k,l}$ in this coordinate system as real-valued functions defined on this neighbourhood of $q_j$. We can write down the Taylor series expansion of these functions around the origin, which here corresponds to $q_j$,
\begin{align*} 
    g_{k,l}(x) = \delta_k^l + \frac{1}{3}\sum_{m,n}R_{klmn} x_m x_n + O(|x|^3), \\
    g^{k,l}(x) = \delta_k^l - \frac{1}{3}\sum_{m,n}R_{klmn} x_m x_n + O(|x|^3),
\end{align*}
where $R_{klmn}$ is the $(4,0)$-Riemann curvature tensor. Furthermore,
\begin{align*}
    vol_g(x) = \left( 1 - \frac{1}{6}\sum_{m,n}R_{m,n}x_m x_n + O(|x|^3) \right) dx_1 dx_2 dx_3,
\end{align*}
where $R_{m,n}$ denotes the Ricci curvature tensor. 

For any $2$-form $\beta \in \Omega^2(M;V)$ with values in any vector bundle $V$ equipped with a fiber-wise inner product  $\langle -,-\rangle$, which in the given coordinates system $\beta = \sum_{\text{cyclic }i,j,k}\beta_i dx_j dx_k$, we have
\begin{align*}
    |\beta|_g^2 = \sum_{i,j} \langle\beta_k, \beta_l\rangle g^{k,l}
    = |\beta|^2_{g_0} - \frac{1}{3}\sum_{k,l,m,n}R_{klmn} \langle\beta_k, \beta_l\rangle x_m x_n + O(|x|^3).
\end{align*}
On the other hand 
\begin{align*}
    \langle\beta \wedge * \beta\rangle &= |\beta|_g^2 vol_g,
\end{align*}
here $\langle - \wedge -\rangle$ is wedge product on the real differential form parts and inner product on $V$-valued parts.
\begin{align*}
    \langle \beta \wedge * \beta \rangle =
    \left( 
    |\beta|_{g_0}^2
    - \sum_{m,n}
    (
    \frac{1}{3}\sum_{k,l}R_{klmn} \langle\beta_k, \beta_l \rangle + \frac{1}{6}|\beta|^2_{g_0}R_{m,n}
    )
    x_m x_n  +O(|x|^3)
    \right) dx_1 dx_2 dx_3,
\end{align*}
and therefore,
\begin{align*}
    \beta \wedge (*-*_0)\beta = 
    \left( -\sum_{m,n}( \frac{1}{3}\sum_{k,l}R_{klmn} \langle\beta_k, \beta_l\rangle  + \frac{1}{6}|\beta|^2_{g_0}R_{m,n}
    )x_m x_n + O(|x|^3) \right) dx_1 dx_2 dx_3
\end{align*}
hence
\begin{align*}
    (*-*_0)\beta = 
    -\sum_k (\frac{1}{3}\sum_{l,m,n}R_{klmn}\beta_l x_m x_n 
    + \frac{1}{6} \sum_{l,m,n} \beta_l g^{k,l} R_{m,n} x_m x_n 
    )dx_k + O(|x|^3), 
\end{align*}
and therefore, pointwise and with respect to the metric $g$,
\begin{align*}
    |(*-*_0)\beta|_{g} \leq C |\beta|_g |x|^2,
\end{align*}
where the constant $C$ depends only on the curvature tensor of $(M,g)$. Going back to the curvature $2$-form $F_{\eta_j^*(A^{\lambda_j}_{BPS})}$ on ${B_{2\epsilon_j} (q_j)}$, the computations above show 
\begin{align*}
    |(*_0 - *)F_{\eta_j^*(A^{\lambda_j}_{BPS})}|_g \leq C |F_{\eta_j^*(A^{\lambda_j}_{BPS})}|_{g} |x|^2.
\end{align*}

Following \cite[Section IV.1]{MR614447} and using the same notations as in Definition \ref{deflem}, 
\begin{align*}
    (d_{A_{BPS}} \varPhi_{BPS})^L &= 
    (\frac{1}{\sinh^2(|x|)}-\frac{1}{|x|^2}) (n \cdot  \sigma) n \cdot  dx, \\
    (d_{A_{BPS}} \varPhi_{BPS})^T &= 
    (\frac{1}{|x|}-\frac{1}{\tanh(|x|)})(\frac{1}{\sinh(|x|)}) (\sigma - (n \cdot \sigma)\sigma) \cdot dx.
\end{align*}
Although $(d_{A_{BPS}} \varPhi_{BPS})^L$ and $(d_{A_{BPS}} \varPhi_{BPS})^T$ look singular at the origin, they extend smoothly to the origin. In fact, the maximum of both of these components are achieved at the origin,
\begin{align*}
    |(d_{A_{BPS}} \varPhi_{BPS})^L|_{g_0} \leq \frac{1}{3}, \quad 
    |(d_{A_{BPS}} \varPhi_{BPS})^T|_{g_0} \leq \frac{1}{3}.
\end{align*}
Furthermore, 
\begin{align*}
    |d_{A^{\lambda_j}_{BPS}} \varPhi^{\lambda_j}_{BPS}|_{g_0} \leq \frac{C}{\lambda_j^{-2} + |x|^2},
\end{align*}
for a constant $C>0$.

For any $2$-form $\beta$, with values in any vector bundle, we have
\begin{align*}
    |\beta|_g^2 - |\beta|_{g_0}^2 = \sum_{k,l} \langle\beta_k, \beta_l\rangle (g^{k,l} - g_0^{k,l}) &=  - \frac{1}{3} \sum_{k,l,m,n} \langle\beta_k, \beta_l\rangle R_{klmn} x_m x_n + O(|x|^3) \\ &\leq C R | \beta|^2_{g_0}|x|^2,
\end{align*}
where $C > 0$ is a constant and $R$ is the maximum of the Riemann curvature tensor of $g$. Therefore, 
\begin{align*}
    |\beta|_g^2 \leq  |\beta|_{g_0}^2 +  C R | \beta|^2_{g_0}|x|^2,
\end{align*}
Let $\beta = (*_0 - *) F_{\eta_j^*(A^{\lambda_j}_{BPS})}$
\begin{align*}
    |(*_0 - *)F_{\eta_j^*(A^{\lambda_j}_{BPS})}|_g^2 \leq C \left( \frac{|x|^4}{(\lambda_j^{-2} + |x|^2)^2} +  R \frac{|x|^6}{(\lambda_j^{-2} + |x|^2)^2} \right),
\end{align*}
for a constant $C$, when $\lambda_j$ is sufficiently large.

These sum up to 
\begin{align} 
    (e_j^{BPS})_{|_{B_{2\epsilon_j} (q_j)}} \leq C',
\end{align}
for a positive constant $C'$.

On the neck ${B_{2\epsilon_j}(q_j)}\setminus {B_{\epsilon_j} (q_j)}$, the cut-off function is another source of error. On this region, we have
\begin{align*}
    e_0 =& \sum_{j=1} ^k \bigg( (*F_{\eta_j^*(A^{\lambda_j}_{BPS})} - d_{\eta_j^*(A^{\lambda_j}_{BPS})}\eta_j^*(\varPhi_{BPS}^{\lambda_j}))_{|_{ B_{ \epsilon_j}(q_j)}}\\ 
    +& \xi_0 (*d_{\eta_j^*(A^{\lambda_j}_{BPS})}
    (A^{\lambda_j}_D - \eta_j^*(A^{\lambda_j}_{BPS}))
    - 
    d_{\eta_j^*(A^{\lambda_j}_{BPS})}(\varPhi^{\lambda_j}_D-\eta_j^*(\varPhi^{\lambda_j}_{BPS})
    )\\
    +& *(d\xi_0 \wedge (A^{\lambda_j}_D-\eta_j^*(A^{\lambda_j}_{BPS})))
    - d \xi_0 (\varPhi^{\lambda_j}_D - 
    \eta_j^*(\varPhi^{\lambda_j}_{BPS}))\\
    +& (\xi_0(A^{\lambda_j}_D - \eta_j^*(A^{\lambda_j}_{BPS})))^2 -
    \xi_0^2 [A^{\lambda}_D - \eta_j^*(A^{\lambda}_{BPS}),\varPhi^{\lambda}_D - \eta_j^*(\varPhi^{\lambda}_{BPS})] \bigg),
\end{align*}
where $*$ is the Hodge star of $g$ on $M$. 

First consider the case where the Riemannian metric $g$ is flat, as we were to glue a scaled BPS-monopole to a scaled Dirac monopole on $\mathbb{R}^3$. Then following Lemma \ref{preerror}, we would have
\begin{align*}
    |\varPhi^{\lambda_j}_{BPS} - \varPhi_D^{\lambda_j}| =  O(\lambda_j e^{ -\lambda_j r_j  }) , \quad
    |A^{\lambda_j}_{BPS} - A_D^{\lambda_j}| = 
    O(\lambda_j e^{ -\lambda_j r_j}).
\end{align*}
and therefore, on this region
\begin{align*}
    |\xi_0 (*d_{\eta_j^*(A^{\lambda_j}_{BPS})}
    (A^{\lambda_j}_D - \eta_j^*(A^{\lambda_j}_{BPS}))| &\leq c_1 (\lambda_j^2 e^{-{\lambda_j}r_j}),\\
    |d_{\eta_j^*(A^{\lambda_j}_{BPS})}(\varPhi^{\lambda_j}_D-\eta_j^*(\varPhi^{\lambda_j}_{BPS}))| &\leq c_2 (\lambda_j^2 e^{-{\lambda_j}r_j}),\\
    |(d\xi_0 \wedge (A^{\lambda_j}_D-\eta_j^*(A^{\lambda_j}_{BPS})))|
    &\leq c_3(\frac{\lambda_j}{\epsilon_j}e^{-{\lambda_j}r_j}),\\
    |d \xi_0 (\varPhi^{\lambda_j}_D - 
    \eta_j^*(\varPhi^{\lambda_j}_{BPS}))| &\leq c_4 (\frac{\lambda_j}{\epsilon_j}e^{-{\lambda_j}r_j}),\\
    |(\xi_0(A^{\lambda_j}_D - \eta_j^*(A^{\lambda_j}_{BPS})))^2| &\leq
    c_5 (\lambda_j^2 e^{-2{\lambda_j}r_j}),\\
    | \xi_0^2 [A^{\lambda}_D - \eta_j^*(A^{\lambda}_{BPS}),\varPhi^{\lambda}_D - \eta_j^*(\varPhi^{\lambda}_{BPS})]| &\leq
    c_6 (\lambda_j^2 e^{-2{\lambda_j}r_j}),
\end{align*}
for constants $c_1, \hdots, c_6$, independent of $\epsilon_j$ and $\lambda_j$.

Here $\lambda_j$ and $\epsilon_j$ should be understood as a very large and a very small number, respectively. As we increase $\lambda_j$, we can make $\epsilon_j$ smaller. Although there is no unique choice for these parameters for the gluing construction to work, sometimes there are choices which minimize the error of the approximate solution. 

Let's let $\epsilon_j = \lambda_j^l$ for some $-1<l<0$. For $l$ outside of this interval the errors listed above can be large. The appropriate value for $l$ depends on the functional spaces we choose to work with. For sufficiently large $\lambda_j$, the leading term of the bounds for the error $e_j^{neck}$, up to a constant, would be $\lambda_j^2e^{-{\lambda_j}^{l+1}}$. These errors are exponentially small and favorable. 

However, the case over arbitrary Riemannian 3-manifolds is different, since the Green's function on a neighbourhood of a point $q_j$ is not necessarily equal to $-\frac{1}{2 r_j}+m_j$, but potentially there are higher order terms, and therefore, 
\begin{align*}
    |\varPhi^{\lambda_j}_{BPS} - \varPhi_D^{\lambda_j}| &=
    |\varPhi^{\lambda_j}_{BPS} - (-\frac{k_j}{2r_j} + m_j)| + 
    |(-\frac{k_j}{2r_j} + m_j) - \varPhi_D^{\lambda_j}| \\&= 
    O(\lambda_j e^{ -\lambda_j r_j  }) 
    + O(r_j), 
\end{align*}
and similarly,
\begin{align*}
    |A^{\lambda_j}_{BPS} - A_D^{\lambda_j}| = 
    O(\lambda_j e^{ -\lambda_j r_j  }) 
    + O(r_j).
\end{align*}
and therefore, for sufficiently large $\lambda_j$, 
\begin{align*}
    |\varPhi^{\lambda_j}_{BPS} - \varPhi_D^{\lambda_j}| = O(r_j), \quad
    |A^{\lambda_j}_{BPS} - A_D^{\lambda_j}| =  O(r_j).
\end{align*}
hence, for $l = - \frac{1}{2}$,
\begin{align*}
    |\xi_0 (*d_{\eta_j^*(A^{\lambda_j}_{BPS})}
    (A^{\lambda_j}_D - \eta_j^*(A^{\lambda_j}_{BPS}))| &\leq c_1,\\
    |d_{\eta_j^*(A^{\lambda_j}_{BPS})}(\varPhi^{\lambda_j}_D-\eta_j^*(\varPhi^{\lambda_j}_{BPS}))| &\leq c_2,\\
    |(d\xi_0 \wedge (A^{\lambda_j}_D-\eta_j^*(A^{\lambda_j}_{BPS})))|
    &\leq c_3(\frac{r_j}{\epsilon_0}),\\
    |d \xi_0 (\varPhi^{\lambda_j}_D - 
    \eta_j^*(\varPhi^{\lambda_j}_{BPS}))| &\leq c_4(\frac{r_j}{\epsilon_0}),\\
    |(\xi_0(A^{\lambda_j}_D - \eta_j^*(A^{\lambda_j}_{BPS})))^2| &\leq
    c_5 (r_j^2),\\
    | \xi_0^2 [A^{\lambda}_D - \eta_j^*(A^{\lambda}_{BPS}),\varPhi^{\lambda}_D - \eta_j^*(\varPhi^{\lambda}_{BPS})]| &\leq
    c_6 (r_j^2),
\end{align*}
where the constants $c_1, \hdots, c_6$ are independent of $\epsilon_j$ and $\lambda_j$, and only depend on the geometry of $(M,g)$. We denote this error on the neck containing the terms $A^{\lambda_j}_D - \eta_j^*(A^{\lambda_j}_{BPS})$ and $\varPhi^{\lambda_j}_D - \eta_j^*(\varPhi^{\lambda_j}_{BPS})$ by $e^{neck}_j$. We get 
\begin{align*}
        (e_0)_{|_{B_{2 \epsilon_j}(q_j) \setminus B_{ \epsilon_j}(q_j)}} & = O(1).
\end{align*}
\end{proof}

\section{Solving the Equation, the Linear Theory}\label{solveequation}

The goal is to show there is a solution to the Bogomolny equation near the constructed approximate monopole $(A_0, \varPhi_0)$. In other words, we are looking for a small $(a,\varphi)$ --- small in a suitable norm --- such that $(A_0 + a, \varPhi_0 + \varphi)$ is a genuine monopole. In this section, we set up the equations for $(a,\varphi)$ and state the strategy to solve these equations.

We can write the equation for the pair $(a,\varphi)$,
\begin{align*}
    *F(A_0 + a)-d_{A_0+a}(\varPhi_0+\varphi) &= 0 \Rightarrow \\
    \label{eq1} (*F_{A_0} -d_{A_0}\varPhi_0)+ &(*d_{A_0}a - d_{A_0}\varphi -[a,\varPhi_0]) + (*\frac{[a\wedge a]}{2}-[a,\varphi]) = 0.
\end{align*}
Let $d_2^{(A_0, \varPhi_0)}: \Omega^1(M \setminus S_p,\mathfrak{g}_P) \oplus \Omega^0(M  \setminus S_p,\mathfrak{g}_P) \to \Omega^1(M  \setminus S_p,\mathfrak{g}_P)$ be the operator that appeared in the linearization of the Bogomolny equation at $(A_0, \varPhi_0)$,
\begin{align}
    d_2^{(A_0, \varPhi_0)}(a,\varphi) = *d_{A_0}a - d_{A_0}\varphi -[a,\varPhi_0],
\end{align}
where $S_p = \{ p_1,  \hdots , p_n \}$. Although $d_2^{(A_0, \varPhi_0)}$ depends on the pair $(A_0, \varPhi_0)$, whenever there is no fear of confusion, we drop the subscript $(A_0, \varPhi_0)$ and denote it by $d_2$. 

Let $Q(a,\varphi): \Omega^1(M \setminus S_p,\mathfrak{g}_P) \oplus \Omega^0(M \setminus S_p,\mathfrak{g}_P) \to \Omega^1(M \setminus S_p,\mathfrak{g}_P)$ be the operator defined by the quadratic part,
\begin{align*}
   Q(a,\varphi) = *\frac{[a\wedge a]}{2}-[a,\varphi].
\end{align*}
Equation \ref{moneq} can be written as
\begin{align} \label{mmeq}
    (d_2+Q)(a,\varphi) = - e_0.
\end{align}
The Bogomolny equation is invariant under the action of the gauge group, and therefore, not elliptic. In fact, it is elliptic modulo the action of the gauge group. The linearization of the gauge group action is given by
\begin{align*}
    &d_1^{(A_0, \varPhi_0)}:\Omega^0(M \setminus S_p,\mathfrak{g}_P) \to \Omega^1(M \setminus S_p,\mathfrak{g}_P) \oplus \Omega^0(M \setminus S_p,\mathfrak{g}_P), \\ &d_1^{(A_0, \varPhi_0)} \xi = (-d_{A_0} \xi,-[\varPhi_0, \xi]).
\end{align*}
Similar to $d_2$, this operator also depends on the pair ${(A_0, \varPhi_0)}$, but we drop this subscript when there is no fear of confusion, and denote it by $d_1$. 

The gauge fixing equation $d_1^*(a,\varphi) = 0$ describes a local slice of the action of the gauge group at $(A_0, \varPhi_0)$, where 
\begin{align*}
    d_1^*(a,\varphi) = -d_{A_0}^*a - [\varPhi_0,\varphi],
\end{align*}
is the formal adjoint of $d_1$ with respect to the $L^2$-inner product. Let 
\begin{align*}
D_{(A_0,\varPhi_0)}:  \Omega^1(M \setminus S_p,\mathfrak{g}_P) \oplus \Omega^0(M \setminus S_p,\mathfrak{g}_P) \to \Omega^1(M \setminus S_p,\mathfrak{g}_P) \oplus \Omega^0(M \setminus S_p,\mathfrak{g}_P), 
\end{align*}
be the elliptic operator defined by 
\begin{align*}
    D: = D_{(A_0,\varPhi_0)} = d_2 \oplus d_1^*.
\end{align*} 
This can be used to define an elliptic equation. Instead of $d_2(a, \varphi) = f$, we can consider the equation 
\begin{align} \label{01}
    D(a, \varphi) = (f, 0).
\end{align}
Two important properties of \ref{01}: it is elliptic, and, for any small $f$, any solution of $d_2(a, \varphi) = f$ can be gauged into a solution of \ref{01}.

These operators fit into a sequence 
\begin{align} \label{seq}
    \Omega^0(M \setminus S_p,\mathfrak{g}_P)
    \xrightarrow{d_1}
    \Omega^1(M \setminus S_p,\mathfrak{g}_P) \oplus 
    \Omega^0(M \setminus S_p,\mathfrak{g}_P) 
    \xrightarrow{d_2}
    \Omega^1(M \setminus S_p,\mathfrak{g}_P).
\end{align}
Note that $d_2 \circ d_1 \xi = *[ (d_{A_0}\varPhi_0 - *F_{A_0}) \wedge \xi ]$, and therefore, $d_2 \circ d_1 = 0$ when $(A_0, \varPhi_0)$ is a monopole. In fact, \ref{seq} is an elliptic complex when $(A_0, \varPhi_0)$ is a monopole. 

The formal adjoint of $d_2$ with respect to the $L^2$-inner product is given by 
\begin{align*}
    d_2^*: \Omega^1(M \setminus S_p) \to \Omega^1(M \setminus S_p) \oplus \Omega^0(M \setminus S_p), \; 
      d_2^* u = (*d_{A_0} u + [u, \varPhi_0], -d^*_{A_0} u).
\end{align*}
We look for solutions to the equation \ref{mmeq}, which are of the form $(a,\varphi) = d_2^* u$, and therefore, the equation \ref{mmeq} can be written as
\begin{align} \label{meq}
    ( d_2 d_2^* +Qd_2^* )u = - e_0.
\end{align}
This equation is elliptic. In fact, $d_2 d_2^*$ has the same symbol as the Laplacian on $\mathfrak{su}(2)$-valued 1-forms. A key step in solving this equation would be solving the linear equation
\begin{align} \label{lin}
    d_2 d_2^*u = f.
\end{align}

The method for solving this linear equation can be summarized into 4 steps:

\begin{itemize}
    \item solving the linearized equation $d_2 d_2^* u = f$ on $B_{3 \epsilon_j}(q_j)$ via a variational method;
    \item  solving the linearized equation  $d_2 d_2^* u = f$ on $M \setminus (\cup_j B_{2\epsilon_j}(q_j) \cup S_p)$ via a variational method;
    \item solving the linearized equation  $d_2 \xi = f$ on $M \setminus  S_p$ via an iteration method;
    \item  solving the Bogomolny equation on $M \setminus  S_p$ using a fixed point theorem.
\end{itemize}

\subsection{Analytic Preliminaries}\label{AnalyticPreliminaries}
In this section, we review the necessary background material to solve the linearized Bogomolny equation.

We start with the monopole Weitzenb\"ock formulas. These formulas follow from the standard Weitzenb\"ock formula for a connection on a vector bundle.

\begin{lemma}[The Monopole Weitzen\"ock Formulas  \cite{MR866030}] \label{monopoleweitzenbock}
Let $(A_0,\varPhi_0)$ be a pair of a connection and a Higgs field on a principal bundle $P \to M$ where $(M,g)$ is an oriented Riemannian 3-manifold. Let $e_0 = *F_{A_0} - d_{A_0} \varPhi_0$. Let $ad^2(\varPhi_0) \xi = [\varPhi_0, [\varPhi_0, \xi]]$. Let $u \in \Omega^1(M,\mathfrak{g}_P)$. The Monopole Weitzenb\"ock formulas are 
\begin{align} \label{Weitzenbockformula}
    d_2 d_2^* u &= \nabla_{A_0}^* \nabla_{A_0} u - ad(\varPhi_0)^2(u) + Ric(u) + *[e_0 \wedge u],\\ 
    D D^*(a, \varphi) &= 
    \nabla_{A_0}^* \nabla_{A_0} (a, \varphi) - ad(\varPhi_0)^2(a, \varphi)  + Ric(a, \varphi)
      + *[e_0 \wedge (a, \varphi)],\\
    D^* D (a, \varphi) &= D D^* (a, \varphi) + 2 \langle d_{A_0} \varPhi_0, (a, \varphi) \rangle.
\end{align}
\end{lemma}

Variations of the Poincar\'e inequality are essential in the analysis of the linear problem. The standard Poincar\'e inequality $\|u\|_{L^p(U)} \leq C \| \nabla u \|_{L^p(U)}$ is stated for compactly supported functions $u \in W^{1,p}(U)$, where $U \subset \mathbb{R}^n$ is a bounded domain and $C$ is a positive constant. This inequality is also valid when $U$ is a ball in a Riemmanian manifold $(M,g)$ with a positive constant $C$ which depends on the geometry of $U$. A variation of this inequality also holds for the compactly supported functions on $\mathbb{R}^n$, when $n \geq 2$.

\begin{lemma}[The Gagliardo-Nirenberg-Sobolev Inequality \cite{MR1688256}] \label{CNSI-inequality}
Let $n \geq 2$ and $1 \leq p < n$. Let $p^*$ be the Sobolev
conjugate of $p$; i.e., $p^*$ satisfies $\frac{1}{p^*} = \frac{1}{p} - \frac{1}{n}$. Then
\begin{align*}
    \|u\|_{L^{p^*}(\mathbb{R}^n)} \leq C \| \nabla u \|_{L^p(\mathbb{R}^n)},
\end{align*}
for a constant $C$ which depends on $n$ and $p$ and for all compactly supported functions $u \in C^1_c(\mathbb{R}^n)$.

Moreover, this inequality holds when $\mathbb{R}^n$ is equipped with a metric $g$ which is asymptotically Euclidean rather than Euclidean, for a positive constant $C_g$.
\end{lemma}

The space of smooth compactly supported functions on $\mathbb{R}^n$ is dense in $W^{1,p}(\mathbb{R}^n)$, and therefore, Lemma \ref{CNSI-inequality} holds for $u \in W^{1,p}(\mathbb{R}^n)$ when $ n \geq 2$. 

Now we turn to weighted Poincar\'e inequalities. A one dimensional version of this inequality states that if $f: \mathbb{R} \to \mathbb{R}$ is a non-negative function, $F(x) = \int_0 ^xf(t) dt$, and $p>1$, then
\begin{align*}
\int_0^{\infty}
\left(\frac{F(x)}{x}\right)^p dx
\leq \left(\frac{p}{p-1}\right)^p
\int_0^{\infty} f(x)^p dx.
\end{align*}
This inequality, which is called the Hardy's inequality, was first proved by Hardy; however, the constant $(\frac{p}{p-1})^p$ in this inequality, which is sharp, was later discovered by Landau. For a proof consult with the beautiful book `Inequalities' written by Hardy, Littlewood, and P\'olya \cite[Section 9.8]{MR0046395}.

Lewis proved a higher-dimensional version of this inequality \cite{MR654855}, which --- a special case of that --- is stated below. 

\begin{lemma} 
Let $g: \mathbb{R}^n \to \mathbb{R}$ be in $W^{2,2}(\mathbb{R}^n)$. Then for all $u \in C^{1}_c(\mathbb{R}^n)$ we have
\begin{align*}
    \int_{\mathbb{R}^n} |\Delta g(x)| |u(x)|^2 vol_{\mathbb{R}^n}
    \leq
    2 \int_{\mathbb{R}^n} |\nabla g(x)| &|u(x)| |\nabla u (x)|
    vol_{\mathbb{R}^n}
    \\ &\leq
    4 \int_{\mathbb{R}^n} |\Delta g(x)|^{-1} |\nabla g(x)|^2 |\nabla u (x)|^2
    vol_{\mathbb{R}^n}.
\end{align*}
In particular, when $ n \geq 2$, 
\begin{align*}
    \int_{\mathbb{R}^n} |x|^{\beta - 2} |u(x)|^2 vol_{\mathbb{R}^n}
    \leq
    \frac{4}{(\beta -2 + n)^2} \int_{\mathbb{R}^n} |x|^{\beta}|\nabla u (x)|^2
    vol_{\mathbb{R}^n},
\end{align*}
for all $u \in C^1_c(\mathbb{R}^n)$.
\end{lemma}

Using Kato’s inequality, we can extend these results about real-valued functions on $\mathbb{R}^n$ to $\mathfrak{su}(2)$–valued forms and their covariant derivatives. For a $\lambda \geq 0$, let 
\begin{align} \label{w}
    w(x) = 
    \begin{cases}
        \sqrt{\lambda^{-2} + |x|^2}, \quad &|x|\leq \frac{1}{2}\\
        1, \quad &|x| \geq 1.
    \end{cases}
\end{align}

\begin{lemma}[\cite{MR3801425}]\label{Hardy} 
For all $\alpha \neq -1$ and $u \in C_0^{\infty}(\mathbb{R}^3, \mathfrak{su}(2))$ we have
\begin{align} \label{HLP}
\int_{\mathbb{R}^3} w^{-2\alpha-3}|u|^2 vol_{\mathbb{R}^3}\leq \frac{1}{(\alpha+1)^2} \int_{\mathbb{R}^3} w^{-2\alpha-1}|\nabla_Au|^2vol_{\mathbb{R}^3}.
\end{align}
Moreover, this inequality is valid on asymptotically Euclidean manifolds, for a different positive constant $C_{\alpha}$, depending on $\alpha$ and the geometry of the manifold.
\end{lemma}

\subsection[The Linear Equation over $B_{3 \epsilon_j}(q_j)$]{The Linear Equation over $\pmb{B_{3 \epsilon_j}(q_j)}$}\label{linearBPS}
We break solving the linear equation into three parts. In the first part we analyze and solve the linear equation close to the points $q_j$. In the second part we analyze and solve the linear equation away from the points $q_j$, and finally, in the third part we solve the linear equation on the whole manifold. In this section, we set up the framework for studying the linear equation $d_2d_2^*u=f$ on $B_{3 \epsilon_j}(q_j)$, and transfer it to $\mathbb{R}^3$.

Suppose $f$ is supported on $B_{3 \epsilon_j}(q_j)$. We can localize the problem on this region. When the metric is flat on this neighborhood of the point $q_j$, it is convenient to consider $B_{3 \epsilon_j}(q_j)$ as a subset of $\mathbb{R}^3$. This reduces the problem to solving the equation $d_2d_2^*u=f$ on $\mathbb{R}^3$. 

More generally, when the metric around $q_j$ is an arbitrary one, it is still useful to consider $B_{3 \epsilon_j}(q_j)$ as a subset of $\mathbb{R}^3$, but with a non-standard metric. Using a geodesic normal coordinate, we fix a diffeomorphism 
\begin{align*}
    \mu : B_{3 \epsilon_j}(q_j) \subset M \to U \subset \mathbb{R}^3.
\end{align*}
We can pull back the Riemannian metric on $B_{3 \epsilon_j}(q_j)$ via $\mu^{-1}$ to $U \subset \mathbb{R}^3$. Furthermore, one can extend this metric defined on $U$ to a Riemannian metric defined on the whole $\mathbb{R}^3$ such that it is flat outside of a slightly larger open subset $V$ with geodesic radius $4 \epsilon_j$ containing $U$. To prevent any confusion, we denote $\mathbb{R}^3$ with this non-standard metric by $\mathfrak{R}^3$; however, note that the metric need not to be product. By an abuse of notation, we denote both metrics on $M$ and also the induced metric on $\mathfrak{R}^3$ by $g$. The case where the metric is flat has been investigated in the study of periodic singular monopoles by Foscolo \cite{MR3801425}. 

Working with an arbitrary metric introduces two potential sources of difficulty. One is related to the error of the approximate solution, since the BPS-monopole is not a genuine monopole with respect to the arbitrary metric, as we observed earlier. The other difficulty is related to the Ricci terms in the monopole Wietzenb\"ock formulas in Lemma \ref{monopoleweitzenbock}, which appear in the estimations in the linear problem.

Let $q \in S_q = \{q_1,  \hdots , q_k\}$. Let $(A_0, \varPhi_0)$ be the approximate monopole defined on $B_{3 \epsilon}(q)$ by gluing the pull-back of the scaled BPS-monopole on $\mathbb{R}^3$ to the scaled Dirac monopole, both with mass $\lambda$ and centered at the origin, using a cut-off function $\xi$,
\begin{align}\label{modmon}
    (A_0, \varPhi_0)(x) = 
    \begin{cases}
        (\eta^* (A_{BPS}^{\lambda}), \eta^*(\varPhi_{BPS}^{\lambda}))(x) \quad &r \leq \lambda^{-\frac{1}{2}}, \\
        (A_{D}^{\lambda}, \varPhi_{D}^{\lambda})(x) \quad &r \geq 2\lambda^{-\frac{1}{2}},
    \end{cases}
\end{align}
where $r$ denotes the geodesic distance to the origin and $|\nabla \xi| \leq \lambda^{\frac{1}{2}}$. 

Using the diffeomorphism $\mu^{-1}$ one can pull back the pair $(A_0, \varPhi_0)$ to $U \subset \mathfrak{R}^3$. We can extend this pair $((\mu^{-1})^*(A_0), (\mu^{-1})^*\varPhi_0)$ to a pair on $\mathfrak{R}^3$, by gluing it to the standard Dirac monopole scaled by the factor $\lambda$ on $\mathfrak{R}^3 \setminus V = \mathbb{R}^3 \setminus V$, using a cut-off function $\tilde{\xi}$ such that $|\nabla \tilde{\xi}| \leq \lambda^{\frac{1}{2}}$. By an abuse of notation, we still denote this pair on $\mathfrak{R}^3$ by $(A_0, \varPhi_0)$. 

The first step is to set up the suitable Sobolev spaces for the linear problem. In order to solve \ref{lin} on $\mathfrak{R}^3$, naively, one might let $d_2d_2^*: W^{2,2}(\Omega^1(\mathfrak{R}^3)) \to L^2(\Omega^1(\mathfrak{R}^3))$; however, this is not a suitable choice, and one needs to use the weighted Sobolev spaces. 

\subsection[Function Spaces on $\mathfrak{R}^3$]{Function Spaces on $\pmb{\mathfrak{R}^3}$}\label{FunctionSpacesonR3}

In this section, we introduce the appropriate weighted Sobolev spaces to study the linear operator $d_2d_2^*$ on $\mathfrak{su}(2)$-valued 1-forms on $\mathfrak{R}^3$, for a pair $(A_0, \varPhi_0)$ with a large mass. These weighted spaces have been  investigated by Biquard \cite{MR1116847, MR1168354}, and used in the case of monopoles by Foscolo \cite{MR3439230, MR3801425}. They are also related to the weighted norms in \cite{MR1309165}.

The weights used in the definition of our function spaces are designed for the linear operator $d_2d_2^*$ to have a bounded right-inverse. The motivation for the specific choices of the weights comes from the observation that the terms $F_{A_0}$ and $d_{A_0}\varPhi_0$ appearing in the monopole Weitzenb\"ock formula of $d_2d_2^*$ blow up as the scaling factor $\lambda \to \infty$ and $r \to 0$.

\begin{definition} \label{weightpointq}
Let 
\begin{align*}
       w(x) = 
    \begin{cases}
        \sqrt{\lambda^{-2} + r^2}, \quad & r\leq \frac{1}{2}\\
        1, \quad & r \geq 1,
    \end{cases} 
\end{align*}
where $r: \mathfrak{R}^3 \to \mathbb{R}^{\geq 0}$ denotes the geodesic distance from the origin.

Let $(A_0, \varPhi_0)$ be the approximate monopole on $\mathfrak{R}^3$, as we constructed earlier. Let $\alpha \in \mathbb{R}$. For all smooth compactly supported $\mathfrak{su}(2)$-valued differential forms $ u \in \Omega_c^{\boldsymbol{\cdot}} (\mathfrak{R}^3, \mathfrak{su}(2))$, let
\begin{align*}
    &\| u \|_{L^2_{\alpha}(\mathfrak{R}^3)} = \| w^{- \alpha - \frac{3}{2}} u \|_{L^2(\mathfrak{R}^3)},\\
    &\| u \|^2_{W^{1,2}_{\alpha}(\mathfrak{R}^3)} = 
    \| u \|_{L^2_{\alpha}(\mathfrak{R}^3)}^2 + \| \nabla_{A_0} u \|_{L^2_{\alpha - 1}(\mathfrak{R}^3)}^2 + \| [\varPhi_0,u] \|_{L^2_{\alpha-1}(\mathfrak{R}^3)}^2. 
\end{align*}
Let the spaces $L^2_{\alpha}(\mathfrak{R}^3)$ and $W^{1,2}_{\alpha}(\mathfrak{R}^3)$ be the completion of $C^{\infty}_0(\mathfrak{R}^3)$ with respect to these norms. Furthermore,
\begin{align*}
    \|u\|_{W^{2,2}_{\alpha}(\mathfrak{R}^3)}^2 := \| u \|^2_{W^{1,2}_{\alpha}(\mathfrak{R}^3)} +
    \|\nabla_{A_0}(d_2^*u)\|^2_{L^2_{\alpha-2}(\mathfrak{R}^3)} + \|[\varPhi_0, d_2^*u]\|^2_{L^2_{\alpha-2}(\mathfrak{R}^3)}.
\end{align*}
More generally, for any $p>1$, we can define the norm
\begin{align*}
    \| u \|_{L^p_{\alpha}(\mathfrak{R}^3)} = \|w^{- \alpha - \frac{3}{p}} u \|_{L^p(\mathfrak{R}^3)},
\end{align*}
and $L^p_{\alpha}(\mathfrak{R}^3)$ as the completion of $C^{\infty}_0(\mathfrak{R}^3)$ with respect to this norm. 
\end{definition}

These spaces satisfy similar properties as the ordinary Sobolev spaces.

\begin{lemma}\label{D1}
The weighted Sobolev spaces enjoy the following properties:
\begin{itemize}
    \item Let $k \in \{0, 1, 2\}$. Let $u \in W^{k,p}_{\alpha,loc}(\mathfrak{R}^3)$. Suppose $W^{k,p}_{\alpha}(\mathfrak{R}^3)$-norm of $u$ converges. Then $u \in W^{k,p}_{\alpha}(\mathfrak{R}^3)$.

    \item For any 
    $u \in W^{2,2}_{\alpha}(\mathfrak{R}^3)$,
    we have $d_2 d_2^* u \in L^2_{\alpha-2}(\mathfrak{R}^3)$, when $\lambda$ is sufficiently large.
    
    \item $C^{\infty}_c(\mathfrak{R}^3)$ is dense in $W^{1,2}_{\alpha}(\mathfrak{R}^3)$, and therefore, Lemma \ref{Hardy} holds for the elements of $W^{1,2}_{\alpha}(\mathfrak{R}^3)$.
\end{itemize}
\end{lemma}

The proofs are straightforward. 

\subsection[Solving the Linear Equation on $\mathfrak{R}^3$]{Solving the Linear Equation on $\pmb{\mathfrak{R}^3}$}\label{SolvingtheLinearEquationonR3}

The main theorem of this section is the following. In the case where the Riemannian metric $g$ is the Euclidean metric, this is proposition 5.8. in \cite{MR3801425}.

\begin{theorem} \label{q} 
Let $d_2d_2^*: W^{2,2}_{\alpha}(\Omega^1(\mathfrak{R}^3,\mathfrak{su}(2))) \to L^2_{\alpha-2}(\Omega^1(\mathfrak{R}^3,\mathfrak{su}(2)))$. 
For all $-\frac{1}{2} \leq \alpha < 0$ there exist $\delta > 0$  such that if $\| w e_0(A_0 \varPhi_0) \|_{L^3(\mathfrak{R}^3)} < \delta$, then $d_2d_2^*$ is invertible. For $\delta > 0$ sufficiently small, there exists $ C > 0 $ such that for all $f \in L^2_{\alpha-2}(\Omega^1(\mathfrak{R}^3,\mathfrak{su}(2)))$,
there exists a unique solution $ u \in W^{2,2}_{\alpha}(\Omega^1(\mathfrak{R}^3,\mathfrak{su}(2)))$ to $d_2d_2^* u = f$ with 
\begin{align*}
    \| u \|_{W^{2,2}_{\alpha}(\mathfrak{R}^3)}
    \leq C 
    \| f \|_{L^2_{\alpha-2}(\mathfrak{R}^3)},
\end{align*}
where the constant $C$ is independent of $\lambda$, which appears in the definition of the approximate monopole $(A_0, \varPhi_0)$.
\end{theorem}

The proof is based on a direct variational method. We present the proof of Theorem \ref{q} in 10 steps, presented in a sequence of lemmas. The line of the proof follows \cite{MR3801425}. Since $C^{\infty}_{c}(\mathfrak{R}^3)$ is dense in $W^{2,2}_{\alpha}$, we only need to prove the theorem when $f$ is a smooth compactly supported $\mathfrak{su}(2)$-valued 1-form on $\mathfrak{R}^3$.

Before stating the proof, we want to assume a normalizing condition for the Riemannian metric $g$ on $M$. Note that $(A,\varPhi)$ is a monopole on $(M,g)$ if and only if $(A,\frac{1}{c} \varPhi)$ is a monopole on $(M,c^2g)$ for any positive constant $c$, and therefore, there is a one-to-one correspondence between monopoles on $(M,g)$ and monopoles on $(M,c^2g)$. Therefore, without loss of generality, by multiplying the metric $g$ by a sufficiently small positive constant number $c^2$, we can assume $\sup_{x \in M}|Ric(x)| < \frac{1}{100}$, and therefore, $\sup_{x \in \mathfrak{R}^3}|Ric(x)| < \frac{1}{100}$.

\begin{lemma} [Step 1] \label{s1}
Suppose $f$ is a smooth, compactly supported, $\mathfrak{su}(2)$-valued 1-form on $\mathfrak{R}^3$. Let $\alpha \in [-\frac{1}{2},0)$. Let 
\begin{align} \label{E2}
    E: W^{1,2}_{\alpha}(\Omega^1(\mathfrak{R}^3,\mathfrak{su}(2))) \to \mathbb{R}, \quad \quad     E(u) := \frac{1}{2}\int_{\mathfrak{R}^3} |d_2^*u|^2 vol_g - \langle u , f \rangle_{L^2(\mathfrak{R}^3)}.
\end{align}
The functional $E(u)$ is convergent when $u \in W^{1,2}_{\alpha}(\Omega^1(\mathfrak{R}^3,\mathfrak{su}(2)))$. 
\end{lemma}

\begin{proof}
We first show $\|d_2^*u\|^2_{L^2(\mathfrak{R}^3)}$ is finite. The key fact in proving this is the monopole Weitzenb\"ock formula. 
\begin{align} \label{22}
    \| d_2^* u \|_{L^2(\mathfrak{R}^3)}^2 = \| \nabla_{A_0} u \|_{L^2(\mathfrak{R}^3)}^2 + \| [\varPhi_0, u]\|_{L^2(\mathfrak{R}^3)}^2 &+ \langle u, Ric(u)\rangle_{L^2(\mathfrak{R}^3)} \nonumber  \\&+ \langle *[e_0 \wedge u],u \rangle_{L^2(\mathfrak{R}^3)}.
\end{align}
Note that since $u \in W^{1,2}_{\alpha}(\Omega^1(\mathfrak{R}^3,\mathfrak{su}(2)))$, the asymptotic terms do not appear in the formula above. 

The first two terms on the right hand side of \ref{22} are finite, since $ 1 \leq w^{- \alpha - \frac{1}{2}}$ when $-\frac{1}{2} \leq \alpha $, and therefore,
\begin{align*}
    &\| \nabla_{A_0} u \|_{L^2(\mathfrak{R}^3)}^2 \leq 
    \| w^{- \alpha - \frac{1}{2}} \nabla_{A_0} u \|_{L^2(\mathfrak{R}^3)}^2 =
    \| \nabla_{A_0} u \|_{L^2_{\alpha - 1}(\mathfrak{R}^3)}^2 \leq \|u \|_{W^{1,2}_{\alpha}(\mathfrak{R}^3)}^2 < \infty, \\
    &\| [\varPhi_0, u]\|_{L^2(\mathfrak{R}^3)}^2 \leq  
    \| w^{- \alpha - \frac{1}{2}}[\varPhi_0, u]\|_{L^2(\mathfrak{R}^3)}^2 =
    \| [\varPhi_0, u] \|_{L^2_{{\alpha - 1}}(\mathfrak{R}^3)}^2  \leq \|u \|_{W^{1,2}_{{\alpha}}(\mathfrak{R}^3)}^2 < \infty.
\end{align*}

Regarding the Ricci term, since $\mathfrak{R}^3$ is flat outside of a compact subset, $\sup_{\mathfrak{R}^3}|Ric|$ is finite, and therefore,
\begin{align*}
    \langle u, Ric(u)\rangle_{L^2(\mathfrak{R}^3)}
    \leq \sup_{\mathfrak{R}^3}|Ric| \|u\|^2_{L^2(\mathfrak{R}^3)} \leq \sup_{\mathfrak{R}^3}|Ric| \|u\|^2_{W^{1,2}_{\alpha}(\mathfrak{R}^3)} < \infty.
\end{align*}
As for the error term in \ref{22}, by applying the H\"older's inequality twice, we get
\begin{align*}
    |\langle *[e_0 &\wedge u],u \rangle_{L^2(\mathfrak{R}^3)}| \\&= |\langle *[w e_0 \wedge u],w^{-1}u \rangle_{L^2(\mathfrak{R}^3)}|
    \leq \|[w e_0 \wedge u] \|_{L^2(\mathfrak{R}^3)} \|w^{-1}  u \|_{L^2(\mathfrak{R}^3)}  
    \\&\leq 2 \| we\|_{L^3(\mathfrak{R}^3)} \|u \|_{L^6(\mathfrak{R}^3)}
    \|w^{-1}  u \|_{L^2(\mathfrak{R}^3)}
    \leq 2 \delta \|u \|_{L^6(\mathfrak{R}^3)}
    \|w^{-1}  u \|_{L^2(\mathfrak{R}^3)}.
\end{align*}
Regarding the term $\|u \|_{L^6(\mathfrak{R}^3)}$, by the Sobolev inequality we have $\| u \|_{L^6(\mathfrak{R}^3)} \leq C_{Sob}\| u \|_{W^{1,2}(\mathfrak{R}^3)}$, where $C_{Sob}$ is independent of $\lambda$, and since $0 < w \leq 1$, on $\mathfrak{R}^3 \setminus \{ 0 \}$,
\begin{align*}
    \|u\|^2_{W^{1,2}(\mathfrak{R}^3)} &= 
    \|u\|^2_{L^2(\mathfrak{R}^3)} + 
    \|\nabla_{A_0} u\|^2_{L^2(\mathfrak{R}^3)} 
    \\&\leq
    \|w^{-\alpha - \frac{3}{2}}u\|^2_{L^2(\mathfrak{R}^3)} +
    \|w^{-\alpha - \frac{1}{2}}\nabla_{A_0} u\|^2_{L^2(\mathfrak{R}^3)}
    \leq  \|u\|^2_{W^{1,2}_{\alpha}(\mathfrak{R}^3)}
    < \infty.
\end{align*}
Regarding the term $\|w^{-1}  u \|_{L^2(\mathfrak{R}^3)}$, following Lemma \ref{Hardy}, we have
\begin{align*}
    \|w^{-1}  u \|_{L^2(\mathfrak{R}^3)} &\leq C \| \nabla_{A_0} u \|_{L^2(\mathfrak{R}^3)} 
    \leq C \| w^{-\alpha - \frac{1}{2}}\nabla_{A_0} u \|_{L^2(\mathfrak{R}^3)}
    = C \| \nabla_{A_0} u \|_{L^2_{\alpha-1}(\mathfrak{R}^3)}
    \\ &\leq C \| u \|_{W^{1,2}_{\alpha}} < \infty,
\end{align*}
for a constant $C$ which depends on $\alpha$ and the metric on $\mathfrak{R}^3$, and therefore,
\begin{align*}
    |\langle *[e_0 \wedge u],u \rangle_{L^2(\mathfrak{R}^3)}| \leq 2 \delta C C_{Sob} \|u\|^2_{W^{1,2}_{\alpha}(\mathfrak{R}^3)}
    \leq C_1 \|u\|^2_{W^{1,2}_{\alpha}(\mathfrak{R}^3)},
\end{align*}
for a constant $C_1 = 2 \delta C C_{Sob}$, and therefore,
\begin{align*}
    \|d_2^*u\|_{L^2(\mathfrak{R}^3)}^2 &\leq \| \nabla_{A_0} u \|_{L^2(\mathfrak{R}^3)}^2 + \| [\varPhi_0, u]\|_{L^2(\mathfrak{R}^3)}^2 + C_1 \|u\|^2_{W^{1,2}_{\alpha}(\mathfrak{R}^3)} + \sup_{\mathfrak{R}^3} |Ric| \|u\|^2_{W^{1,2}_{\alpha}(\mathfrak{R}^3)}
    \\& \leq 
    (2 +  C_1 + \sup_{\mathfrak{R}^3} |Ric|) \|u\|^2_{W^{1,2}_{\alpha}(\mathfrak{R}^3)} < \infty.
\end{align*}

After observing that $\|d_2^*u\|_{L^2(\mathfrak{R}^3)}^2$ is convergent, we should prove the same for 
$\langle u , f \rangle_{L^2(\mathfrak{R}^3)}$. By the Cauchy–Schwarz inequality we have
\begin{align*}
    \langle u,f \rangle_{L^2(\mathfrak{R}^3)} \leq \|u\|_{L^2(\mathfrak{R}^3)} \|f\|_{L^2(\mathfrak{R}^3)} \leq \|u\|_{L^2_{\alpha}(\mathfrak{R}^3)} \|f\|_{L^2(\mathfrak{R}^3)}
    \leq \|u\|_{W^{1,2}_{\alpha}(\mathfrak{R}^3)}
    \|f\|_{L^2(\mathfrak{R}^3)} < \infty,
\end{align*}
and therefore, we have a well-defined action functional $E: W^{1,2}_{\alpha}(\Omega^1(\mathfrak{R}^3,\mathfrak{su}(2))) \to \mathbb{R}$. 
\end{proof}

\begin{lemma}[Step 2]
Let $f$ be a smooth, compactly supported, $\mathfrak{su}(2)$-valued 1-form on $\mathfrak{R}^3$. Let $\alpha \in [-\frac{1}{2},0)$. The functional $E: W^{1,2}_{\alpha}(\Omega^1(\mathfrak{R}^3,\mathfrak{su}(2))) \to \mathbb{R}$ is continuous. 
\end{lemma}

\begin{proof}
We should show the following functions are continuous, 
\begin{align*}
\|d_2^*-\|^2_{L^2(\mathfrak{R}^3)}:
W^{1,2}_{\alpha}(\Omega^1(\mathfrak{R}^3,\mathfrak{su}(2))) \to \mathbb{R}, 
\quad
\langle -,f \rangle_{L^2(\mathfrak{R}^3)}:
W^{1,2}_{\alpha}(\Omega^1(\mathfrak{R}^3,\mathfrak{su}(2))) \to \mathbb{R}. 
\end{align*}
To prove $\|d_2^*-\|^2_{L^2(\mathfrak{R}^3)}$ is continuous we should show 
\begin{align} \label{conn}
\|d_2^* u_i\|^2_{L^2(\mathfrak{R}^3)} \to \|d_2^*u\|^2_{L^2(\mathfrak{R}^3)}, \quad \text{ when } \quad u_i \to u \quad \text{ in } \quad W^{1,2}_{\alpha}(\Omega^1(\mathfrak{R}^3,\mathfrak{su}(2))).
\end{align} 
To show this we use the monopole Weitzenb\"ock formula. It follows directly from the definition of $W^{1,2}_{\alpha}(\Omega^1(\mathfrak{R}^3,\mathfrak{su}(2)))$ that 
\begin{align*}
    \|\nabla_{A_0} u_i \|_{L^2(\mathfrak{R}^3)}^2 + \| [\varPhi_0, u_i]\|_{L^2(\mathfrak{R}^3)}^2 \to 
    \|\nabla_{A_0} u \|_{L^2(\mathfrak{R}^3)}^2 + \| [\varPhi_0, u]\|_{L^2(\mathfrak{R}^3)}^2,
\end{align*}
when $ u_i \to u$ in $W^{1,2}_{\alpha}(\Omega^1(\mathfrak{R}^3,\mathfrak{su}(2)))$.

Regarding the Ricci term in the monopole Weitzenb\"ock formula,
\begin{align*}
    |\langle & Ric(u),u \rangle_{L^2(\mathbb{R}^3)} 
    -
    \langle Ric(u_i),u_i \rangle_{L^2(\mathbb{R}^3)}|
    \\ &=
    |\langle w Ric(u), w^{-1} u \rangle_{L^2(\mathbb{R}^3)}
    -
    \langle w Ric(u_i), w^{-1} u_i \rangle_{L^2(\mathbb{R}^3)}|
    \\& \leq
    |\langle w Ric(u-u_i), w^{-1} u \rangle_{L^2(\mathbb{R}^3)}|
    +
    |\langle w Ric(u_i), w^{-1} (u-u_i) \rangle_{L^2(\mathbb{R}^3)}|
    \\ &\leq  \|w Ric\|_{L^3} \|u-u_i \|_{L^6} \|w^{-1}u\|_{L^2} + 
    \|w Ric\|_{L^3} \|u _i\|_{L^6} \|w^{-1}(u-u_i)\|_{L^2}
    \\ &\leq  C  \left(
    \|w Ric\|_{L^3} \| u-u_i \|_{W^{1,2}} \|\nabla_{A_0} u\|_{L^2} + 
    \|w Ric\|_{L^3} (\|u\|_{W^{1,2}}+1) \|u- u_i\|_{W^{1,2}}\right),
\end{align*}
which goes to zero as $i \to \infty$.

Furthermore, when $i$ is sufficiently large,
\begin{align*}
    | 
    (
    |\langle & *[e_0 \wedge u],u \rangle_{L^2(\mathfrak{R}^3)}| 
    -
    |\langle *[e_0 \wedge u_i],u_i \rangle_{L^2(\mathfrak{R}^3)}|
    )
    |
    \\&=
    |(|\langle *[we_0 \wedge u],w^{-1}u \rangle_{L^2(\mathfrak{R}^3)}| 
    -
    |\langle *[we_0 \wedge u_i],w^{-1}u_i \rangle_{L^2(\mathfrak{R}^3)}|)|
    \\& \leq
    |\langle *[we_0 \wedge u],w^{-1}(u - u_i) \rangle_{L^2(\mathfrak{R}^3)}|
    +
    |\langle *[we_0 \wedge (u - u_i)],w^{-1}u_i \rangle_{L^2(\mathfrak{R}^3)}|
    \\ &\leq 
    \|[we_0 \wedge u]\|_{L^2(\mathfrak{R}^3)}
    \|w^{-1}(u - u_i)\|_{L^2(\mathfrak{R}^3)}
    +
    \|[we_0 \wedge (u - u_i)]\|_{L^2(\mathfrak{R}^3)}
    \|w^{-1}u_i\|_{L^2(\mathfrak{R}^3)}
    \\ &\leq 
    2 \delta \|u \|_{L^6(\mathfrak{R}^3)}
    \|w^{-1}  (u-u_i) \|_{L^2(\mathfrak{R}^3)}
    +
    2 \delta \|u-u_i \|_{L^6(\mathfrak{R}^3)}
    \|w^{-1}  u_i \|_{L^2(\mathfrak{R}^3)}
    \\ &\leq
    2 \delta \|u \|_{W^{1,2}(\mathfrak{R}^3)}
    \|w^{-1}  (u-u_i) \|_{L^2(\mathfrak{R}^3)}
    +
    2 \delta \|u-u_i \|_{W^{1,2}(\mathfrak{R}^3)}
    \|w^{-1}  u_i \|_{L^2(\mathfrak{R}^3)}
    \\ &\leq
    2 \delta \|u \|_{W^{1,2}_{\alpha}(\mathfrak{R}^3)}
    \|u-u_i\|_{W^{1,2}_{\alpha}(\mathfrak{R}^3)}
    +
    2 \delta \|u-u_i \|_{W^{1,2}_{\alpha}(\mathfrak{R}^3)}
    (\| u \|_{W^{1,2}_{\alpha}(\mathfrak{R}^3)}+1),
\end{align*}
which converges to 0 as $i \to \infty$, and therefore, by the monopole Weitzenb\"ock formula \ref{conn} follows.

In order to prove $\langle -,f \rangle_{L^2(\mathfrak{R}^3)}:
W^{1,2}_{\alpha}(\Omega^1(\mathfrak{R}^3,\mathfrak{su}(2))) \to \mathbb{R}$ is continuous, note that this map is linear, and since $W^{1,2}_{\alpha}(\Omega^1(\mathfrak{R}^3,\mathfrak{su}(2)))$ is a Hilbert space, continuity is equivalent to being bounded. 
\begin{align*}
    \langle u , f \rangle_{L^2(\mathfrak{R}^3)} \leq
        \|u\|_{L^2(\mathfrak{R}^3)} \|f\|_{L^2(\mathfrak{R}^3)} \leq
    \|u\|_{W^{1,2}_{\alpha}(\mathfrak{R}^3)} \|f\|_{L^2(\mathfrak{R}^3)},
\end{align*}
and therefore, the operator norm 
\begin{align*}
    \sup \{\langle u , f \rangle_{L^2(\mathfrak{R}^3)} \; | \; {u \in W^{1,2}_{\alpha}, \|u\|_{W^{1,2}_{\alpha}}  = 1} \} \leq 
     \|f\|_{L^2(\mathfrak{R}^3)} < \infty,
\end{align*}
which proves the linear map is bounded and continuous. 
\end{proof}

\begin{lemma}[Step 3]
Let $f$ be a smooth, compactly supported, $\mathfrak{su}(2)$-valued 1-form on $\mathfrak{R}^3$. Let $\alpha \in [-\frac{1}{2},0)$. $E: W^{1,2}_{\alpha}(\Omega^1(\mathfrak{R}^3,\mathfrak{su}(2))) \to \mathbb{R}$ is Gateaux-differentiable.
\end{lemma}

\begin{proof}
A direct computation shows
\begin{align*}
    d_u E(v) = \int_{\mathfrak{R}^3} \langle d_2^*u , d_2^*v \rangle_{L^2} vol_g - \langle v , f \rangle_{L^2(\mathfrak{R}^3)}.
\end{align*}
Note that for any $u ,v \in W^{1,2}_{\alpha}$,
\begin{align*}
    \int_{\mathfrak{R}^3} \langle d_2^*u , d_2^*v \rangle_{L^2} vol_g \leq \|d_2^*u\|_{L^2(\mathfrak{R}^3)} \| d_2^*v\|_{L^2(\mathfrak{R}^3)} < \infty, 
\end{align*}
and 
\begin{align*}
    \langle v , f \rangle_{L^2(\mathfrak{R}^3)} \leq \|v\|_{L^2(\mathfrak{R}^3)} \| f\|_{L^2(\mathfrak{R}^3)} < \infty.
\end{align*}
\end{proof}

\begin{lemma}[Step 4]\label{itisnorm} Suppose $f$ is a smooth, compactly supported, $\mathfrak{su}(2)$-valued 1-form on $\mathfrak{R}^3$. Let $\alpha \in [-\frac{1}{2},0)$. There exists $\delta>0$ such that if $\|we_0\|_{L^3(\mathfrak{R}^3)}<\delta$, then the functional\begin{align}
    E: W^{1,2}_{\alpha}(\Omega^1(\mathfrak{R}^3,\mathfrak{su}(2))) \to \mathbb{R}, \quad \quad     E(u) := \frac{1}{2}\int_{\mathfrak{R}^3} |d_2^*u|^2 vol_g - \langle u , f \rangle_{L^2(\mathfrak{R}^3)},
\end{align} 
is strictly convex. 
\end{lemma}

\begin{proof}
Since $\langle u,f \rangle_{L^2(\mathfrak{R}^3)}$ is linear in $u$, we only need to show $\|d_2^*u\|_{L^2(\mathfrak{R}^3)}^2$ is strictly convex. Let $u, v \in W^{1,2}_{\alpha}(\Omega^1(\mathfrak{R}^3,\mathfrak{su}(2)))$ where $u_1 \neq u_2$. Let $t \in (0,1)$. We should prove
\begin{align*}
    \|d_2^*(tu_1+(1-t)u_2)\|_{L^2(\mathfrak{R}^3)}^2 <
    t\|d_2^*u_1\|_{L^2(\mathfrak{R}^3)}^2 + (1-t)\|d_2^*u_2\|_{L^2(\mathfrak{R}^3)}^2.
\end{align*}
In fact, we only need to check this for $t = \frac{1}{2}$. The strict inequality 
\begin{align*}
\|d_2^*(\frac{u_1+u_2}{2})\|_{L^2(\mathfrak{R}^3)}^2 <
    \frac{1}{2}\|d_2^*u_1\|_{L^2(\mathfrak{R}^3)}^2 + \frac{1}{2}\|d_2^*u_2\|_{L^2(\mathfrak{R}^3)}^2,    
\end{align*}
is equivalent to $\|d_2^*(u_1-u_2)\|_{L^2(\mathfrak{R}^3)}^2 > 0$. Let $u = u_1 - u_2$. We should show 
\begin{align} \label{n}
    u \in W^{1,2}_{\alpha},\quad d_2^*u = 0\quad \Rightarrow \quad u = 0. 
\end{align}
The property \ref{n} implies $\| d_2^* - \|_{L^2}$ is a norm on $W^{1,2}_{\alpha}$. The key fact to prove this is the Gagliardo–Nirenberg–Sobolev inequality. Suppose $\| d_2^* u \|_{L^2(\mathfrak{R}^3)} =0$,
\begin{align*}
0 = \| d_2^* u \|_{L^2(\mathfrak{R}^3)}^2 = \| \nabla_{A_0} u \|_{L^2(\mathfrak{R}^3)}^2 + \| [\varPhi_0, u]\|_{L^2(\mathfrak{R}^3)}^2 &+ \langle Ric(u),u \rangle_{L^2(\mathfrak{R}^3)} \\&+ \langle *[e_0 \wedge u],u \rangle_{L^2(\mathfrak{R}^3)}.  
\end{align*} 
Using Lemma \ref{Hardy}, we have
\begin{align*}
    \langle u, Ric(u)\rangle_{L^2(\mathfrak{R}^3)} &=
    \langle w^{-1} u, w Ric(u)\rangle_{L^2(\mathfrak{R}^3)} \leq \|w^{-1} u\|_{L^2(\mathfrak{R}^3)}
    \|w Ric(u)\|_{L^2(\mathfrak{R}^3)}
    \\ &\leq  5
    \|\nabla_{A_0} u\|_{L^2(\mathfrak{R}^3)}
    \|w Ric\|_{L^3} \|u\|_{L^6(\mathfrak{R}^3)} \leq \frac{1}{4}
    \|\nabla_{A_0} u\|^2_{L^2(\mathfrak{R}^3)},
\end{align*}
for $\epsilon$ small enough such that $\|w Ric\|_{L^3} < \frac{1}{20 C_{Sob}}$, which can be arranged when $\lambda$ is sufficiently large, since $Ric(x)=0$ outside of $B_0(4\epsilon)$. Moreover, in the inequality
\begin{align*}
    \|w^{-1} u\|_{L^2(\mathfrak{R}^3)} \leq
    C_{g, \alpha}
    \|\nabla_{A_0} u\|_{L^2(\mathfrak{R}^3)}.
\end{align*}
by taking $\epsilon > 0$ small enough, we can take $C_{g,\alpha} < 5$. In fact, when the metric $g$ on $\mathfrak{R}^3$ is flat, this constant is $1/(1+\alpha)^2$. In our case, the metric on  $\mathfrak{R}^3$ coincide with the flat metric outside of a small ball $B_{4 \epsilon}(0)$. By taking the $\epsilon > 0$ sufficiently small, we can take the constant $C_{g, \alpha}$ to be sufficiently close to $1/(\alpha+1)^2$, and therefore, less than $5$.

Furthermore,
\begin{align*}
    |\langle *[we_0 \wedge u], w^{-1} u \rangle_{L^2(\mathfrak{R}^3)}| &\leq  2 \|w e\|_{L^3(\mathfrak{R}^3)} \| u \|_{L^6(\mathfrak{R}^3)}
    \|w^{-1} u \|_{L^2(\mathfrak{R}^3)} \\&\leq 
    2C  \| we\|_{L^3(\mathfrak{R}^3)} 
    \|\nabla_{A_0} u \|_{L^2(\mathfrak{R}^3)}^2.
\end{align*}
Pick $\delta < \frac{1}{8C}$, 
\begin{align*}
    0 &= \| d_2^* u \|_{L^2(\mathfrak{R}^3)}^2 \\ &= \| \nabla_{A_0} u \|_{L^2(\mathfrak{R}^3)}^2 + \| [\varPhi_0, u]\|_{L^2(\mathfrak{R}^3)}^2 + \langle Ric(u),u \rangle_{L^2(\mathfrak{R}^3)}+ \langle *[e_0 \wedge u],u \rangle_{L^2(\mathfrak{R}^3)} \\&\geq \| \nabla_{A_0} u \|_{L^2(\mathfrak{R}^3)}^2 + \| [\varPhi_0, u]\|_{L^2(\mathfrak{R}^3)}^2 - \frac{1}{4}
    \|\nabla_{A_0} u\|^2_{L^2(\mathfrak{R}^3)}  -
    \frac{1}{4}\|\nabla_{A_0} u \|_{L^2(\mathfrak{R}^3)}^2
    \\&=
    \frac{1}{2}\| \nabla_{A_0} u \|_{L^2(\mathfrak{R}^3)}^2 + \| [\varPhi_0, u]\|_{L^2(\mathfrak{R}^3)}^2,
\end{align*}
and therefore, $\nabla_{A_0} u = 0 = [\varPhi_0,u]$. Therefore, $u$ is a covariantly constant section in $W^{1,2}_{\alpha}$, hence, $u=0$. 
\end{proof}

\begin{lemma}[Step 5]\label{minsol}
Let $f$ be a smooth, compactly supported, $\mathfrak{su}(2)$-valued 1-form on $\mathfrak{R}^3$. Let $\alpha \in [-\frac{1}{2},0)$. The action functional $E: W^{1,2}_{\alpha}(\Omega^1(\mathfrak{R}^3)) \to \mathbb{R}$ has a unique minimizer, and therefore, $d_2d_2^* u = f$ has a unique solution. 
\end{lemma}

\begin{proof}
The proof of this lemma is based on the following fact.

Let $E:W \to \mathbb{R}$ be a convex, continuous, and real Gateaux-differentiable functional defined on a real reflexive Banach space $W$ such that $d_uE(u) > 0$, for any $u \in W$ with $\|u\|_{W} \geq R > 0$, for a positive constant $R$. Then there exists an interior point $u_0$ of $\{u \in W \; | \; \|u\|_{W} < R \}$ which is the unique minimizer of $E$ and $d_{u_0} E = 0$.

Let $W = W^{1,2}_{\alpha}(\Omega^1(\mathfrak{R}^3))$ but equipped with the norm $\|d^* -\|_{L^2(\mathfrak{R}^3)}$ --- it follows from the proof of Lemma \ref{itisnorm} that this is a norm. Moreover, similar to the Sobolev space $W^{1,2}(\Omega^1(\mathfrak{R}^3))$, it is straightforward to see $ W^{1,2}_{\alpha}(\Omega^1(\mathfrak{R}^3))$ is reflexive too. 

We should show there is a constant $R>0$ such that we have $d_uE(u) > 0$ for any $u \in W$ with $\|u\|_{W} \geq R > 0$. As in the proof of Lemma \ref{itisnorm},
\begin{align*}
    \| d_2^* u \|_{L^2(\mathfrak{R}^3)}^2 &= \| \nabla_{A_0} u \|_{L^2(\mathfrak{R}^3)}^2 + \| [\varPhi_0, u]\|_{L^2(\mathfrak{R}^3)}^2 + \langle w Ric(u), w^{-1} u \rangle_{L^2(\mathfrak{R}^3)}+ \langle *[e_0 \wedge u],u \rangle_{L^2(\mathfrak{R}^3)} \\&\geq \| \nabla_{A_0} u \|_{L^2(\mathfrak{R}^3)}^2 + \| [\varPhi_0, u]\|_{L^2(\mathfrak{R}^3)}^2 - \sup_{x\in\mathfrak{R}^3}Ric(x)\|w^{-1} u\|^2_{L^2(\mathfrak{R}^3)}  -
    \frac{1}{4}\|\nabla_{A_0} u \|_{L^2(\mathfrak{R}^3)}^2 \\&\geq
    \| \nabla_{A_0} u \|_{L^2(\mathfrak{R}^3)}^2 + \| [\varPhi_0, u]\|_{L^2(\mathfrak{R}^3)}^2 - \frac{1}{4}\|\nabla_{A_0} u \|_{L^2(\mathfrak{R}^3)}^2  -
    \frac{1}{4}\|\nabla_{A_0} u \|_{L^2(\mathfrak{R}^3)}^2 
    \\&=
    \frac{1}{2}\| \nabla_{A_0} u \|_{L^2(\mathfrak{R}^3)}^2 + \| [\varPhi_0, u]\|_{L^2(\mathfrak{R}^3)}^2 
    \geq \frac{1}{2}\| \nabla_{A_0} u \|_{L^2(\mathfrak{R}^3)}^2,
\end{align*}
and therefore,
\begin{align*}
    \|d_2^*u\|^2_{L^2(\mathfrak{R}^3)} \geq \frac{1}{2}\|\nabla_{A_0}u\|^2_{L^2(\mathfrak{R}^3)}.
\end{align*}
This shows
\begin{align*}
    d_u E(u) &= 
    \| d_2^*u \|^2_{L^2(\mathfrak{R}^3)} - \langle u , f \rangle_{L^2(\mathfrak{R}^3)} = 
    \| d_2^*u \|^2_{L^2(\mathfrak{R}^3)} - \langle w^{-1} u , wf \rangle_{L^2(\mathfrak{R}^3)}
    \\ &\geq \| d_2^*u \|^2_{L^2(\mathfrak{R}^3)} -
    \|w^{-1} u\|_{L^2(\mathfrak{R}^3)} \| w f\|_{L^2(\mathfrak{R}^3)}\\
     &\geq \| d_2^*u \|^2_{L^2(\mathfrak{R}^3)} - C
    \|\nabla_{A_0} u\|_{L^2(\mathfrak{R}^3)} \| f\|_{L^2(\mathfrak{R}^3)} 
    \\ &\geq \| d_2^*u \|^2_{L^2(\mathfrak{R}^3)} - \sqrt{2} C
    \|d_2^* u\|_{L^2(\mathfrak{R}^3)} \| f\|_{L^2(\mathfrak{R}^3)}
    \\& = \|d_2^* u\|_{L^2(\mathfrak{R}^3)}(\|d_2^* u\|_{L^2(\mathfrak{R}^3)} - \sqrt{2}C\|f\|_{L^2(\mathfrak{R}^3)}).
\end{align*}
Let  $R = 1 + \sqrt{2}C\|f\|_{L^2(\mathfrak{R}^3)}$, and therefore, $\|d_2^* u\|_{L^2(\mathfrak{R}^3)} > R$ implies
\begin{align*}
\|d_2^* u\|_{L^2(\mathfrak{R}^3)} - \sqrt{2}C\|f\|_{L^2(\mathfrak{R}^3)} > 1,    
\end{align*}
hence
\begin{align*}
    d_uE(u) > \|d_2^* u\|_{L^2(\mathfrak{R}^3)} > R > 0,
\end{align*}
and therefore, $E$ has a unique minimizer in $W^{1,2}_{\alpha}(\mathfrak{R}^3)$, inside
\begin{align*}
    \{ u \in W^{1,2}_{\alpha}(\mathfrak{R}^3) \; | \;
    \|d_2^* u\|_{L^2(\mathfrak{R}^3)} < 1 + \sqrt{2}C\|f\|_{L^2(\mathfrak{R}^3)} \}.
\end{align*}
\end{proof}

$\varPhi_0$ is non-zero outside of a large ball $V \subset \mathfrak{R}^3$, and therefore, it induces a decomposition of the adjoint bundle to the longitudinal and transverse parts. Let $u^L$ and $u^T$ denote the longitudinal and transverse components, respectively. Note that since the metric $g$ is flat on $\mathfrak{R}^3 \setminus V$, we have $\mathfrak{R}^3 \setminus V = \mathbb{R}^3 \setminus V$

\begin{lemma}[Step 6]\label{Step6}
Let $u$ be the unique solution of Lemma \ref{minsol}, and $u^L$ be its longitudinal component with respect to the decomposition of the bundle over $\mathbb{R}^3 \setminus V$ induced by $\varPhi_0$. We have 
\begin{align*}
    |u^L(x)| \leq \frac{C}{|x|},
\end{align*}
for a constant $C$.
\end{lemma}

\begin{proof}
Since $f$ is compactly supported, if necessary we can enlarge $V$ such that $f_{|_{\mathbb{R}^3 \setminus V}} = 0$, and therefore, $u^L$ satisfies the equation
\begin{align*}
    \Delta u^L = 0,
\end{align*} 
on $\mathbb{R}^3 \setminus V$, thus, $u^L$ is a harmonic real-valued 1-form outside of a compact subset of $\mathbb{R}^3$. Let $u^L = u^L_1 dx_1 + u^L_2 dx_2 + u^L_3 dx_3$ on $\mathbb{R}^3 \setminus V$. With respect to the Euclidean metric, we have 
\begin{align*}
    \Delta  u^L = (\Delta u^L_1)dx_1 + 
    (\Delta u^L_2)dx_2 + (\Delta u^L_3)dx_3,
\end{align*}
and therefore, the 1-form $u^L$ is harmonic if and only if its coefficients are harmonic. This shows 
\begin{align*}
    \Delta u^L_1 = \Delta u^L_2 = \Delta u^L_3 = 0,
\end{align*}
on $\mathbb{R}^3 \setminus V$, and therefore, they are harmonic functions on the complement of a compact subset of $\mathbb{R}^3$. Functions of this type have been studied in \cite{MR633276}, which we burrow the following fact from.

Let $v$ be a subharmonic function defined over $\{ x \in \mathbb{R}^n \; | \; |x| > R\}$ for a positive real number $R$. Then there exist a non-constant subharmonic function $s(x)$ defined over $\mathbb{R}^n$, a real number $r > R$, and a constant $c\leq 0$ such that
\begin{align*}
    v(x) = s(x) + c |x|^{2-n} \quad \text{when} \quad |x| > r.
\end{align*}

Letting $v(x) = u^L_i(x)$ for $i \in \{1,2,3\}$, $n = 3$, and $R$ large enough such that $V \subset B_0(R)$, we get $u^L_i(x) = s_i(x) + \frac{c_i}{|x|}$ when $|x|>r$ for some $r>R$ and constants $c_i$. Furthermore, a similar statement holds for superharmonic functions, and therefore, in our case $s_i(x)$ is harmonic over entire $\mathbb{R}^3$. 

$u^L_i(x)$ is a harmonic section in $W^{1,2}(\mathbb{R}^3 \setminus V)$, and therefore, $\lim_{|x| \to \infty} u^L_i(x) = 0$, hence, $\lim_{|x| \to \infty} s_i(x) = 0$. This shows $s_i(x)$ is a bounded harmonic function on $\mathbb{R}^3$; thus $s_i \equiv 0$. This implies $u^L_i(x) = \frac{c_i}{|x|}$, and therefore,
\begin{align*}
    u^L = \sum_{i=1}^3 \frac{c_i}{|x|} dx_i,
\end{align*}
on $\mathbb{R}^3 \setminus V$, which proves the lemma.
\end{proof}

\begin{lemma}[Step 7]\label{Step7}
$u^T = O(e^{-r})$ as $r \to \infty$.
\end{lemma}

\begin{proof}
The argument is completely similar to the case studied by Foscolo, Proposition 5.8 in \cite{MR3801425}, and the step 3 of the proof of Lemma 7.10 in \cite{MR3439230}, since on $\mathfrak{R}^3 \setminus V$ the metric is flat.
\end{proof}

\begin{lemma}[Step 8] \label{Step8}
Let $f$ be a smooth, compactly supported, $\mathfrak{su}(2)$-valued 1-form on $\mathfrak{R}^3$. Let $\alpha \in [-\frac{1}{2},0)$. There  exist sufficiently small $\delta > 0$ and $\epsilon_0 > 0$ such that if $\| w e_0(A_0 \varPhi_0) \|_{L^3(\mathfrak{R}^3)} < \delta$ and $\epsilon < \epsilon_0$, then the unique solution $u(x)$ of Lemma \ref{minsol} satisfies 
\begin{align} \label{fboundsu}
        \|u\|_{W^{1,2}_{\alpha}} \leq C \|f\|_{L^2_{\alpha-2}},
\end{align}
for a constant $C$ independent of $\epsilon$.
\end{lemma}

\begin{proof}
If $\|u\|_{W^{1,2}_{\alpha}} = 0$, then \ref{fboundsu} is trivial. Suppose $\|u\|_{W^{1,2}_{\alpha}} \neq 0$. Then \ref{fboundsu} is equivalent to 
\begin{align}
         \|u\|^2_{W^{1,2}_{\alpha}} \leq C \|u\|_{W^{1,2}_{\alpha}}\|f\|_{L^2_{\alpha-2}}.
\end{align}

By the integration by parts we have
\begin{align*}
    \|f\|&_{L^2_{\alpha-2}} 
    \|u\|_{W^{1,2}_{\alpha}} \geq \|f\|_{L^2_{\alpha-2}} 
    \|u\|_{L^2_{\alpha}} = 
    \|w^{-\alpha+\frac{1}{2}}f\|_{L^2} 
    \|w^{-\alpha-\frac{3}{2}}u\|_{L^2} 
    \geq
    \int_{\mathfrak{R}^3} w^{-2\alpha-1}\langle f,u \rangle vol_{g} \\&=
    \int_{\mathfrak{R}^3} w^{-2\alpha-1}\langle d_2d_2^*u,u \rangle vol_{g} = 
    \int_{\mathfrak{R}^3} w^{-2\alpha-1}(|\nabla_{A_0} u|^2 + |[\varPhi_0,u]|^2) vol_{g}
    \\&+
    \int_{\mathfrak{R}^3} w^{-2\alpha-1}\langle Ric(u),u \rangle vol_{g}
    +
    \int_{\mathfrak{R}^3} w^{-2\alpha-1}\langle *[e_0 \wedge u],u \rangle vol_{g} -
    \alpha (1+2 \alpha) \|u\|^2_{L^2_{\alpha}}.
\end{align*}
The previous two lemmas show $u(x) = O(|x|^{-1})$ as $|x| \to \infty$, and therefore,
\begin{align*}
d_2^*u(x)u(x) = O(|x|^{-3}),    
\end{align*}
hence, the asymptotic terms do not appear in the integration by parts.

Let
\begin{align*}
    \|\nabla_{A_0}u\|_{L^2_{\alpha-1}}^2 = a\|\nabla_{A_0}u\|_{L^2_{\alpha-1}}^2 + b \|\nabla_{A_0}u\|_{L^2_{\alpha-1}}^2 \; \text{where} \; a+b=1 \; \text{ and } \; a,b > 0.
\end{align*}
Using Lemma \ref{Hardy},
\begin{align*}
    \|\nabla_{A_0}u\|_{L^2_{\alpha-1}}^2 = 
    a\|\nabla_{A_0}u\|_{L^2_{\alpha-1}}^2 + b \|\nabla_{A_0}u\|_{L^2_{\alpha-1}}^2 
    \geq \frac{a}{C_{g,\alpha}}\|u\|_{L^2_{\alpha}}^2 + b \|\nabla_{A_0}u\|_{L^2_{\alpha-1}}^2.
\end{align*}
Let $a/C_{g,\alpha} = b$. We get 
\begin{align*}
    a = \frac{C_{g, \alpha}}{C_{g, \alpha}+1},
    \quad \quad b = \frac{1}{C_{g, \alpha}+1}.
\end{align*}
For instance, when $g$ is flat and $\mathfrak{R}^3 = \mathbb{R}^3$, we have $C_{g,\alpha} = 1/(\alpha+1)^2$, and therefore,
\begin{align*}
    a = \frac{1}{(\alpha+1)^2+1}, \quad \quad b = \frac{(\alpha+1)^2}{(\alpha+1)^2+1}.
\end{align*}
Note that $C_{g, \alpha} \to 1/(\alpha+1)^2$ as $\epsilon \to 0$.

For these specific choices for $a$ and $b$, we get
\begin{align*}
    \|\nabla_{A_0}u\|_{L^2_{\alpha-1}}^2 \geq
    \frac{1}{C_{g, \alpha}+1} \|u\|^2_{W^{1,2}_{\alpha}},
\end{align*}
and therefore,
\begin{align*}
    \|f\|_{L^2_{\alpha-2}} 
    \|u\|_{W^{1,2}_{\alpha}} &\geq \frac{1}{C_{g, \alpha}+1} \|u\|^2_{W^{1,2}_{\alpha}}  -
    |\alpha (1+2 \alpha)| \|u\|^2_{L^2_{\alpha}}
    \\&+
    \int_{\mathfrak{R}^3} w^{-2\alpha-1}\langle Ric(u),u \rangle vol_{g}
    +
    \int_{\mathfrak{R}^3} w^{-2\alpha-1}\langle *[e_0 \wedge u],u \rangle vol_{g} \\&
    \geq \left( \frac{1}{C_{g, \alpha}+1} - |\alpha(1+2\alpha)| \right) \|u\|^2_{W^{1,2}_{\alpha}} \\& + \int_{\mathfrak{R}^3} w^{-2\alpha-1}\langle Ric(u),u \rangle vol_{g}
    +
    \int_{\mathfrak{R}^3} w^{-2\alpha-1}\langle *[e_0 \wedge u],u \rangle vol_{g}. 
\end{align*}
As $\epsilon \to 0$,
\begin{align}\label{balpha2}
     \frac{1}{C_{g, \alpha}+1} + \alpha(1+2\alpha)
     \to b_{\alpha} :=
     \frac{(\alpha+1)^2}{(\alpha+1)^2+1} + \alpha(1+2\alpha).
\end{align}
Let $\epsilon_1 > 0$ be sufficiently small such that for any $0 < \epsilon < \epsilon_1$, 
\begin{align*}
    |\left( \frac{1}{C_{g, \alpha}+1} + \alpha(1+2\alpha)
 \right) - b_{\alpha}| \leq \frac{1}{1000}.
\end{align*}
The equation $b_{\alpha} = 0$ has no solutions. In fact, $b_{\alpha} > 0.18 > 0$. Moreover, for any $\epsilon \in (0, \epsilon_1)$, the solutions to the equation $\frac{1}{C_{g, \alpha}+1} + \alpha(1+2\alpha) > 0$. 

Moreover, for all $\alpha \in [-\frac{1}{2},0)$ and $0 < \epsilon < \epsilon_1$, 
\begin{align*}
    \left( \frac{1}{C_{g, \alpha}+1} + \alpha(1+2\alpha)
 \right) > \frac{1}{10},
\end{align*}
and therefore,
\begin{align*}
    \|f\|_{L^2_{\alpha-2}} 
    \|u\|_{W^{1,2}_{\alpha}} \geq \frac{1}{10} \|u\|^2_{W^{1,2}_{\alpha}} + \int_{\mathfrak{R}^3} w^{-2\alpha-1}\langle Ric(u),u \rangle vol_{g}
    +
    \int_{\mathfrak{R}^3} w^{-2\alpha-1}\langle *[e_0 \wedge u],u \rangle vol_{g}. 
\end{align*}
As mentioned, by scaling the metric on $M$ and without loss of generality, we can assume $sup_{x \in M} |Ric(x)| \leq \frac{1}{100}$, and therefore, by Lemma \ref{Hardy},
\begin{align*}
    \int_{\mathfrak{R}^3} w^{-2\alpha-1}\langle Ric(u),u \rangle vol_{g} 
    & \leq
    \frac{1}{100}
    \| u\|^2_{L^2_{\alpha-1}} 
    \leq \frac{1}{100}
    \| u\|^2_{L^2_{\alpha}}  
    \leq \frac{C_{g,\alpha}}{100}
    \| \nabla_{A_0}u\|^2_{L^2_{\alpha-1}}
    \leq \frac{C_{g,\alpha}}{100}
    \| u\|^2_{W^{1,2}_{\alpha}}.
\end{align*}
hence,
\begin{align*}
    \|f\|_{L^2_{\alpha-2}} 
    \|u\|_{W^{1,2}_{\alpha}} &\geq \frac{1}{10} \|u\|^2_{W^{1,2}_{\alpha}} + \int_{\mathfrak{R}^3} w^{-2\alpha-1}\langle Ric(u),u \rangle vol_{g}
    +
    \int_{\mathfrak{R}^3} w^{-2\alpha-1}\langle *[e_0 \wedge u],u \rangle vol_{g} \\& \geq
    (\frac{1}{10}
     - \frac{C_{g,\alpha}}{100}
     )
    \|u\|^2_{W^{1,2}_{\alpha}}
    +
    \int_{\mathfrak{R}^3} w^{-2\alpha-1}\langle *[e_0 \wedge u],u \rangle vol_{g}. 
\end{align*}
Moreover, $1/(\alpha+1)^2 \leq 4$ for $\alpha \in [-\frac{1}{2},0)$. Let $\epsilon_2 > 0$ be sufficiently small such that 
\begin{align*}
    | C_{g, \alpha} - \frac{1}{(\alpha+1)^2} | \leq \frac{1}{10},
\end{align*}
and therefore, $C_{g,\alpha} \leq \frac{9}{2}$. This implies 
\begin{align*}
    \|f\|_{L^2_{\alpha-2}} 
    \|u\|_{W^{1,2}_{\alpha}} &\geq
    (\frac{1}{10}
     - \frac{C_{g,\alpha}}{100}
     )
    \|u\|^2_{W^{1,2}_{\alpha}}
    +
    \int_{\mathfrak{R}^3} w^{-2\alpha-1}\langle *[e_0 \wedge u],u \rangle vol_{g} \\ &  \geq  \frac{1}{20}
    \|u\|^2_{W^{1,2}_{\alpha}}
    +
    \int_{\mathfrak{R}^3} w^{-2\alpha-1}\langle *[e_0 \wedge u],u \rangle vol_{g}. 
\end{align*}

Regarding the error term,
\begin{align*}
    \int_{\mathfrak{R}^3} w^{-2\alpha-1}\langle *[e_0 \wedge u],u \rangle vol_{g} &\leq 
    \|w e_0 \|_{L^3} \|u \|_{L^2_{\alpha}} \|w^{-\alpha - \frac{1}{2}} u\|_{L^6}
    \\&\leq 
    \|w e_0 \|_{L^3} \|u \|_{W^{1,2}_{\alpha}} \|w^{-\alpha - \frac{1}{2}} u\|_{L^6}.
\end{align*}
By the Sobolev inequality 
\begin{align*}
\|w^{-\alpha - \frac{1}{2}} u\|_{L^6}
\leq C_{Sob}\|w^{-\alpha - \frac{1}{2}} u\|_{W^{1,2}} \leq
C' C_{Sob} \|u\|_{W^{1,2}_{\alpha}},
\end{align*}
for a uniform constant $C'$, and therefore,
\begin{align*}
    \int_{\mathfrak{R}^3} w^{-2\alpha-1}\langle *[e_0 \wedge u],u \rangle vol_{g} \leq \delta C' C_{Sob} \|u\|^2_{W^{1,2}_{\alpha}}.
\end{align*}
Taking $\delta$ small enough such that $\delta < \frac{1}{100 C' C_{Sob}}$, we get 
\begin{align*}
    \|f\|_{L^2_{\alpha-2}} 
    \|u\|_{W^{1,2}_{\alpha}} \geq 
        \frac{1}{25} \|u\|^2_{W^{1,2}_{\alpha}},
\end{align*}
and therefore,
\begin{align*}
    \|u\|_{L^2_{\alpha}} \leq C \|f\|_{L^2_{\alpha-2}}.
\end{align*}
\end{proof}

We are progressing towards proving 
\begin{align*}
        \|u\|_{W^{2,2}_{\alpha}} \leq C \|f\|_{L^2_{\alpha-2}}.
\end{align*}
The next lemma is a necessary estimation in this direction.

\begin{lemma}[Step 9] \label{curvaturebound}
Let $(A_0, \varPhi_0)$ be the constructed approximate monopole. Then we have the following pointwise approximation 
\begin{align*}
    |F_{A_0}| = |d_{A_0} \varPhi_0| 
    \leq \frac{C}{\lambda^{-2}+r^2}.
\end{align*}
\end{lemma}

\begin{proof}
The proof in the case $\mathfrak{R}^3 = \mathbb{R}^3$ can be found in \cite[Proposition 4.14]{MR3801425}. The essential point is that this approximation holds everywhere on $\mathbb{R}^3$, where the pair $(A_0, \varPhi_0)$ is equal to the scaled BPS-monopole, where it is equal to the scaled Dirac monopole, and also over the region in between. 

Let $g$ and $g_0$ denote the Riemannian metrics on $\mathfrak{R}^3$ and $\mathbb{R}^3$, respectively. Recall that for any $2$-form $\beta$ with valued in any vector bundle, we have
\begin{align*}
    |\beta|_g - |\beta|_{g_0} = \sum_{k,l} \langle\beta_k, \beta_l\rangle (g^{k,l} - g_0^{k,l}) &=  - \frac{1}{3} \sum_{k,l,m,n} \langle\beta_k, \beta_l\rangle R_{klmn} \beta_k \beta_l x_m x_n + O(|x|^3) \\ &\leq C_1 R | \beta|^2_{g_0}|x|^2,
\end{align*}
where $C_1> 0$ is a constant and $R$ is the maximum of the Riemann curvature tensor of $g$. Therefore, 
\begin{align*}
    |\beta|_g \leq  |\beta|_{g_0} +  C_1 R | \beta|^2_{g_0}|x|^2,
\end{align*}
Let $\beta = F_{\eta_j^*(A^{\lambda_j}_{BPS})}$,
\begin{align*}
    |(*_0 - *)F_{\eta_j^*(A^{\lambda_j}_{BPS})}|_g \leq C_2( \frac{1}{\lambda_j^{-2} + |x|^2} +  R \frac{|x|^2}{\lambda_j^{-2} + |x|^2}) 
    \leq C_2( \frac{1 + R\epsilon^2}{\lambda_j^{-2} + |x|^2})
    \leq \frac{C}{\lambda_j^{-2} + |x|^2},
\end{align*}
for a constant $C$, when $\lambda_j$ is sufficiently large.
\end{proof}

The following lemma is the last step of proving Theorem \ref{q}.

\begin{lemma}[Step 10]
Let $f$ be a smooth, compactly supported, $\mathfrak{su}(2)$-valued 1-form on $\mathfrak{R}^3$. Let $\alpha \in [-\frac{1}{2},0)$. There exists a sufficiently small $\delta > 0$ and $\epsilon_0$ such that if $\| w e_0(A_0 \varPhi_0) \|_{L^3(\mathfrak{R}^3)} < \delta$ and $\epsilon < \epsilon_0$, then the unique solution $u(x)$ of Lemma \ref{minsol} is an element of $W^{2,2}_{\alpha}$ and satisfies 
\begin{align*}
        \|u\|_{W^{2,2}_{\alpha}} \leq C \|f\|_{L^2_{\alpha-2}},
\end{align*}
for a constant $C$ independent of $\lambda$.
\end{lemma}
\begin{proof}
We should find a uniform bound on 
\begin{align*}
\|\nabla_{A_0} (d_2^*u)\|_{L^2_{\alpha-2}} + \|[\varPhi_0, d_2^*u] \|_{L^2_{\alpha-2}} = \|w^{-\alpha+\frac{1}{2}}\nabla_{A_0} (d^*_2u)\|_{L^2}^2 + \|w^{-\alpha+\frac{1}{2}}[\varPhi_0,d_2^*u]\|_{L^2}^2,    
\end{align*}
in terms of $\|f\|_{L^2_{\alpha-2}}$. 

Let $(a, \phi) = D^*(u,0) = d_2^*u$. By the monopole Weitzenb\"ock formula we have
\begin{align*}
    D^* D (d_2^* u) &= D D^* (d_2^* u) + 2 \langle d_{A_0} \varPhi_0 , d_2^* u\rangle \\ & =
    \nabla_A^* \nabla_A (d_2^*u) - ad(\varPhi)^2 (d_2^*u)  + \langle Ric(x), d_2^*u \rangle
      + *[e_0 \wedge d_2^*u] + 2 \langle d_{A_0} \varPhi_0 , d_2^* u\rangle.
\end{align*}
By multiplying the formula by $w^{-2\alpha+1}d_2^*u$ and integrating over $\mathfrak{R}^3$, we get
\begin{align*}
    \|\nabla_A (d_2^*u)\|_{L^2_{\alpha - 2}}^2 &+ \|[\varPhi,(d_2^*u)]\|_{L^2_{\alpha - 2}}^2 \\ &\leq
    \|D (d_2^* u)\|_{L^2_{\alpha - 2}}^2
    + \sup_{x \in \mathfrak{R}^3} |Ric(x)| \|d_2^*u\|_{L^2_{\alpha - 2}}^2
    \\&+ \int_{\mathfrak{R}^3}
    w^{-2\alpha+1}(|e_0| + 2|d_{A_0} \varPhi_0|) |d_2^*u|^2 vol_g \\& + 
    (-2\alpha+1)\int_{\mathfrak{R}^3} w^{-2\alpha} |\nabla w| |d_2^*u|(|\nabla_{A_0}(d_2^* u)|+|D(d_2^* u)|) vol_g, 
\end{align*}
where the last integral is the asymptotic term of the Stokes' theorem.

We start by bounding $\|D (d_2^* u)\|_{L^2_{\alpha - 2}}^2$. First note that
\begin{align*}
    D (d_2^*u) = DD^*(u,0)=(f,*[e_0 \wedge *u]).
\end{align*}
By multiplying this formula by $w^{-\alpha+\frac{1}{2}}$ and taking the $L^2$-norm over $\mathfrak{R}^3$, we get
\begin{align*} 
    \|D(d_2^*u)\|_{L^2_{\alpha-2}}^2 &= 
    \|f\|_{L^2_{\alpha-2}}^2 + \|[we_0 \wedge *w^{-\alpha-\frac{1}{2}}u]\|_{L^2}^2
    \\ &\leq \|f\|_{L^2_{\alpha-2}}^2 + \|w e\|^2_{L^3}\|w^{-\alpha-\frac{1}{2}}u\|^2_{L^6} \\& \leq 
    \|f\|_{L^2_{\alpha-2}}^2 + \delta^2 C_{Sob} \|w^{-\alpha-\frac{1}{2}}u\|^2_{W^{1,2}}
    \\& = 
    \|f\|_{L^2_{\alpha-2}}^2 + \delta^2 C_{Sob}
    \left( \|w^{-\alpha-\frac{1}{2}}u\|^2_{L^2}
    + 
    \|\nabla_{A_0}(w^{-\alpha-\frac{1}{2}}u)\|^2_{L^2} \right) 
    \\& \leq
    \|f\|_{L^2_{\alpha-2}}^2 + \delta^2 C_{Sob}
    \left( \|u\|^2_{L^2_{\alpha}}
    + 
    \|\nabla_{A_0}(w^{-\alpha-\frac{1}{2}}u)\|^2_{L^2} \right) 
    \\& \leq
    C'\|f\|_{L^2_{\alpha-2}}^2 + \delta^2 
    C_{Sob}
    \|\nabla_{A_0}(w^{-\alpha-\frac{1}{2}}u)\|^2_{L^2}  
    \\& \leq
    C' \|f\|_{L^2_{\alpha-2}}^2 + \delta^2 C_{Sob}
    \left(
    (-\alpha-\frac{1}{2})^2\|w^{-\alpha-\frac{3}{2}} |\nabla w| u\|^2_{L^2}
    +
    \|\nabla_{A_0} u\|^2_{L^2_{\alpha-1}}
    \right)
     \\& \leq
    C'' \|f\|_{L^2_{\alpha-2}}^2 + \delta^2 C_{Sob}
    \|\nabla_{A_0} u\|^2_{L^2_{\alpha-1}}
    \\& \leq C
    \|f\|^2_{L^{2}_{\alpha-2}},
\end{align*}
for positive constants $C_{Sob}, C'$, and $C''$.

Regarding the term
\begin{align*}
\sup_{x \in \mathfrak{R}^3} |Ric(x)| \|d_2^*u\|_{L^2_{\alpha - 2}}^2 \leq \sup_{x \in \mathfrak{R}^3} |Ric(x)| \|d_2^*u\|_{L^2_{\alpha - 1}}^2,    
\end{align*}
the Ricci curvature is bounded and $\|d_2^*u\|_{L^2_{\alpha - 1}}^2$ can be bounded uniformly by $\|f\|_{L^2_{\alpha-2}}^2$, as we observed in the proof of Lemma \ref{Step8}.

Regarding the error term, using Lemma \ref{errores},
\begin{align*}
    \int_{\mathfrak{R}^3}
    w^{-2\alpha+1}&|e_0||d_2^*u|^2 vol_g  \leq
    c \|d_2^*u\|_{L^2_{\alpha - 2}}^2 \leq 
    c \|d_2^*u\|_{L^2_{\alpha - 1}}^2 \leq
    C \|f\|_{L^2_{\alpha-2}}^2,
\end{align*}
for positive constants $c$ and $C$.

Regarding the term,
\begin{align*}
    \int_{\mathfrak{R}^3}
    w^{-2\alpha+1}|d_{A_0} \varPhi_0||d_2^*u|^2 vol_g,
\end{align*}
recall that following Lemma \ref{curvaturebound}, the term $|d_{A_0} \varPhi_0|$ can be estimated, 
\begin{align*}
    |d_{A_0} \varPhi_0|  \leq 
    \frac{c}{\lambda^{-2}+r^2} \rightarrow
    w^2|d_{A_0} \varPhi_0|
    \leq 
    \frac{cw^2}{\lambda^{-2}+r^2} \leq 
    C',
\end{align*}
for a uniform constant $C'$, and therefore, 
\begin{align*}
    \int_{\mathfrak{R}^3}
    w^{-2\alpha+1}|d_{A_0} \varPhi_0||d_2^*u|^2 vol_g &\leq 
    C'\int_{\mathfrak{R}^3}
    w^{-2\alpha-1}|d_2^*u|^2 vol_g
    \\&= C' \|d_2^*u\|^2_{L^2_{\alpha-1}}
    \leq C'' \|f\|^2_{L^2_{\alpha-2}}.
\end{align*}
Therefore,
\begin{align*}
        \|u\|_{W^{2,2}_{\alpha}} \leq C \|f\|_{L^2_{\alpha-2}},
\end{align*}
for a uniform constant $C$.
\end{proof}

The assumption in Theorem \ref{q} is that the error estimate $\| w e_0(A_0 \varPhi_0) \|_{L^3(\mathfrak{R}^3)} < \delta$ for a sufficiently small $\delta$. The following theorem states $\| w e_0(A_0 \varPhi_0) \|_{L^3(\mathfrak{R}^3)}$ can be made as small as necessary by increasing the masses $\lambda_j$.

\begin{theorem}\label{errorsmall}
For any $\delta > 0$, there exists a sufficiently large $\lambda_j = \epsilon_j^{-2} > 0$ such that the monopole $(A_0, \varPhi_0)$ defined in \ref{modmon} with the parameters $\lambda_j$, satisfies 
\begin{align}\label{L3error}
    \| w e_0(A_0, \varPhi_0) \|_{L^3(\mathfrak{R}^3)} < \delta.
\end{align}
\end{theorem}

\begin{proof}
For each $j \in \{1, \hdots, k\}$, from Lemma \ref{errores}, we have
\begin{align*}
    (e_0^{BPS})_{|_{B_{3\epsilon_j}(q_j)}}\leq C.
\end{align*}
and therefore,  
\begin{align*}
    \| w e_0^{BPS}\|_{L^3(B_{3\epsilon_j}(q_j))}^3 &= 
    \int_{B_{3\epsilon_j}(q_j)}
    |we_0^{BPS}|^3 vol_g
    \leq 
    C_1 \int_{B_{3\epsilon_j}(q_j)}
   |w|^3 vol_g \\&=
    C_1
    \int_{B_{3\epsilon_j}(q_j)}
    (\lambda_j^{-2} + |x|^2)^{\frac{3}{2}} vol_g
    \leq C_2 \epsilon_j^3
    (\lambda_j^{-2} + \epsilon_j^2)^{\frac{3}{2}}.
    \end{align*}
for a positive uniform constants $C_1$ and $C_2$, and therefore, it can be made as small as necessary. 
\end{proof}

\subsection[The Linear Equation over $M \setminus (\cup_j B_{2 \epsilon_j}(q_j) \cup S_p )$]{The Linear Equation over $\pmb{M \setminus (\cup_j B_{2 \epsilon_j}(q_j) \cup S_p )}$} \label{linearlong}

In this section, we study the linearized equation $d_2 d_2^* u = f$ on $M \setminus (\cup_j B_{2 \epsilon_j}(q_j) \cup S_p )$, away from the points where the scaled BPS-monopoles are located, and set the stage for solving this linearized equation.

On this region the pair $(A_0, \varPhi_0)$ is a reducible monopole with a non-zero Higgs field $\varPhi_0$, and therefore, it induces a decomposition of the adjoint bundle as $\mathfrak{g}_P = \underline{\mathbb{R}} \oplus L$, where $\underline{\mathbb{R}}$ is the sub-bundle generated by the image of $\varPhi_0$ and $L$ is the orthogonal sub-bundle. Corresponding to the bundle decomposition $\mathfrak{g}_P = \underline{\mathbb{R}} \oplus L$, a section or a $\mathfrak{g}_P$-valued tensor $f$ supported on this region can be written as $f = (f^L, f^T)$.

This bundle decomposition is preserved by $d_2$, $d_2^*$, and $d_2 d_2^*$. Hence the equation $d_2 d_2^* u = f$ on this region reduces to two equations for $u^L$ and $u^T$. The equation for $u^L$ is given by
\begin{align}
    \Delta u^L = f^L,
\end{align}
and the equation for $u^T$ is 
\begin{align*}
d_2 d_2^* u^T = f^T.
\end{align*}

In the following section, we will introduce the appropriate function spaces to solve these equations.

\subsection[Function Spaces on $M\setminus (\cup_j B_{2 \epsilon_j}(q_j) \cup S_p)$]{Function Spaces on $\pmb{M\setminus (\cup_j B_{2 \epsilon_j}(q_j) \cup S_p)}$}\label{FunctionSpacesonMsetminus}

In this section, we set the stage to study the linearized equation over $M \setminus (\cup_j B_{2 \epsilon_j}(q_j) \cup S_p)$. We start by defining the suitable weighted Sobolev spaces on $\mathfrak{su}(2)$-valued differential forms on $M \setminus (\cup_j B_{2 \epsilon_j}(q_j) \cup S_p)$.
These spaces can be used to solve the problem away from the points $q_j$. 

Let $\delta_{p_i}$ be the injectivity radius at $p_i$. For each point $p_i \in S_p$, let 
\begin{align} \label{wp}
    w_i(x) = 
    \begin{cases}
        r_i, \quad &r_i\leq \delta_{p_i}\\
        1, \quad &r_i \geq 1,
    \end{cases}
\end{align}
where $r_i$ is the geodesic distance from $p_i$.

\begin{definition}\label{weighted-ext}
Let $U = M \setminus (\cup_j B_{2\epsilon_j}(q_j) \cup S_p)$ and $U_{ext} = M \setminus (\cup_j B_{2\epsilon_j}(q_j) \cup_i B_{2\epsilon_i}(p_i))$. Let $\alpha \in \mathbb{R}$. For all smooth compactly supported $\mathfrak{su}(2)$-valued differential forms $ u \in \Omega^{\boldsymbol{\cdot}}(U, \mathfrak{su}(2))$, let
\begin{align*}
    \|u\|^2_{L^2_{\alpha}(U)} = 
    \|u\|^2_{L^2(U_{ext})}
    +
    \sum_{i=1}^n
    \|u^T\|^2_{L^2(B_{2 \epsilon_i}(p_i))}
    +
    \sum_{i=1}^n
    \|w_i^{-\alpha-\frac{3}{2}} u^L\|^2_{L^2(B_{2 \epsilon_i}(p_i))}.
\end{align*}
Furthermore, 
\begin{align*}
    \|u\|^2_{W^{1,2}_{\alpha}(U)} &= 
    \|u\|^2_{L^2_{\alpha}(U)} + 
    \|\nabla_{A_0} u\|^2_{L^2_{\alpha-1}(U)} + 
    \frac{1}{2} \| [\varPhi_0,u] \|_{L^2_{\alpha-1}(U)}^2.
\end{align*}
Moreover,
\begin{align*}
    \|u\|^2_{W^{2,2}_{\alpha}(U)} = 
    \|u\|^2_{W^{1,2}_{\alpha}(U)} + 
    \|\nabla_{A_0}(d_2^* u)\|^2_{L^2_{\alpha-2}(U)}
    + \frac{1}{2} \|[\varPhi_0, d_2^*u]\|^2_{L^2_{\alpha-2}(U)}.
\end{align*}
The spaces $W^{k,2}_{\alpha}(U)$ are defined as the completion of $C^{\infty}_0(U)$ with respect to the corresponding norms for $k \in \{0, 1, 2\}$. Furthermore, one can define similar norms and weighted Sobolev spaces $W^{k,p}_{\alpha}(U)$ for any $p \geq 2$ and $k \in \{0,1,2\}$.
\end{definition}

\subsection[The Longitudinal Component]{The Longitudinal Component and the Lockhart-McOwen Theory}\label{TheLongitudinalComponent}

In this section, we study the weighted Sobolev spaces of the sections of the longitudinal component, and set the necessary background to solve $\Delta u^T = f^T$. The main goal of this section is to show that these weighted Sobolev spaces on the longitudinal component are suitable for studying elliptic operators, more specifically, the Laplacian. These spaces are closely related to the Lockhart-McOwen Sobolev spaces on asymptotically cylindrical manifolds \cite{MR837256}. 

The following example gives a good picture of the real-valued sections in these weighted Sobolev spaces.

\begin{example}
Let $(M,g)$ be a closed, Riemannian, $n$-dimensional manifold. Let $p \in M$. Let $\delta_p$ be the injectivity radius at $p$ and $r: M \to \mathbb{R}$ a smooth function such that
\begin{align*}
    r(x) = 
    \begin{cases}
    \text{geodesic distance from p} \quad & \text{ on } \quad B_{\delta}(p),\\
    1 \quad & \text{ on } \quad M \setminus B_{2\delta}(p_i).
    \end{cases}
\end{align*}
Then $r^{\delta} \in L^p_{\alpha}(M)$ if and only if $\delta > \alpha$.
\end{example}

To understand the longitudinal part of these weighted Sobolev spaces, using a conformal mapping, one can transform them into the weighted Sobolev spaces over asymptotically cylindrical manifolds. The punctured ball $B_{\epsilon}(0) \setminus \{ 0 \} \subset \mathbb{R}^3$ can be identified with the cylinder $ (-\log(\epsilon), + \infty) \times S^2$ using a map $L_0: B_{\epsilon}(0) \setminus \{ 0 \} \to (-\log(\epsilon), + \infty) \times S^{2}$, defined by
\begin{align}\label{LLL}
    (t, \theta, \varphi) := L_0(r, \theta, \varphi) = (-\log (r), \theta, \varphi), 
\end{align}
where $(r, \theta, \varphi)$ denotes the spherical coordinates on $B_{\epsilon}(0) \subset \mathbb{R}^3$. 

Equip the punctured ball with the flat metric $g_0 = dx^2 + dy^2 + dz^2$, which in spherical coordinates can be written as $g_0 = dr^2 + r^2 g_{S^2} = dr^2 + r^2 (d \theta^2 + \sin^2 \theta d \varphi^2)$, and the cylinder with the standard product metric $g_{Cyl} = dt^2 + g_{S^2}$. The map $L_0$ takes the flat metric on the punctured ball to $e^{-2t}(dt^2 + g_{S^2})$ on $(-\log(\epsilon), + \infty) \times S^2$, which is conformally equivalent to the cylindrical metric $g_{cyl}$. 

More generally, the Riemannian metric $g$ on each ball $B_{ \epsilon_i}(p_i)$ using the exponential map and in geodesics normal coordinates can be written as $g = dr^2 + \psi(r, \theta)g_{S^2}$, where $\psi$ is a smooth positive function such that $\lim_{r \to 0} \psi(r, \theta) \to 1$. Let $\mu(t, \theta) =r^{-2} \psi(r, \theta) = e^{-2t} \psi(e^{-t}, \theta)$. One can define the diffeomorphism $L_i: B_{ \epsilon_i}(p_i) \to (-\log(\epsilon_i), +\infty) \times \mathbb{R}$, similar to \ref{LLL}, that takes the metric $g$ to $e^{-2t}(dt^2 + \mu(t. \theta) g_{S^2})$, which is conformally equivalent to
\begin{align*}
    \tilde{g} = dt^2 + \mu(t, \theta) g_{S^2},
\end{align*}
where $\mu(t, \theta) \to 1$ as $t \to \infty$.  The metric $\tilde{g}$ is asymptotically cylindrical

By gluing the maps $L_i$ to the identity map on $M \setminus \cup_i B_{2 \epsilon_i}(p_i)$ and extend it smoothly to the necks $\cup_i (B_{2\epsilon_i}(p_i) \setminus B_{ \epsilon_i}(p_i))$,
we get a diffeomorphism 
\begin{align*}
    L: M \setminus S_p \to M_{Cyl} := (M \setminus \cup_{i=1}^n B_{\epsilon_i}(p_i)) \bigcup ( \cup_i (-log(\epsilon_i),+\infty) \times S^2 ),
\end{align*}
where $M_{Cyl}$ is equipped with an asymptotically cylindrical metric. Furthermore, $L$ takes the vector bundle of differential forms on $M$ to asymptotically translation-invariant asymptotically cylindrical bundles over $M_{Cyl}$. 

In order to have the Fredholm property for the elliptic differential operators like Laplacian or $d+d^*$ on asymptotically cylindrical manifolds, one should use suitable classes of Banach spaces as domain and co-domain, as introduced by Lockhart and McOwen \cite{MR837256}. 

\begin{definition}[Lockhart-McOwen Sobolev Spaces]
Let $(X,g_X)$ be an $n$-dimensional asymptotically cylindrical Riemannian manifold. Let $X_0 \subset X$ be a compact subset. Let $\rho: X \to \mathbb{R}$ be a smooth function such that on $X \setminus X_0$ it agrees with the geodesic distance from a point $x_0 \in X$. Let $(E,h_E,\nabla_E) \to X$ be an asymptotically cylindrical bundle. Let $p \geq 1$, $ k\geq 0$, and $\beta \in \mathbb{R}$. For any smooth compactly supported section $u \in \Gamma(E)$, let 
\begin{align*}
\|u\|^p_{W^{k,p}_{Cyl,\beta}(E)} = 
\sum_{j=0}^k \int_X | e^{-\beta \rho} \nabla^j_E u|_{h_E}^p vol_{g_X}.
\end{align*}
Let $W^{k,p}_{Cyl,\beta}(X,E)$ denote the completion of $C_0^{\infty}(X,E)$ with respect to this norm.
\end{definition}

These weighted Sobolev spaces over asymptotically cylindrical manifolds are closely related to the weighted spaces defined in Definition \ref{weighted-ext}.

\begin{lemma}
Let $L_0$ be the map defined in \ref{LLL}.
Let $f$ be a section of a vector bundle above $B_{\epsilon}(0) \setminus \{0\} \subset \mathbb{R}^3$. We have
\begin{align*}
f \in W^{k,p}_{\alpha}(B_{\epsilon}(0)) \Longleftrightarrow (L_0^{-1})^*f \in W^{k,p}_{Cyl,-\alpha}((-\log(\epsilon), + \infty) \times S^2).  
\end{align*}
Moreover,
\begin{align*}
 \|f\|_{W^{k,p}_{\alpha}(B_{\epsilon}(0))} = \|(L_0^{-1})^* f\|_{W_{Cyl,-\alpha}^{k,p}((-\log(\epsilon), + \infty) \times S^2)},
\end{align*}
and therefore, $\|f\|_{W^{k,p}_{\alpha}(M \setminus S_p)}$ and $ \|(L_0^{-1})^* f\|_{W_{Cyl,-\alpha}^{k,p}(M_{Cyl})}$ are equivalent norms.
\end{lemma}

\begin{proof}
Let $L^p_{Cyl}$ and $W^{k,p}_{Cyl}$ denote the ordinary Sobolev spaces on the cylinder with respect to its cylindrical Riemannian metric $g_{Cyl}$. 

Let $k=0$. By a change of variable we have 
\begin{align*}
    \|f\|_{L^{p}_{\alpha}(B_{\epsilon}(0))} =  
\|e^{\alpha t}(L_0^{-1})^*f\|_{L^p_{Cyl}((-\log(\epsilon), + \infty) \times S^2)} = \|(L_0^{-1})^*f\|_{L^p_{Cyl, -\alpha}((-\log(\epsilon), + \infty)\times S^2)}.
\end{align*}

Let $k>0$. Let $\nabla$ be the Levi-Civita connection on $B_{\epsilon}(0)$ with respect to the standard euclidean metric and $\nabla_{Cyl}$ be the Levi-Civita connection on $(-\log(\epsilon), +\infty) \times S^2$ with respect to the cylindrical metric. We have the pointwise equality
\begin{align*}
    |e^{-jt} \nabla^j f| = |\nabla^j_{Cyl} (L_0^{-1})^*f|,
\end{align*}
and therefore,
\begin{align*}
    \|r^{-\alpha - \frac{3}{2}+j}\nabla^j f\|_{L^p(B_{\epsilon}(0))} &=
\|e^{\alpha t} \nabla_{Cyl}^j(L_0^{-1})^*f\|_{L^p_{Cyl}((-\log(\epsilon), + \infty)\times S^2)} \\&= \|\nabla_{Cyl}^j(L_0^{-1})^*f\|_{L^p_{Cyl, -\alpha}((-\log(\epsilon), + \infty)\times S^2)}.
\end{align*}
By summing over $j$, 
\begin{align*}
 \|f\|_{W^{k,p}_{\alpha}(B_{\epsilon}(0))} &=  \|e^{\alpha t} (L_0^{-1})^*f\|_{W_{Cyl}^{k,p}((-\log(\epsilon), + \infty) \times S^2)} = \|(L_0^{-1})^* f\|_{W_{Cyl,-\alpha}^{k,p}((-\log(\epsilon), + \infty) \times S^2)}.
\end{align*}
\end{proof}

The key property of these weighted Sobolev spaces is described in the following lemma.

\begin{lemma}
On real-valued 1-forms,
$\Delta: W^{2,2}_{\alpha}(\Omega^1(M \setminus S_p)) \to L^2_{\alpha-2}(\Omega^1(M \setminus S_p))$ is Fredholm, for $\alpha$ outside of a discrete subset of $\mathbb{R}$, denoted by $D(\Delta)$.
\end{lemma}

This lemma follows from Corollary 3.2.13, Lemma 3.3.4, and Lemma 3.3.6 in \cite{foscolo2013moduli}.

\begin{lemma}
Let $\Delta: W^{2,2}_{\alpha}(\Omega^1(M \setminus S_p)) \to L^2_{\alpha-2}(\Omega^1(M \setminus S_p))$, where $M$ is a rational homology 3-sphere. Let $\alpha \notin D(\Delta)$. Then 
\begin{align*}
    \ker \Delta = \{0\}.
\end{align*}
\end{lemma}

The proof is similar to Lemma 4.4 in \cite{MR2004129} and Lemma 3.4.4 in \cite{foscolo2013moduli}.

\begin{proof}
To show $\Delta$ is Fredholm, one can construct an inverse for this operator, by taking the inverse of the Laplacian away from the singular points $M \setminus \cup_i B_{ \epsilon_i}(p_i)$ --- which can be done since $H^1(M \setminus \cup_i B_{ \epsilon_i}(p_i), \mathbb{R}) = 0$ --- and glue it to the inverse of the Laplacian on the weighted Sobolev spaces defined on the neighbourhood of the singular points. These local inverses close to the singular points can be constructed by solving the Dirichlet problem, similar to Proposition 3.3.11 in \cite{foscolo2013moduli}.
\end{proof}

\subsection[Solving the Longitudinal Part over $M \setminus (\cup_j B_{2 \epsilon_j}(q_j) \cup S_p)$]{Solving the Longitudinal Part over $\pmb{M \setminus (\cup_j B_{2 \epsilon_j}(q_j) \cup S_p)}$}

In this section, we solve the linear equation over $M \setminus (\cup_j B_{2 \epsilon_j}(q_j) \cup S_p )$. In Theorem \ref{longcomp}, we solve the longitudinal component of the linear problem, and in Theorem \ref{trancomp}, the transverse one. 

\begin{theorem}\label{longcomp} Let $-\frac{1}{2} \leq \alpha < 0$ such that $\alpha, -\alpha-1$ are not in $D(\Delta)$. Let $f^L \in L^2_{\alpha-2}(M \setminus S_p)$. Then there exists a solution to the equation $\Delta u^L = f^L$, where $u^L \in W^{2,2}_{ \alpha}(\Omega^1(M \setminus S_p))$, and
\begin{align*}
    \|u^L\|_{W^{2,2}_{\alpha}(M \setminus S_p)} \leq C \|f^L\|_{L^2_{\alpha-2}(M \setminus S_p)},    
\end{align*} 
for a constant $C$, independent of $\lambda$.
\end{theorem}

\begin{proof}
Pick $\alpha \in [-\frac{1}{2},0) \setminus D(\Delta)$. Therefore, $\Delta: W^{2,2}_{\alpha}(\Omega^1(M \setminus S_p)) \to L^2_{\alpha-2}(\Omega^1(M \setminus S_p))$ is Fredholm. We have
\begin{align*}
    L^2_{\alpha-2} = 
    \Delta(W^{2,2}_{\alpha}) \oplus 
    \ker \Delta_{W^{2,2}_{-\alpha-1}}.
\end{align*}
On the other hand,
\begin{align*}
    \ker \Delta_{W^{2,2}_{-\alpha-1}}
    \cong H^1(M \setminus S_p, \mathbb{R}) = 0,
\end{align*}
and therefore, $\Delta: W^{2,2}_{\alpha}(\Omega^1(M \setminus S_p)) \to L^2_{\alpha-2}(\Omega^1(M \setminus S_p))$ is surjective. Furthermore, since it is a Fredholm operator, there is a constant $C$ such that 
\begin{align*}
    \|u^L\|_{W^{2,2}_{\alpha}}
    \leq C \|\Delta u^L\|_{L^2_{\alpha-2}}.
\end{align*}
\end{proof}

\subsection{The Transverse Component}\label{lineartrans}

In this section, we solve the transverse part of the linear equation, away from the points where the scaled BPS-monopoles are located. Recall that on the transverse part the weight does not appear. For all smooth compactly supported $L$-valued differential forms $ u^T \in \Omega^{\boldsymbol{\cdot}}(U, \mathfrak{su}(2))$, let
\begin{align*}
    &\|u^T\|^2_{L^2_{\alpha}(U)} = 
    \|u^T\|^2_{L^2(U)},\\
    &\|u^T\|^2_{W^{1,2}_{\alpha}(U)} =
    \|u^T\|^2_{W^{1,2}(U)} :=
    \|u^T\|^2_{L^2(U)} + 
    \|\nabla_{A_0} u^T\|^2_{L^2(U)} 
    + \frac{1}{2} \| [\varPhi_0,u] \|_{L^2(U)}^2,\\
    &\|u^T\|^2_{W^{2,2}_{\alpha}(U)} =
    \|u^T\|^2_{W^{2,2}(U)} :=
    \|u^T\|^2_{W^{1,2}(U)} + 
    \|\nabla_{A_0}(d_2^* u^T)\|^2_{L^2(U)}
    + \frac{1}{2} \|[\varPhi_0, d_2^*u]\|^2_{L^2(U)}.
\end{align*}
Moreover, let $W_0^{2,2}$ be the subset of $W^{2,2}$ which vanish on $\partial U = \cup_j (\partial B_{2 \epsilon_j}(q_j))$.

The main theorem of this section is the following. The essential assumption is that the monopole has a very large mass.

\begin{theorem}\label{trancomp}
Let $U = M \setminus ( \cup_j B_{2\epsilon_j}(q_j) \cup S_p )$. Let $d_2d_2^*: W_0^{2,2}(\Omega^1(U,L)) \to L^2(\Omega^1(U,L))$. Let $f^T \in L^2(\Omega^1(U; L))$ such that $f^T = 0$ on $\partial U$. For a sufficiently large average mass $\overline{m}$, there exists a unique solution $ u^T \in W_0^{2,2}(\Omega^1(U,L))$ to $d_2 d_2^* u^T = f^T$ such that
\begin{align*}
    \| u^T \|_{W^{2,2}(U)}
    \leq C 
    \| f^T \|_{L^2(U)},
\end{align*}
where the constant $C$ is independent of $\lambda$, which appears in the definition of the approximate monopole $(A_0, \varPhi_0)$. 
\end{theorem}

The proof of this theorem is presented in a series of lemmas. We only to prove the lemma for the case where $f^T$ is a smooth compactly supported element in $\Omega^1_c(U,L)$.

\begin{lemma} [Step 1] \label{s1t}
Suppose $f^T$ is a smooth compactly supported $L$-valued 1-form on $U$. Let 
\begin{align*} 
    &E: W^{1,2}_0(\Omega^1(U,L)) \to \mathbb{R},\\     &E(u^T) := \frac{1}{2}\int_{U} |d_2^*u^T|^2 vol_g - \langle u^T , f^T \rangle_{L^2(U)},
\end{align*} 
Then $E(u^T)$ is convergent for all $u^T \in W^{1,2}_0(\Omega^1(U,L))$.
\end{lemma}

\begin{proof}
First we show $\|d_2^*u^T\|^2_{L^2(U)}$ is finite. Again, the key fact in proving this is the monopole Weitzenb\"ock formula. Note that on this region $e_0 = 0$. 
\begin{align*} 
    \| d_2^* u^T\|_{L^2(U)}^2 = \| \nabla_{A_0} u^T\|_{L^2(U)}^2 &+ \| [\varPhi_0, u^T]\|_{L^2(U)}^2 + \langle u^T, Ric(u^T)\rangle_{L^2(U)}\\
    &\leq
    (1 +
    \sup_x |Ric(x)|)\|u^T\|_{W^{1,2}(U)}^2 < \infty.
\end{align*}

By the Cauchy–Schwarz inequality we have
\begin{align*}
    \langle u^T,f^T \rangle_{L^2(U)} \leq \|u^T\|_{L^2(U)} \|f^T\|_{L^2(U)} \leq \|u^T\|_{W^{1,2}(U)}
    \|f^T\|_{L^2(U)} < \infty,
\end{align*}
and therefore, the action functional $E(u^T)$ is convergent for all $u^T \in \Omega^1(U,L)$. 
\end{proof}

\begin{lemma}[Step 2]
The functional $E: W^{1,2}(\Omega^1(U,L)) \to \mathbb{R}$ is continuous. 
\end{lemma}

\begin{proof}
We should show the following functions are continuous, 
\begin{align*}
\|d_2^*-\|^2_{L^2(U)}:
W^{1,2}_0(\Omega^1(U,L)) \to \mathbb{R}, 
\quad \quad
\langle -,f^T \rangle_{L^2(U)}:
W^{1,2}(\Omega^1(U,L)) \to \mathbb{R}. 
\end{align*}
To prove $\|d_2^*-\|^2_{L^2(U)}$ is continuous we should show 
\begin{align*} 
\|d_2^* u_i^T\|^2_{L^2(U)} \to \|d_2^*u^T\|^2_{L^2(U)}, \quad \text{ when } \quad u_i^T \to u^T \quad \text{ in } \quad W^{1,2}(\Omega^1(U,L)).
\end{align*} 
By the Weitzenb\"ock formula we know
\begin{align*}
    \| d_2^* (u^T - u_i^T) \|_{L^2(U)}^2 &= \| \nabla_{A_0} (u^T - u_i^T) \|_{L^2(U)}^2 + \| [\varPhi_0, (u^T - u_i^T)]\|_{L^2(U)}^2 \\&+\langle Ric((u^T - u_i^T)),(u^T - u_i^T) \rangle_{L^2(U)} \\& \leq
    \|(u^T - u_i^T)\|_{W^{1,2}(U)}^2 
    + \| [\varPhi_0, (u^T - u_i^T)]\|_{L^2(U)}^2 \\&+
    \sup_{x}|Ric(x)|\|u^T-u_i^T\|_{L^2(U)},
\end{align*}
which goes to zero as $i \to \infty$.

In order to prove $\langle -,f^T \rangle_{L^2(U)}:
W^{1,2}_0(\Omega^1(U,L)) \to \mathbb{R}$ is continuous, note that this map is linear, and since $W^{1,2}(\Omega^1(U,L))$ is a Hilbert space, continuity is equivalent to being bounded. 
\begin{align*}
    \langle u^T , f^T \rangle_{L^2(U)} \leq
        \|u^T\|_{L^2(U)} \|f^T\|_{L^2(U)} \leq
    \|u^T\|_{W^{1,2}(U)} \|f^T\|_{L^2(U)},
\end{align*}
and therefore, the operator norm 
\begin{align*}
    \sup \{\langle u^T , f^T \rangle_{L^2(U)} \; | \; {u^T \in W^{1,2}, \|u^T\|_{W^{1,2}}  = 1} \} \leq 
     \|f\|_{L^2(U)} < \infty,
\end{align*}
which proves the linear map is bounded and continuous. 
\end{proof}

\begin{lemma}[Step 3]
$E: W^{1,2}_0(\Omega^1(U,L)) \to \mathbb{R}$ is Gateaux-differentiable.
\end{lemma}

\begin{proof}
A direct computation shows
\begin{align*}
    d_u^T E(v^T) = \int_{U} \langle d_2^*u^T , d_2^*v^T \rangle_{L^2} vol_g - \langle v^T , f^T \rangle_{L^2(U)}.
\end{align*}
Note that for any $u^T ,v^T \in W_0^{1,2}$,
\begin{align*}
    \int_{U} \langle d_2^*u^T , d_2^*v^T \rangle_{L^2} vol_g \leq \|d_2^*u^T\|_{L^2(U)} \| d_2^*v^T\|_{L^2(U)} < \infty, 
\end{align*}
and 
\begin{align*}
    \langle v^T , f^T \rangle_{L^2(U)} \leq \|v^T\|_{L^2(U)} \| f^T\|_{L^2(U)} < \infty.
\end{align*}
\end{proof}

\begin{lemma}[Step 4]\label{itisnormt} The functional
\begin{align*}
    E: W^{1,2}_0(\Omega^1(U,L)) \to \mathbb{R}, \quad \quad     E(u^T) := \frac{1}{2}\int_{U} |d_2^*u^T|^2 vol_g - \langle u^T , f^T \rangle_{L^2(U)}.
\end{align*} 
is strictly convex, when $\overline{m}$ is sufficiently large.
\end{lemma}

\begin{proof}
$\langle u^T,f^T \rangle_{L^2(U)}$ is linear in $u$ and we only need to show $\|d_2^*u^T\|_{L^2(U)}^2$ is strictly convex. This reduces to showing that $\|d_2^*(u_1^T-u_2^T)\|_{L^2(U)}^2 > 0$ when $u_1^T, u_2^T \in W_0^{1,2}(\Omega^1(U,L))$ and $u_1^T \neq u_2^T$. Let $u^T = u_1^T - u_2^T$. We should show 
\begin{align*} 
    u^T \in W_0^{1,2}(\Omega^1(U,L)),\quad d_2^*u^T = 0\quad \Rightarrow \quad u^T = 0. 
\end{align*}
In fact, this shows $\| d_2^* - \|_{L^2}$ is a norm on $W_0^{1,2}(\Omega^1(U,L))$.
\begin{align*}
0 = \| d_2^* u \|_{L^2(\mathfrak{R}^3)}^2 & \geq \| \nabla_{A_0} u^T\|_{L^2(U)}^2  + (\frac{\overline{m}^2}{4} - \sup_x |Ric(x)|)\|u^T \|_{L^2(U)}^2  \\&\geq 
\| \nabla_{A_0} u^T\|_{L^2(U)}^2 +
\|u^T \|_{L^2(U)}^2,
\end{align*} 
when $\overline{m}$ is large enough such that
\begin{align*}
    \frac{\overline{m}^2}{4} - \sup_x |Ric(x)| \geq 1,
\end{align*}
and therefore, $u^T=0$. 
\end{proof}

\begin{lemma}[Step 5]\label{minsolt}
$E: W^{1,2}_0(\Omega^1(U,L)) \to \mathbb{R}$ has a unique minimizer, when $\overline{m}$ is sufficiently large, and therefore, $d_2 d_2^* u^T = f^T$ has a unique solution. 
\end{lemma}

\begin{proof}
In order to prove this, we use the following fact. 

Let $E:W \to \mathbb{R}$ be a convex, continuous, real-Gateaux-differentiable functional defined on a real reflexive Banach space $W$ such that $d_uE(u) > 0$ for any $u \in W$ with $\|u\|_{W} \geq R > 0$. Then there exists an interior point $u_0$ of $\{u \in W \; | \; \|u\|_{W} < R \}$ which is the unique minimizer of $E$ and $d_{u_0} E = 0$.

Let $W = W_0^{1,2}(\Omega^1(U,L))$ equipped with the norm $L^2$ defined by $\|-\|_{L^2(U)}$. Then
\begin{align*}
    \| d_2^* u^T \|_{L^2(U)}^2 &= \| \nabla_{A_0} u^T \|_{L^2(U)}^2 + \| [\varPhi_0, u^T]\|_{L^2(U)}^2 + \langle Ric(u^T),u^T \rangle_{L^2(U)} \\ &\geq
    \| \nabla_{A_0} u^T \|_{L^2(U)}^2  + (\frac{\overline{m}^2}{4} - \sup_x |Ric(x)|)\|u^T\|^2_{L^2(U)}
    \\ &\geq  \| \nabla_{A_0} u^T \|_{L^2(U)}^2  + \|u^T\|^2_{L^2(U)},
\end{align*}
when 
\begin{align*}
    \frac{\overline{m}^2}{4} - \sup_x |Ric(x)| \geq 1.
\end{align*}
Let $R = 1 + \|f^T\|_{L^2(U)}$. Then $\|u^T\|_{L^2(U)} > 1 + \|f^T\|_{L^2(U)}$. Therefore,
\begin{align*}
    d_u^T E(u^T) &= 
    \| d_2^*u^T \|^2_{L^2(U)} - \langle u^T , f^T \rangle_{L^2(U)}
    \\&\geq \| \nabla_{A_0} u^T \|_{L^2(U)}^2  + \|u^T\|^2_{L^2(U)} -
    \|u^T\|_{L^2(U)} \| f^T\|_{L^2(U)}\\&
    =
    \|\nabla_{A_0} u^T \|_{L^2(U)}^2 +
    \|u^T\|_{L^2(U)} (\|u^T\|_{L^2(U)} -
    \| f^T\|_{L^2(U)} )
    \\&\geq
    \|\nabla_{A_0} u^T \|_{L^2(U)}^2 +
    \|u^T\|_{L^2(U)} > R > 0,
\end{align*}
when 
\begin{align*}
\|u^T\|_{L^2(U)} > R= 1 + \|f^T\|_{L^2(U)} \quad \text{ and } \quad
\overline{m} > \sqrt{1+ 4\sup_x |Ric(x)|}.
\end{align*}
Therefore, $E$ has a unique minimizer $u^T$ where $\|u^T\|_{L^2(U)} < 1 + \|f^T\|_{L^2(U)}$.
\end{proof}

\begin{lemma}[Step 6] \label{Step6t}
For $\overline{m}$ sufficiently large, the unique solution $u^T$ of the previous lemma is in $W_0^{1,2}(U,L)$, and satisfies 
\begin{align*}
        \|u^T\|_{W^{1,2}(U)} \leq C \|f\|_{L^2(U)},
\end{align*}
for a uniform constant $C$.
\end{lemma}

\begin{proof}
We have
\begin{align*}
    \|f^T\|_{L^2}^2
    = \|d_2d_2^*u^T\|_{L^2}^2
    \geq \|\nabla_A u^T\|_{L^2}^2
    + \frac{1}{2}\| [\varPhi_0,u] \|_{L^2_{\alpha-1}(\mathfrak{R}^3)}^2
    +(\frac{\overline{m}^2}{8}
    -\sup |Ric_{x}|)
    \|u^T\|_{L^2}^2,
\end{align*}
assuming $\overline{m}$ is sufficiently large, we get 
\begin{align*}
    \|f^T\|_{L^2}^2
    \geq C \|u^T\|_{W^{1,2}}^2,
\end{align*}
for a positive constant $C$.
\end{proof}

\begin{lemma}[Step 7]
Suppose $\overline{m}$ is sufficiently large. Then the unique solution $u^T$ of Lemma \ref{minsolt} is in $W_0^{2,2}(U,L)$, and satisfies 
\begin{align*}
 \|u^T\|_{W^{2,2}(U)} \leq C \|f^T\|_{L^2(U)},
\end{align*}
for a positive constant $C$.
\end{lemma}

\begin{proof}
We should find a uniform bound on 
$\|\nabla_{A_0} d_2^*u^T\|_{L^2}$ in terms of $\|f^T\|_{L^2}$. Let $(a^T,\varphi^T) = D^*(u^T,0) = d^*_2 u^T$ in the Weitzenb\"ock formula for $DD^*$. By multiplying this formula by $d_2^*u^T$, and integrating over $U$, we get
\begin{align*}
 \|\nabla_A (d_2^*u^T)\|_{L^2}^2
 \leq
    \|\nabla_A (d_2^*u^T)\|_{L^2}^2 &+ \|[\varPhi,(d_2^*u^T)]\|_{L^2}^2 \\ &\leq
    \|D (d_2^* u^T)\|_{L^2}^2
    + \sup_{x} |Ric(x)| \|d_2^*u^T\|_{L^2}^2.
\end{align*}
We start by bounding $\|D (d_2^* u^T)\|_{L^2}^2$. Recall that
\begin{align*}
    D (d_2^*u^T) = DD^*(u^T,0)=(f^T,*[e_0 \wedge *u^T]) = (f^T,0).
\end{align*}
By taking the $L^2$-norm over $U$, we get
\begin{align*} 
    \|D(d_2^*u^T)\|_{L^2}^2 &= 
    \|f^T\|_{L^2}^2,
\end{align*}
The term $\|d_2^*u^T\|_{L^2}^2$ can be bounded uniformly by $\|f^T\|_{L^2}^2$, and therefore,
\begin{align*}
    \|\nabla_{A_0} (d^*_2u^T)\|_{L^2(U)}^2  &\leq 
   \|D (d_2^* u^T)\|_{L^2}^2
    + \sup_{x} |Ric(x)| \|d_2^*u^T\|_{L^2}^2 \leq (1+ \sup_{x} |Ric(x)|)\|f^T\|_{L^2}^2.
\end{align*}
hence,
\begin{align*}
        \|u^T\|_{W^{2,2}(U)} \leq C \|f^T\|_{L^2(U)},
\end{align*}
for a uniform constant $C$.
\end{proof}

\subsection[The Function Spaces over $M \setminus S_p$]{The Function Spaces over $\pmb{M \setminus S_p}$}

We start by setting up the suitable function spaces over $M \setminus S_p$ by gluing the weighted function spaces defined over $\mathfrak{R}^3$ and the function spaces defined for the longitudinal component and transverse component over $U$. Then we use these spaces to solve the linearized problem on $M \setminus S_p$. 

Let $\delta_{q_j}$ and $\delta_{p_i}$ be the injectivity radius at the points $q_j \in S_q$ and $p_i \in S_p$, respectively. For each $q_j \in S_q$, let 
\begin{align} \label{wq}
    w_j(x) = 
    \begin{cases}
        \sqrt{\lambda_j^{-2} + r_j^2}, \quad &r_j\leq \delta_{q_j}\\
        1, \quad &r_j \geq 1,
    \end{cases}
\end{align}
and for each point $p_i \in S_p$, let 
\begin{align}
    w_i(x) = 
    \begin{cases}
        r_i, \quad &r_i\leq \delta_{p_i}\\
        1, \quad &r_i \geq 1,
    \end{cases}
\end{align}
where $r_j$ and $r_i$ are the geodesic distance from the points $q_i$ and $p_i$, respectively.

\begin{definition}\label{weighted}
Let $U_{ext} = M \setminus (\cup_i B_{2\epsilon_i}(p_i)\cup_j B_{2\epsilon_j}(q_j))$. For any smooth compactly supported differential form $u \in \Omega_c(M \setminus S, \mathfrak{su}(2))$, let
\begin{align*}
    \|u\|^2_{L^2_{\alpha_1,\alpha_2}(M \setminus S_p)} = 
    \|u\|^2_{L^2(U_{ext})}
    &+ \sum_{j=1}^k
    \|w_j^{-\alpha_1-\frac{3}{2}}u\|^2_{L^2(B_{2 \epsilon_j}(q_j))}
    \\&+
    \sum_{i=1}^n
    \|u^T\|^2_{L^2(B_{2 \epsilon_i}(p_i))}
    +
    \sum_{i=1}^n
    \|w_i^{-\alpha_2-\frac{3}{2}} u^L\|^2_{L^2(B_{2 \epsilon_i}(p_i))}.
\end{align*}
Furthermore, 
\begin{align*}
    \|u\|^2_{W^{1,2}_{\alpha_1, \alpha_2}(M \setminus S_p )} = 
    \|u\|^2_{L^2_{\alpha_1, \alpha_2}(M \setminus S_p)} + 
    \|\nabla_{A_0} u\|^2_{L^2_{\alpha_1 -1, \alpha_2 -1}(M \setminus S_p)} &+
    \sum_{j=1}^k
    \|[\varPhi_0, u]\|^2_{L^2_{\alpha_1-1}(B_{2 \epsilon_j}(q_j))}
    \\&+ \frac{1}{2}
    \|[\varPhi_0, u]\|^2_{L^2(U)}.
\end{align*}
Moreover,
\begin{align*}
    \|u\|^2_{W^{2,2}_{\alpha_1, \alpha_2}(M \setminus S_p)} &= 
    \|u\|^2_{W^{1,2}_{\alpha_1, \alpha_2}(M \setminus S)} + 
    \|\nabla_{A_0}(d_2^* u)\|^2_{L^2_{\alpha_1-2, \alpha_2-2}(M \setminus S_p)}
    \\&+
    \sum_{j=1}^k
    \|[\varPhi_0,d_2^* u]\|^2_{L^2_{\alpha_1-2}(B_{2 \epsilon_j}(q_j))} 
    +  \frac{1}{2}
    \|[\varPhi_0, d_2^*u]\|^2_{L^2(U)}.
\end{align*}
The spaces $W^{k,2}_{\alpha_1,\alpha_2}$ are defined as the completion of $C^{\infty}_0$ with respect to the corresponding norms, for $k \in \{0, 1, 2\}$. Furthermore, one can define similar norms and weighted Sobolev spaces $W^{k,p}_{\alpha_1, \alpha_2}$ for any $p \geq 2$.
\end{definition}

\subsection[Solving the Linear Equation over $M \setminus S_p$]{Solving the Linear Equation over $\pmb{M \setminus S_p}$ and the Fixed Point Theorem}

In this section, we solve the linear equation for an $\mathfrak{su}(2)$-valued 1-form $f$ on $M \setminus S_p$. This can be done by patching Theorems \ref{q}, \ref{longcomp} and \ref{trancomp}. The idea is writing $ f = \chi_{0}f + \sum_{j=1}^k \chi_j f$, where $\chi_j$ is a cut-off function which is equal to 1 on $B_{ 2\epsilon_j}(q_j)$, $\chi_j = 0$ on $M \setminus B_{3 \epsilon_j}(q_j)$, and $\chi_0 + \chi_j = 1$ for all $j  \in \{1, \hdots, k\}$. Each equation $d_2d_2^* u_j = \chi_{j}f$ can be solved using Theorem \ref{q}, and the equation $d_2d_2^* u_0 = \chi_{0}f$, using Theorems \ref{longcomp} and \ref{trancomp}.

Then we would glue these solutions to get an approximate solution for the linear equation on $M \setminus S_p$. Using an iteration argument one can see the linear equation has a solution. However, in order to be able to use an iteration argument, we will need more subtle cut-off functions. This is the 3-dimensional version of Lemma 7.2.10 in \cite{MR1079726}, which is about cut-off functions on 4-dimensional manifolds.

\begin{lemma}\label{cutoff}
There is a constant $K$ and for any $N$ and $\lambda$ there is a smooth function $\beta = \beta_{N,\lambda}$ on $\mathbb{R}^3$ with $\beta(x) = 1$ where $|x| \leq N^{-1} \lambda^{\frac{1}{2}}$ and $\beta(x) = 0$ where $|x| \geq N \lambda^{\frac{1}{2}}$ such that 
\begin{align*}
    \|\nabla \beta\|_{L^3} \leq \frac{K}{(\log(N))^{\frac{2}{3}}}.
\end{align*}
\end{lemma}

\begin{proof}
This lemma is clearer in cylindrical coordinates. We can transfer the problem to a cylindrical space since the $L^p$-norm on $\Omega^1(M)$ is conformally invariant if and only if $p = \dim(M)$. Therefore, $L^3$ is conformally invariant on 1-forms on $\mathbb{R}^3$. Let $L: \mathbb{R}^3 \setminus \{ 0 \} \cong (-\infty, + \infty) \times S^2$, be the identification defined by
\begin{align*}
    L(r, \theta, \varphi) = (-\log(r), \theta, \varphi),
\end{align*}
where $(r, \theta, \varphi)$ denotes the spherical coordinates on $\mathbb{R}^3$.

Let $(t, \theta, \varphi)$ denotes the cylindrical coordinates on $(-\infty, + \infty) \times S^2$. We are looking for a cut-off function 
\begin{align*}
    \widetilde{\beta}: \mathbb{R} \times S^2
    \to \mathbb{R},
\end{align*}
such that
\begin{align*}
    &\widetilde{\beta} (t, \theta, \varphi) = 0, \quad \text{ when } \quad t < -\frac{1}{2}\log(\lambda) - \log(N), \\
    &\widetilde{\beta} (t, \theta, \varphi) = 1, \quad \text{ when } \quad t > -\frac{1}{2}\log(\lambda) + \log(N), \\
    &\|\nabla \widetilde{\beta}\|_{L^3} \leq \frac{K}{\log (N)}.
\end{align*}
Moreover, we ask $\widetilde{\beta}$ to be only a function of $t$,
\begin{align*}
    \widetilde{\beta} (t, \theta_1, \varphi_1) = 
    \widetilde{\beta} (t, \theta_2, \varphi_2),
\end{align*}
for all $(\theta_1, \varphi_1), (\theta_2, \varphi_2) \in S^2$.

One can take $\widetilde{\beta}$ to be a smooth function where
\begin{align*}
    &\widetilde{\beta} (t, \theta, \varphi) = 0, \quad \text{ when } \quad t < -\frac{1}{2}\log(\lambda) - \log(N), \\
    &\widetilde{\beta} (t, \theta, \varphi) = 1, \quad \text{ when } \quad t > -\frac{1}{2}\log(\lambda) + \log(N), \\
    & \partial_t \widetilde{\beta} \approx - \frac{1}{2 \log(N)},
    \quad \text{ when } \quad  -\frac{1}{2}\log(\lambda) - \log(N) < t < -\frac{1}{2}\log(\lambda) + \log(N).
\end{align*}
Then, we have 
\begin{align*}
    \|\nabla \widetilde{\beta}\|_{L^3} = 
    \|\partial_t \widetilde{\beta}\|_{L^3} \leq \frac{K}{(\log(N))^{\frac{2}{3}}},
\end{align*}
for a constant $K$.
\end{proof}

\begin{theorem} \label{cut} 
Let 
\begin{align*}
d_2: W^{1,2}_{\alpha_1-1, \alpha_2-1}(\Omega^1 (M \setminus S_p, \mathfrak{g}_P) \oplus \Omega^0 (M \setminus S_p, \mathfrak{g}_P)) \to L^2_{\alpha_1-2, \alpha_2-2}(\Omega^1 (M \setminus S_p, \mathfrak{g}_P)).    
\end{align*} 
For any $\alpha_1 \in [-\frac{1}{2},0)$ and $\alpha_2$ outside of a discrete subset, and for sufficiently large $\overline{m}$, the linear equation $d_2 \xi = f$ has a unique solution $\xi \in W^{1,2}_{\alpha_1-1, \alpha_2-1}$, for each $f \in L^2_{\alpha_1-2, \alpha_2-2}$. Moreover,
\begin{align*}
    \|\xi\|_{W^{1,2}_{\alpha_1-1, \alpha_2-1}(M \setminus S_p)} \leq \|f\|_{L^2_{\alpha_1-2, \alpha_2-2}(M \setminus S_p)}.
\end{align*}
\end{theorem}

\begin{proof}
Let $\chi_j$ be the cut-off function, centered at $q_j$,
\begin{align*}
    \chi_j = 
    \begin{cases}
        1 &\text{ if } \quad dist(x,q_j) \leq 2\epsilon_j, \\
        0 &\text{ if } \quad dist(x,q_j) \geq 3\epsilon_j,
    \end{cases}
\end{align*}$\chi_j$ 
for $j \in \{ 1, \hdots, k \}$ such that $|\nabla \chi_j|\leq \frac{2}{\epsilon_j}$. 

Let $\chi_0$ be the cut-off function, centered away from the points $\{q_1, \hdots, q_k \}$,
\begin{align*}
    \chi_0 = 
    \begin{cases}
        1 &\text{ if } \quad dist(x,q_j) \geq 3\epsilon_j \quad \text{ for all } j \in \{1, \hdots, k\} \\
        0 &\text{ if } \quad dist(x,q_j) \leq 2\epsilon_j \quad \text{ for some } j \in \{1, \hdots, k\},
    \end{cases}
\end{align*}
such that $\chi_0 + \chi_j = 1$, and $|\nabla \chi_0|\leq \max_j \{ 2/\epsilon_j \}$. 

Let $f_j = \chi_j f$. For each $j \in \{1, \hdots, k\}$, $f_j$ is supported close to a point $q_j$, and as explained before, we can localize the problem on $B_{3 \epsilon_j}(q_j)$, and transfer it to $\mathfrak{R}^3$. Note that
\begin{align*}
\|f_j\|_{L^2_{\alpha_1-2}} \leq \|f\|_{L^2_{\alpha_1-2, \alpha_2-2}} < \infty,    
\end{align*}
and therefore, using theorem \ref{q}, there are $u_j \in W^{2,2}_{\alpha}$ such that on $B_{3 \epsilon_j}(q_j)$ we have $d_2 d_2^* u_j = f_j$, where 
\begin{align*}
\|u_j\|_{W^{2,2}_{\alpha_1}(\mathfrak{R}^3)} \leq C\|f_j\|_{L^2_{\alpha_1-2, \alpha_2-2}(\mathfrak{R}^3)} \leq C\|f\|_{L^2_{\alpha_1-2, \alpha_2-2}(M \setminus S_p)}.  
\end{align*}

Let $f_0 = \chi_0 f$. The section $f_0$ is supported on $M \setminus( \cup_{j=1} B_{2\epsilon_j}(q_j) \cup S_p)$, and is in $L^2_{\alpha_2-2}$. Moreover, it vanishes on $\partial (\cup_j B_{2\epsilon_j}(q_j))$. Using theorems \ref{longcomp} and \ref{trancomp}, there is $u_0 \in W^{2,2}_{\alpha_2}(U)$ such that 
\begin{align*}
    d_2 d_2^* u_0 = f_0, \quad \quad 
    \|u_0\|_{W^{2,2}_{\alpha_2}(U)} \leq C \|f_0\|_{L^2_{\alpha_2-2}(U)}.
\end{align*}

Let $\beta_j$ be the cut-off function introduced in Lemma \ref{cutoff}, centered at $q_j$, 
\begin{align*}
    &\beta_j(x) = 1, \quad \text{ if } \quad |x| < N^{-1} \lambda^{\frac{1}{2}}, \\
    &\beta_j(x) = 0, \quad \text{ if } \quad |x| > N\lambda^{\frac{1}{2}}, \\
    & \| \nabla \beta_j \|_{L^3} \leq \frac{K}{(\log(N))^{\frac{2}{3}}},
\end{align*}
for all $j \in \{ 1, \hdots, k \}$, for a constant $K>0$, and any $N > 0$. 

Let $\beta_0$ the cut-off function supported away from the points $q_j$ such that $\beta_0 + \beta_j = 1$. Using these cut-off functions we can transfer the solutions back to $M$ and glue them together. Let
\begin{align*}
    &F: L^2_{\alpha_1-2, \alpha_2-2}(\Omega^1 (M \setminus S_p, \mathfrak{g}_P)) \to
    W^{1,2}_{\alpha_1-1, \alpha_2-1}(\Omega^1 (M \setminus S_p, \mathfrak{g}_P) \oplus \Omega^0 (M \setminus S_p, \mathfrak{g}_P)),
\end{align*}
be the map defined by 
\begin{align*}
    F(f) = \sum_{j=0}^k \beta_j d_2^* u_j.
\end{align*}
Note that $F(f) \in W^{1,2}_{\alpha_1 - 1, \alpha_2 -1 }$. In fact,
\begin{align*}
    \|F(f)\|_{L^2_{\alpha_1-1, \alpha_2-1}} &= 
    \|\sum_{j=0}^k \beta_j d_2^* u_j\|_{L^2_{\alpha_1-1, \alpha_2-1}}
    \leq \sum_{j=0}^k
    \|d_2^*u_j\|_{L^2_{\alpha_1-1, \alpha_2-1}}
    \leq  C\sum_{j=0}^k
    \|u_j\|_{W^{1,2}_{\alpha_1, \alpha_2}}
    \\ & \leq  C' \sum_{j=0}^k
    \|f_j\|_{L^{2}_{\alpha_1 -2, \alpha_2 -2}}
    \leq (k+1) C' \|f\|_{L^{2}_{\alpha_1 -2, \alpha_2 -2}}.
\end{align*}
Note that $\beta_j f =\beta_j \xi_j f = \beta_j f_j$, when $\epsilon_j > 0$ is sufficiently small. Moreover, 
\begin{align*}
    \|\nabla_{A_0} F(f) \|_{L^2_{\alpha_1 -2, \alpha_2 -2}} &= \|\sum_{j=0}^k \nabla_{A_0} (\beta_j d_2^* u_j)\|_{L^2_{\alpha_1-2, \alpha_2 -2}} \\ &\leq
    \sum_{j=0}^k(  \|\nabla \beta_j \cdot d_2^* u_j \|_{L^2_{\alpha_1-2, \alpha_2-2}} + \| \beta_j \nabla_{A_0} (d_2^*u_j)\|_{L^2_{\alpha_1 -2, \alpha_2 -2}})
    \\ &\leq 
    \sum_{j=0}^k( \|\nabla \beta_j\|_{L^3}\| d_2^* u_j\|_{L^6_{\alpha_1 -2, \alpha_2 -2}} +\|\nabla_{A_0} (d_2^* u_j)\|_{L^2_{\alpha_1 -2, \alpha_2 -2}}) 
    \\ &\leq
    \sum_{j=0}^k \left( \frac{K}{(\log (N))^{\frac{2}{3}}} +1 \right)\|d_2^* u_j\|_{W^{1,2}_{\alpha_1 -2, \alpha_2 -2}}
    \\ &\leq
    C \sum_{j=0}^k\left( \frac{K}{(\log (N))^{\frac{2}{3}}} +1 \right) \|f_j\|_{L^2_{\alpha_1 -2, \alpha_2 -2}}
    \\ &\leq
    C (k+1)\left(\frac{K}{(\log (N))^{\frac{2}{3}}} +1 \right) \|f\|_{L^2_{\alpha_1-2, \alpha_2-2}},
\end{align*}
where in the second inequality we have used 
\begin{align*}
    \|fg\|_{L^2_{\alpha_1 -2, \alpha_2 -2}} \leq 
    \|f\|_{L^3}
    \|g\|_{L^6_{\alpha_1 -2, \alpha_2 -2}},
\end{align*}
which follows from the H\"older's inequality. Also in the third inequality we have used 
\begin{align*}
    \|h\|_{L^6_{\alpha_1 -2, \alpha_2 -2}} \leq
    \|h\|_{W^{1,2}_{\alpha_1-2, \alpha_2-2}},
\end{align*}
which follows from the Sobolev inequality.

Moreover, we have 
\begin{align*}
    \|[\varPhi,F(f)]\|_{L^2_{\alpha_1 -2, \alpha_2 -2}}
    &\leq \sum_{j=0}^k \|[\varPhi, d_2^* u_j]\|_{L^2_{\alpha_1 -2, \alpha_2 -2}}
    \leq \sum_{j=0}^k \|d_2^* u_j\|_{W^{1,2}_{\alpha_1-1, \alpha_2-1}}
    \\ & \leq C \sum_{j=0}^k \|f_j\|_{L^2_{\alpha_1-2, \alpha_2-2}}
    \leq C (k+1) \|f\|_{L^2_{\alpha_1-2, \alpha_2-2}},
\end{align*}
and therefore, 
\begin{align*}
    \|F(f)\|_{W^{1,2}_{\alpha_1-1, \alpha_2-1}}
    \leq C' \|f\|_{L^2_{\alpha_1-2, \alpha_2-2}},
\end{align*}
for a positive constant $C'$.

The map $F$ is close to being a right-inverse of $d_2$.
\begin{align*}
    \|f - d_2 F(f)\|_{L^2_{\alpha_1-2, \alpha_2-2}} &= 
    \| \sum_{j=0}^k \nabla \beta_j \cdot d_2^* u_j \|_{L^2_{\alpha_1-2, \alpha_2-2}}  \leq \sum_{j=0}^k
    \|\nabla \beta_j \|_{L^3} \|d_2^* u_j\|_{W^{1,2}_{\alpha_1-1, \alpha_2-1}} \\ &\leq
    \frac{C}{(\log (N))^{\frac{2}{3}}} \|f\|_{L^2_{\alpha_1-2, \alpha_2-2}}.
\end{align*}
This shows, the map $(Id - d_2 \circ F)$ is a contraction on $L^2_{\alpha_1-2, \alpha_2-2}$, when $N$ is sufficiently large, and therefore, by the method of iteration, we get a map which is the right-inverse of $d_2$.
\end{proof}

The next step is to solve the full non-linear Bogomolny equation.

\section{The Quadratic Term and the Fixed Point Theorem}\label{quadraticterm}

In this section, we complete the construction of a family of solutions to the Bogomolny equation by the use of a fixed point theorem. In the previous section, we saw that there is a solution to the linearized equation. The remaining part of the equation is quadratic, which we will consider in this section. 

Let
\begin{align*}
Q: W^{1,2}_{\alpha_1-1, \alpha_2-1} \left(\Omega^1(M \setminus S_p,\mathfrak{g}_P) \oplus \Omega^0(M \setminus S_p,\mathfrak{g}_P) \right) \to L^2_{\alpha_1-2, \alpha_2-2}\left(\Omega^1(M \setminus S_p,\mathfrak{g}_P)\right),
\end{align*}
be the map defined by the  quadratic part of the Bogomolny equation, given by
\begin{align*}
   Q(a,\varphi) = *\frac{[a\wedge a]}{2}-[a,\varphi].
\end{align*}
The equation \ref{moneq} can be written as
\begin{align*} 
    (d_2+Q)(a,\varphi) = - e_0.
\end{align*}
Following Theorem \ref{cut}, let $d_2^{-1}$ be a right-inverse of $d_2$. 

Let $f = d_2(a, \varphi)$. Let $\xi$ be a solution to the equation $d_2 \xi = f$, and therefore, $ \xi = (a,\varphi)$. Let 
\begin{align*}
    q(f) := Q \circ d_2^{-1}(f),
\end{align*}
and therefore, we have
\begin{align*}
        f + q(f) = - e_0.
\end{align*}

The proof of the existence of the solution to the Bogomolny equation is based on the following lemma from \cite{MR1079726}.

\begin{lemma}[Donaldson-Kronheimer]\label{DonKro}
Let $B$ be a Banach space and $q:B \to B$ a smooth map such that for all $f, f' \in B$, 
\begin{align} \label{somek}
   \|q(f) - q(f')\| \leq K (\|f\| + \|f'\|) \|f-f'\|, 
\end{align}
for a constant K. Then, if $\|e\| \leq \frac{1}{10K}$ there is a unique solution $f$ to
the equation
\begin{align*}
    f+q(f)=e,
\end{align*}
where $\|f\| \leq 2\|e\|$.
\end{lemma}

Following this lemma, let the Banach space $B = L^2_{\alpha_1-2, \alpha_2-2}(\Omega^1 (M \setminus S_p, \mathfrak{g}_P))$, and let $e = -e_0$. We should show the assumptions of Lemma \ref{DonKro} holds in our case. 

\begin{lemma}
The error estimate $\|e_0\|_{L^2_{\alpha_1-2, \alpha_2-2}}$ can be made sufficiently small.
\end{lemma}

\begin{proof}
Note that $e_0 = 0$ on $M \setminus( \cup_j B_{2 \epsilon_j}(q_j) \cup S_p)$. On each $B_{2 \epsilon_j}(q_j)$, following Lemma \ref{errores},
\begin{align*}
    \|e_0\|_{L^2_{\alpha_1-2}(B_{2 \epsilon_j}(q_j))} =
    \int_{B_{2 \epsilon_j}(q_j)} w_j^{-2\alpha+1}|e_0|^2 vol_g \leq 
    C \int_{B_{2 \epsilon_j}(q_j)} w_j^{-2\alpha+1} vol_g \leq C \epsilon_j^4,
\end{align*}
for a constant $C>0$, and therefore, it can be made as small as necessary.
\end{proof}

To complete the proof, we should show $q(f) = Q \circ d_2^{-1}(f)$ satisfies \ref{somek} for some $K$, which is the content of the following lemma.

\begin{lemma} \label{minus}
There exists a constant $K$ such that 
\begin{align*}
    \| q(f)-q(f')\|_{L^2_{\alpha_1-2,\alpha_2-2}}
    \leq K
    \|f+f'\|_{L^2_{\alpha_1-2,\alpha_2-2}}\|f-f'\|_{L^2_{\alpha_1-2,\alpha_2-2}}.
\end{align*}
\end{lemma}

Let 
\begin{align*}
    \tilde{q}(f,f') := \frac{1}{2}(q(f+f')-q(f)-q(f')).
\end{align*}
The proof of Lemma \ref{minus} is based on the following lemma. 

\begin{lemma}\label{1}
The map 
\begin{align*}
    \tilde{q}: W^{1,2}_{\alpha_1-1,\alpha_2-1} (\Omega^1(M \setminus S_p)) \times W^{1,2}_{\alpha_1-1,\alpha_2-1} (\Omega^1(M \setminus S_p)) \to L^2_{\alpha_1-2,\alpha_2-2} (\Omega^1(M \setminus S_p)),
\end{align*}
is continuous. Moreover 
\begin{align}\label{mulcontinuous}
\| \tilde{q}(f,f')\|_{L^2_{\alpha_1-2,\alpha_2-2}} \leq C \|f\|_{W^{1,2}_{\alpha_1-1,\alpha_2-1}} \|f'\|_{W^{1,2}_{\alpha_1-1,\alpha_2-1}},
\end{align}
for a positive constant $C$.
\end{lemma}

\begin{proof}[Proof of Lemma \ref{minus} using \ref{1}]
We have
\begin{align*}
    \| q(f)-q(f')\|_{L^2_{\alpha_1-2,\alpha_2-2}} &= 
    \| \tilde{q}(f+f',f-f')\|_{L^2_{\alpha_1-2,\alpha_2-2}}
    \\ &\leq 
    C \|f+f'\|_{W^{1,2}_{\alpha_1-1,\alpha_2-1}} \|f-f'\|_{W^{1,2}_{\alpha_1-1,\alpha_2-1}}.
\end{align*}
\end{proof}

\begin{proof}[Proof of Lemma \ref{1}]
One only need to show
\begin{align*}
    \|q(f)\|^2_{L^2_{\alpha_1-2,\alpha_2-2}} \leq C \|f\|_{W^{1,2}_{\alpha_1-1,\alpha_2-1}}^2.
\end{align*}

Lemma \ref{1} can be localized over different regions of $M \setminus S_p$. Suppose $f$ is supported on $B_{2 \epsilon_j}(q_j)$ for some $j \in \{1, \hdots, k\}$. Similar to the study of the linearized equation on this region, we can transform the problem to $\mathfrak{R}^3$. We should show
\begin{align*}
    \|q(f)\|^2_{L^2_{\alpha_1-2}(\mathfrak{R}^3)} \leq C \|f\|_{W^{1,2}_{\alpha_1-1}(\mathfrak{R}^3)}^2.
\end{align*}
By the H\"older's inequality, when $\lambda_j$ is sufficiently large,
\begin{align*}
\|q(f)\|^2_{L^2_{\alpha_1-2}(\mathfrak{R}^3)} &= 
\int_{\mathfrak{R}^3} w_j^{-2\alpha_1+1}|q(f)|^2 vol_g
\leq C_1 \epsilon_j^{2\alpha_1}
\int_{\mathfrak{R}^3} w_j^{-4\alpha_1+1} |f|^4 vol_g
\\& \leq C_1  \epsilon_j^{2\alpha_1} \|w_j^{-\alpha_1-\frac{1}{2}} f\|_{L^2} \|w_j^{-\alpha_1+\frac{1}{2}}f\|_{L^6}^3,
\end{align*}
for a positive constant $C_1$. 

By the Sobolev inequality we have
\begin{align*}
    \|w_j^{-\alpha_1+\frac{1}{2}}f\|_{L^6} 
    \leq C_{Sob} 
    \|w_j^{-\alpha_1+\frac{1}{2}}f\|_{W^{1,2}} \leq C_2 \|f\|_{W^{1,2}_{\alpha_1-2}}
    \leq C_2 \|f\|_{W^{1,2}_{\alpha_1-1}},
\end{align*}
for a positive constant $C_2$, and therefore,
\begin{align*}
\|q(f)\|^2_{L^2_{\alpha_1-2}(\mathfrak{R}^3)} &\leq 
C_1 C_{2} \epsilon_j^{2\alpha_1} 
\|f\|_{W^{1,2}_{\alpha_1-1}}^4,
\end{align*}
hence on this region, \ref{mulcontinuous} holds with
\begin{align*}
    C = C_1 C_2 \epsilon_j^{2\alpha_1}.
\end{align*}

Second, suppose $f$ is supported on $U = M \setminus (\cup_j B_{2\epsilon_j}(q_j) \cup S_p)$. Let $f = f^L + f^T$.
\begin{align*}
    q(f) = \tilde{q}(f,f) = \tilde{q}(f^L+f^T,f^L+f^T) = 
    \tilde{q}(f^L,f^L)+
    \tilde{q}(f^L,f^T)+
    \tilde{q}(f^T,f^L)+
    \tilde{q}(f^T,f^T).
\end{align*}
We have 
\begin{align*}
    \tilde{q}(f^L,f^L) = q(f^L) = 
    Q \circ d_2^{-1}(f^L).
\end{align*}
The linear maps $d_2$ and $d_2^{-1}$, preserve the bundle decomposition induced by a Higgs fields $\varPhi_0$, and therefore, $(a_1^T,\varphi_1^T)=d_2^{-1}(f^L)$ is a section of longitudinal part. However, the Lie bracket vanishes when restricted to the longitudinal sub-bundle, 
\begin{align*}
    Q \circ d_2^{-1}(f^L) = 0,
\end{align*}
and therefore, 
\begin{align*}
q(f) =  (\tilde{q}(f^L,f^T)+ \tilde{q}(f^T,f^L)) + \tilde{q}(f^T,f^T),
\end{align*}
where $(\tilde{q}(f^L,f^T)+ \tilde{q}(f^T,f^L))$ is the transverse component and $\tilde{q}(f^T,f^T)$ is the longitudinal one.
    
For the transverse component \ref{mulcontinuous} becomes
\begin{align}
\| \tilde{q}(f^L,f^T)\|_{L^2(U)} \leq C \|f^L\|_{W^{1,2}_{\alpha_2-1}(U)} \|f^T\|_{W^{1,2}(U)}.
\end{align}
By the H\"older's inequality, we have
\begin{align*}
    \|\tilde{q}(f^L,f^T)\|_{L^2(U)} &
    \leq C 
    \|f^L\|_{L^3}
    \|f^T\|_{L^6}
    \leq C 
    \|w^{-\alpha_2-\frac{1}{2}} f^L\|_{L^3}
    \|f^T\|_{W^{1,2}} \\&\leq C
    \|w^{-\alpha_2-\frac{1}{2}} f^L\|_{W^{1,2}}
    \|f^T\|_{W^{1,2}}
    \leq C
    \|f^L\|_{W^{1,2}_{\alpha_2-1}}
    \|f^T\|_{W^{1,2}},
\end{align*}
where $C$ is a uniform constant.

For the longitudinal component \ref{mulcontinuous} becomes
\begin{align}
\| q(f^T)\|_{L^2_{\alpha_2-2}(U)} \leq C \|f^T\|_{W^{1,2}(U)}^2.
\end{align}

By the H\"older's and Sobolev inequalities we have
\begin{align*}
    \| q(f^T)\|_{L^2_{\alpha_2-2}(U)} &=
    \| w^{-\alpha_2+\frac{1}{2}}q(f^T)\|_{L^2(U)}
    \leq C
    \|w^{-\alpha_2+\frac{1}{2}}f^T\|_{L^3}
    \|f^T\|_{L^6} \\ &\leq 
    C
    \|f^T\|_{L^3}
    \|f^T\|_{L^6} 
    \leq C \|f^T\|_{W^{1,2}}^2,
\end{align*}
for a uniform constant $C$.
\end{proof}

This completes the gluing construction of irreducible $SU(2)$-monopoles with Dirac singularities on rational homology 3-spheres.

\printbibliography

\vspace{10pt}

\noindent
\author{Mathematical Sciences Research Institute, Berkeley, CA} \\ E-mail address: \href{ mailto:Saman.HabibiEsfahani@msri.org}{Saman.HabibiEsfahani@msri.org}

\end{document}